\documentclass[11pt]{amsart}
\usepackage{amsmath, amscd, amssymb, mathrsfs, mathabx, tikz-cd, amsthm, thmtools, chngcntr}
\usepackage[curve]{xypic}
\DeclareFontFamily{OT1}{rsfs}{}
\DeclareFontShape{OT1}{rsfs}{n}{it}{<-> rsfs10}{}
\DeclareMathAlphabet{\curly}{OT1}{rsfs}{n}{it}

\usepackage{hyperref}

\input{xy}
\xyoption{all}

\newtheorem{Thm}{Theorem}[section]
\newtheorem*{Thm*}{Theorem}

\newtheorem{Prop}[Thm]{Proposition}
\newtheorem{Def}[Thm]{Definition}
\newtheorem{Def/Thm}[Thm]{Definition/Theorem}

\newtheorem{Lemma}[Thm]{Lemma}

\theoremstyle{definition}
\newtheorem{Rmk}[Thm]{Remark}

\counterwithout{equation}{section}
\newcommand{\id}{{\operatorname{id}}}
\renewcommand\hom{\curly H\!om}

\newcommand{\rank }{{\mathrm{rank}\,}}
\newcommand{\cA}{{\mathcal{A}}}

\newcommand{\cO}{{\mathcal{O}}}
\newcommand{\cM}{{\mathcal{M}}}
\renewcommand{\cL}{{\mathcal{L}}}
\newcommand{\cE}{{\mathcal{E}}}

\newcommand{\cU}{{\mathcal{U}}}

\newcommand{\cX}{{\mathcal{X}}}

\newcommand{\PP }{{\mathbb P}}
\newcommand{\EE }{{\mathbb E}}

\newcommand{\CC }{{\mathbb C}}
\newcommand{\RR }{{\mathbb R}}

\newcommand{\ch}{\mathrm{ch}}

\let\wt\widetilde

\newcommand{\T}{\mathsf{T}}

\newcommand{\td}{\mathrm{td}}

\renewcommand{\AA}{\mathbb{A}}

\renewcommand\;{\hspace{.7pt}}
\newcommand\C{\mathbb C}
\newcommand\Q{\mathbb Q}
\newcommand\R{\mathbb R}
\newcommand\N{\mathbb N}
\newcommand\Z{\mathbb Z}

\newcommand{\LL}{\mathbb{L}}

\newcommand\loc{_{\mathrm{loc}}^{}}

\renewcommand\t{\mathfrak t}
\renewcommand\c{\mathfrak c}
\newcommand\e{\sqrt{\mathfrak e}\;}
\newcommand\m{\mathfrak m}

\newcommand\Ohat{\widehat\cO^{\;\vir}_{\!M}}

\renewcommand\({\big(}
\renewcommand\){\big)}
\makeatletter
\newcommand{\so}{\ \ext@arrow 0359\Rightarrowfill@{}{\hspace{3mm}}\ }
\makeatother
\newcommand{\rt}[1]{\xrightarrow{\ #1\ }}
\newcommand\To{\longrightarrow}

\newcommand\into{\hookrightarrow}
\newcommand{\Into}{\,\ensuremath{\lhook\joinrel\relbar\joinrel\rightarrow}\,}
\newcommand\INTO{\ \ar@{^(->}[r]<-.2ex>}
\newcommand\onto{\to\hspace{-3mm}\to}

\renewcommand\_{^{}_}
\newcommand\Mapsto{\ \longmapsto\ }
\newcommand\take{\smallsetminus}

\renewcommand\={\ =\ }
\newcommand\dbar{\overline\partial}
\renewcommand\udot{^{\bullet}}

\DeclareMathSymbol{\lefttorightarrow}{3}{mathb}{"FC}
\DeclareMathSymbol{\righttoleftarrow}{3}{mathb}{"FD}

\newcommand\At{\operatorname{At}}

\newcommand\rk{\operatorname{rank}}
\newcommand\vir{\operatorname{vir}}
\newcommand\vd{\operatorname{vd}}
\newcommand\tr{\operatorname{tr}}
\newcommand\coker{\operatorname{coker}}
\newcommand\im{\operatorname{im}}

\newcommand\Hom{\operatorname{Hom}}
\renewcommand\hom{\curly H\!om}
\newcommand\End{\operatorname{End}}
\newcommand\Ext{\operatorname{Ext}}

\newcommand\Aut{\operatorname{Aut}}
\newcommand\Pic{\operatorname{Pic}}

\newcommand\Proj{\operatorname{Proj}}
\newcommand\Spec{\operatorname{Spec}}

\newcommand\ZZ{\textstyle{\Z\big[\frac12\big]}}
\renewcommand\ast{\scalebox{0.65}{*}}
\newcommand\Ld{\Lambda^{\hspace{-.4mm}*}}

\newcommand\beq[1]{\begin{equation}\label{#1}}
\newcommand\eeq{\end{equation}}
\newcommand\beqa{\begin{eqnarray*}}
\newcommand\eeqa{\end{eqnarray*}}
\renewcommand\_{^{}_}

\newcommand\arXiv[1]{\href{http://arxiv.org/abs/#1}{arXiv:#1}}
\newcommand\mathAG[1]{\href{http://arxiv.org/abs/math/#1}{math.AG/#1}}

\begin{document}

\title[Counting sheaves on Calabi-Yau 4-folds, I]{Counting sheaves on Calabi-Yau 4-folds, I}
\author{Jeongseok Oh and Richard P. Thomas}

\begin{abstract}
Borisov-Joyce constructed a real virtual cycle on compact moduli spaces of stable sheaves on Calabi-Yau 4-folds, using derived differential geometry.

We construct an algebraic virtual cycle. A key step is a localisation of Edidin-Graham's square root Euler class for $SO(2n,\CC)$ bundles to the zero locus of an isotropic section, or to the support of an isotropic cone.

We prove a torus localisation formula, making the invariants computable and extending them to the noncompact case when the fixed locus is compact.

We give a $K$-theoretic refinement by defining $K$-theoretic square root Euler classes and their localised versions.

In a sequel we prove our invariants reproduce those of Borisov-Joyce.
\end{abstract}

\maketitle
\vspace{-6mm}
\setcounter{tocdepth}{1}
\tableofcontents
\vspace{-1cm}

\section{Introduction}
Let $X$ be a Calabi-Yau 4-fold: a smooth complex projective variety with trivial canonical bundle $K_X\cong\cO_X$. Let $ M$ denote a moduli space of Gieseker stable sheaves on $X$ of fixed topological type. We assume throughout that there are no strictly semistable sheaves, so that $ M$ is projective.

First order deformations of a sheaf $F\in M$ are given by $\mathrm{Ext}^1(F,F)$ and there is a natural obstruction space $\mathrm{Ext}^2(F,F)$. These fit together over $ M$ to give it an ``obstruction theory" in the sense of \cite{BF, LT}; see Section \ref{Sect:realization} for more details. But this obstruction theory is \emph{not} ``perfect" because the higher obstruction space $\mathrm{Ext}^3(F,F)$ need not vanish. Therefore the theories of Behrend-Fantechi and Li-Tian do not give $ M$ a virtual cycle, and we must find another approach.

\subsection*{Formal picture} The crucial feature of the set-up is that \emph{morally} we should think of $ M$ as the
zero locus of an \emph{isotropic} section $s\in\Gamma(\cA,E)$ of an $SO(r,\CC)$ bundle $E$ over a smooth ambient variety $\cA$,
\begin{equation}\label{model}
\xymatrix@=0pt{
& (E,q)\ddto \\  && \hspace{15mm}q(s,s)=0.\hspace{-5mm} \\
 M\ =\ s^{-1}(0)\ \subset\hspace{-2mm} & \cA,\ar@/^{-2ex}/[uu]_s}
\end{equation}
Here $q$ is the quadratic form on $E$ and $q(s,s)=0$ is the isotropic condition.

This moral picture is literally true --- \emph{globally} --- if we allow $\cA$ to be infinite dimensional. This is for gauge theoretic reasons \cite{DT, RTthesis, CL} explained in Appendix \ref{gauge}; in particular see \eqref{dbarsec}.

In finite dimensional algebraic geometry it is also true \emph{locally} about $F\in M$ by the results of \cite{BBBJ, BG, BBJ}, building on \cite{PTVV}.  Here we can take $\cA=\Ext^1(F,F)$ and $(E,q)$ to be the trivial bundle with fibre $\Ext^2(F,F)$ and quadratic form given by Serre duality. (We also use further work \cite{CGJ} of Team Joyce to show the $O(r,\CC)$ bundle $(E,q)$ can be taken to be an $SO(r,\CC)$ bundle.) And while there is in general \emph{no global} finite dimensional model like \eqref{model}, its infinitesimal version --- the derivative of \eqref{model} at $M\subset\cA$ --- \emph{does} globalise. In particular we show in Proposition \ref{isoprop} that the limit of the graphs $\Gamma_{\!ts}$ of the sections $ts$ defines a global \emph{isotropic cone} $C$ in an $SO(r,\CC)$ bundle $E$ over $M$,
\beq{iscone}
C\ :=\ \lim_{t\to\infty}\Gamma_{\!ts}\ \subset\ E.
\eeq
This will be sufficient to construct the ``\emph{right}" virtual cycle, though that is  \emph{not} the intersection of the graph of $s$, or the cone $C$, with the zero section of $E$ (i.e. a localised Euler class of $E$). In the gauge theoretic set-up of Appendix \ref{gauge} this would be the ``\emph{wrong}" answer because $ds|_M$ is not Fredholm. In finite dimensions it would require the (complex) virtual dimension to be $ext^1(F,F)-ext^2(F,F)$, but this varies as we deform $F$.\footnote{Another point of view is that if we change the local model \eqref{model} by replacing $\cA$ by its product with a vector space $V$, then we replace $E$ by $E\oplus\underline{V}\oplus\underline V^*$ with its obvious quadratic structure. This changes the virtual dimension $\dim\cA-\rank E$ by $-\dim V$.}

\subsection*{Borisov-Joyce}
Since $SO(r,\CC)$ is homotopy equivalent to $SO(r,\RR)$ there is an associated real subbundle $E_\RR\subset E$ such that $E=E_\RR\otimes\CC$ and on which the quadratic form is a real metric $|\,\cdot\,|^2$. With respect to the splitting
\begin{equation}\label{EEE}
E\ \cong\ E_{\RR}\oplus iE_{\RR}
\end{equation}
we have a splitting of the section $s=(s^+,s^-)$. The isotropic condition $q(s,s)=0$ implies $|s^+|^2-|s^-|^2=0$, so $|s|^2=|s^+|^2+|s^-|^2=2|s^+|^2$ and\footnote{These are the precise finite dimensional analogues of the gauge theoretic equations \eqref{2eqns} in Appendix \ref{gauge}.}
\begin{equation}\label{s+-}
s\,=\,0\ \iff\ s^+\,=\,0.
\end{equation}
Therefore we can consider $ M$ as being cut out (set-theoretically) by $s^+=0$ instead of $s=0$. So we might give $ M$ a (real!) virtual cycle by intersecting the graph of $s^+$ with the zero section inside $E_\RR$. Equivalently by the projection formula for $E\to E_\R$ we can intersect $\Gamma_{\!s}$ or our isotropic cone $C$ \eqref{iscone} with $iE_\R$ inside $E$; the result lies in the zero section $M$ since $q$ is definite on $iE_\R$ but vanishes identically on $\Gamma_{\!s}$ or $C$.

This is effectively what Borisov-Joyce do in their seminal work \cite{BJ}. They work with real derived differential geometry to glue (in a weak categorical sense) the finite dimensional local models $(\cA,E_\RR,s^+)$ to produce a Kuranishi structure on $ M$ and hence a virtual cycle $[ M]^{\vir}\in H_{\vd}(M,\Z)$. Here the \emph{real} virtual dimension $\vd:=2ext^1(F,F)-ext^2(F,F)$ is
\begin{equation}\label{vd}
\vd\ =\ ext^1(F,F)-ext^2(F,F)+ext^3(F,F)\=2-\chi(F,F)
\end{equation}
by Serre duality. This is independent of $F\in M$.

However, real derived differential geometry is both unfamiliar and complicated, leading to invariants which seem almost impossible to calculate.

\subsection*{Square root Euler classes} 
Morally then --- pretending for now the model \eqref{model} is global --- the Borisov-Joyce virtual cycle is the Euler class of $E_\RR$, localised to the zeros of $s^+$. When $r$ is odd --- equivalently when $\vd$ \eqref{vd} is odd --- this Euler class is 2-torsion, and in \cite{OT2} we show the Borisov-Joyce class vanishes when we use $\ZZ$ coefficients.

So for the rest of this Introduction we assume $\vd$ is even, write $r=2n$ and consider a holomorphic $SO(2n,\C)$ bundle $E$ over a scheme $Y$. (We do \emph{not} require its quadratic structure to be Zariski locally trivial --- in general it will only be locally trivial in the \'etale or analytic topology.) Then \eqref{EEE} gives $e(E)=(-1)^ne(E_\RR)^2$ so we can think of $e(E_\RR)$ as defining a ``square root Euler class" for $E$,
$$
\sqrt e\;(E)\ :=\ e(E_\RR).
$$
It is natural to ask if it lifts to the Chow cohomology group $A^n(Y,\Z)$. While Field and Totaro \cite{Fi} proved it does not in general, Edidin-Graham \cite[Theorem 3]{EG:Ch} proved that it \emph{does} lift if we use $\ZZ$ coefficients.

In place of the real subbundle $E_\RR\subset E$, Edidin-Graham use a maximal isotropic holomorphic subbundle\footnote{By replacing $Y$ by a certain bundle $\wt Y$ over it (see Section \ref{EdidinGraham} and cf. the splitting principle) we may assume without loss of generality that such a $\Lambda$ exists. This is the step that forces us to invert 2 in the coefficients.} $\Lambda\subset E$. Notice that projection to $E_\RR$ gives an isomorphism $\Lambda\cong E_\RR$ by the computation \eqref{s+-}. Thus $e(E)=(-1)^ne(\Lambda)^2$, as can also be seen from the exact sequence
$$
0\To\Lambda\To E\To\Ld\To0.
$$
We denote the resulting Edidin-Graham class by
$$
\sqrt e\;(E)\ :=\ (-1)^{|\Lambda|}c_n(\Lambda)\ \in\ A^n(Y).
$$
This is independent of $\Lambda$. The sign $(-1)^{|\Lambda|}$ is specified in Definition \ref{Lor}.

\subsection*{Cosection localisation} 
This suggests defining a version of the Edidin-Graham class localised to the zero locus of an isotropic section, or to the support of an isotropic cone, to produce an algebraic version of Borisov-Joyce's class in the Chow group of $ M$.

So given an isotropic section $s$, we try to intersect its graph $\Gamma_{\!s}$ with a maximal isotropic subbundle $\Lambda\subset E$. This is entirely analogous to the real setting where we intersected $\Gamma_{\!s}$ with $iE_\RR\subset E$. But while the computation \eqref{s+-} showed that the intersection $\Gamma_{\!s}\cap iE_\RR$ lies in the zero section $0_E$, this does \emph{not} hold for $\Gamma_{\!s}\cap\Lambda$ in general.

Instead we note that the normal bundle to $\Lambda\subset E$ is (the pullback of) $\Ld$ so the tautological section of $\Lambda$ defines a natural \emph{cosection} $N_{\Lambda/E}\to\cO$. The isotropic condition implies that the cone $C_{\Gamma_{\!s}\cap\Lambda/\Gamma_{\!s}}\subset N_{\Lambda/E}$ lies in the kernel of this cosection. Therefore by Kiem-Li \cite{KL} the Fulton-MacPherson intersection of $\Gamma_{\!s}$ and $\Lambda$ inside $E$ can be localised to the zeros of this cosection on $\Gamma_{\!s}\cap\Lambda$, i.e. to the zeros of $s$.

\subsection*{Results} We summarise some of our results, leaving full explanations and definitions to the main text.

\begin{Thm*}
Fix an $SO(2n,\C)$ bundle $E$ over a scheme $Y$ and an isotropic section $s$ with zeros $\iota\colon Z(s)\into Y$. There is an operator
$$
\sqrt e\;(E,s)\ \colon\,A_*\(Y,\ZZ\)\To A_{*-n}\(Z(s),\ZZ\)
$$
such that $\iota_*\circ\sqrt e\;(E,s)$ is cap product with the Edidin-Graham class $\sqrt e\;(E)$.
\end{Thm*}

The Fulton-MacPherson construction of the localisation of the classical Euler class to the zeros of a section uses the Gysin operator $0_E^{\;!}$ associated to the zero section $0_E$ of the bundle. In our setting there is also a square-rooted analogue.

\begin{Thm*}
Fix an $SO(2n,\C)$ bundle $E$ over a scheme $Y$ and an isotropic cone $C\subset E$ supported over $Z\subseteq Y$. There is an operator
$$
\sqrt{0_E^{\;!}}\ \colon\,A_*\(C,\ZZ\)\To A_{*-n}\(Z,\ZZ\)
$$
with good properties. When $C$ is contained in a maximal isotropic subbundle $\Lambda\subset E$ then $\surd\;{0_E^{\;!}}=(-1)^{|\Lambda|}0_\Lambda^{\;!}$.
\end{Thm*}

Precise details are in Section \ref{EdidinGraham}. There we prove various compatibilities and desirable properties of these operators such as their behaviour under passing to the reduction $K^\perp/K$ by an isotropic subbundle $K\subset E$, and a Whitney sum formula.

We also prove $K$-theoretic analogues. For explanations and full details see Section \ref{Ksec}; for now let us simply recall that the $K$-theoretic Euler class of a bundle $E$ on a scheme $Y$ is
$$
\mathfrak e(E)\ :=\ \sum_{i=0}^{\rk E}(-1)^i\Lambda^iE^*\ \in\ K^0(Y).
$$
When $E$ has a section $s$ transverse to its zero section $0_E$ this is the class $[\cO_{Z(s)}]$ of the structure sheaf of its zero locus. When $E$ is an $SO(2n,\C)$ bundle we define a square-rooted analogue $\e(E)$.

\begin{Thm*}
Fix an $SO(2n,\C)$ bundle $E$ over a scheme $Y$. There is a natural class
$$
\e(E)\ \in\ K^0\(Y,\ZZ\) \quad\text{such that }\ \e(E)^2\=(-1)^n\mathfrak e(E).
$$
When $E$ admits a maximal isotropic subbundle $\Lambda\subset E$ then
$$
\e(E)\=(-1)^{|\Lambda|}\mathfrak e(\Lambda)\otimes\sqrt{\det\Lambda}.
$$
\end{Thm*}

The use of $\ZZ$ coefficients ensures the existence and uniqueness of $\surd\!\det\Lambda$; see Lemma \ref{sqrts}. In various special cases --- such as when the $SO(2n,\C)$ bundle is Zariski locally trivial, admits a spin bundle or admits a maximal isotropic subbundle --- there are many related classes in the literature  \cite{And, CLL, Chiodo, EG:Ch, KO, OS, PV:A}. Some use different twistings, some use rational coefficients, some are localised to the zeros of an isotropic section, and most are defined only up to an overall sign. Our class unifies them, fixing the overall sign and twisting to ensure its square is $(-1)^n\mathfrak e(E)$, while removing any assumptions of Zariski local triviality, the existence of maximal isotropics or of spin bundles.

\begin{Thm*}
Fix an $SO(2n,\C)$ bundle $E$ over a scheme $Y$ and an isotropic section $s$ with zeros $\iota\colon Z(s)\into Y$. There is an operator
$$
\sqrt{\mathfrak e}\;(E,s)\ \colon\,K_0\(Y,\ZZ\)\To K_0\(Z(s),\ZZ\)
$$
such that $\iota_*\circ\sqrt{\mathfrak e}\,(E,s)$ is tensor product with $\e(E)\in K^0\(Y,\ZZ\)$.
\end{Thm*}

The $K$-theoretic analogue of the Gysin operator $0_E^{\;!}$ is the (derived) pullback operator $0_E^*$. Again, there is a square-rooted version.

\begin{Thm*}
Fix an $SO(2n,\C)$ bundle $E$ over a scheme $Y$ and an isotropic cone $C\subset E$ supported over $Z\subset Y$. There is an operator
$$
\sqrt{0_E^*}\ \colon K_0\(C,\ZZ\)\To K_0\(Z,\ZZ\)
$$
with good properties. When $C$ is contained in a maximal isotropic subbundle $\Lambda\subset E$ then $\surd\;{0_E^*}=(-1)^{|\Lambda|}\sqrt{\det\Lambda}\cdot0_\Lambda^*$.
\end{Thm*}

\subsection*{Virtual cycles}
Proposition \ref{isoprop} proves the cone $C\subset E$ of \eqref{iscone} is isotropic. Therefore we may
apply $\surd\;{0_E^{\;!}}$ to $[C]$, and $\surd\;{0_E^*}$ to $[\cO_C]$, to define our virtual cycle and twisted\footnote{The twisting by $\sqrt{\det T_\cA^*}$ is necessary to get a well-defined class; see the discussion after Definition \ref{vss}.} virtual structure sheaf.

\begin{Thm*}\label{thm1} The following are well-defined, independent of choices, and deformation invariant,
\beqa
[M]^{\vir} &:=& \sqrt{0_E^{\;!}}\,[C]\ \in\ A_{\frac12\!\vd}\( M,\ZZ\), \\
\Ohat &:=& \sqrt{0_E^*}\,[\cO_C]\cdot\sqrt{\det T_\cA^*}\ \in\ K_0\(M,\ZZ\).
\eeqa
They are related over $\Q$ by a virtual Riemann-Roch formula,
$$
\tau\_M\(\Ohat\)\=\sqrt\td\(T_M^{\vir}\)\cap[M]^{\vir},
$$
where $T^{\vir}_M|\_F=\Ext^1(F,F)-\Ext^2(F,F)+\Ext^3(F,F)$. In particular,
$$
\chi\(\Ohat\)\=\int_{[M]^{\vir}}\sqrt\td\(T_M^{\vir}\).
$$
\end{Thm*}

In fact we can define $[M]^{\vir},\,\Ohat$ when $X$ is quasi-projective and $M$ parameterises compactly supported sheaves thereon. (Projectivity is only required to prove deformation invariance.) To show they are nontrivial we calculate them on \emph{local} Calabi-Yau 4-folds $X=K_Y$ in Section \ref{KY}. There they reduce to the usual 3-fold virtual cycle --- and Nekrasov-Okounkov-twisted virtual structure sheaf --- on the moduli space of sheaves on $Y$ when all stable sheaves on $X$ are pushed forward from the zero section $Y\into X$ (for instance when $Y$ is Fano).

In the sequel \cite{OT2} we show that on projective Calabi-Yau 4-folds $X$ our virtual cycle agrees with Borisov-Joyce's \cite{BJ} on inverting 2. 

\begin{Thm*}[\cite{OT2}] For $X$ projective and $\vd$ even our class $[M]^{\vir}$ and Borisov-Joyce's $[M]^{\vir}_{BJ}\in H_{\vd}(M,\Z)$ have the same image under the maps
$$
\xymatrix@R=-5pt{
[M]^{\vir}\,\in\,A_{\frac12\!\vd}\(M,\ZZ\) \\
& H_{\vd}\(M,\ZZ\). \ar@{<-}[ul]+<65pt,-1pt>\ar@{<-}[dl]+<65pt,1pt> \\
\hspace{7mm}[M]_{BJ}^{\vir}\,\in\,H_{\vd}(M,\Z)}
$$
If $\vd$ is odd then Borisov-Joyce's class is zero after inverting 2.
\end{Thm*}

In particular the resulting invariants --- given by integrating insertions in integral cohomology over the virtual cycle --- lie in $\Z\subset\ZZ$, which is not immediately obvious from our construction.

As a further indication that our classes should be more computable than Borisov-Joyce's, we prove torus localisation formulae for them. 

\begin{Thm*} Suppose $\T:=\C^*$ acts on a quasi-projective Calabi-Yau 4-fold $X$ \emph{preserving the holomorphic 4-form}. Let $\iota\colon M^\T\into M$ denote the fixed locus of the induced $\T$ action on $M$. Then
\beqa
\quad [M]^{\vir} &=& \iota_*\,\frac{\big[M^\T\big]^{\vir}}{\sqrt{e\_{\T}}\;(N^{\vir})}\ \in\ A^\T_{\frac12\!\vd}(M,\Q)\big[t^{-1}\big], \\
\Ohat &=& \iota_*\,\frac{\widehat\cO^{\;\vir}_{\!M^\T}}{\sqrt{\mathfrak e\_{\T}}\;(N^{\vir})}\ \in\ K_0^\T(M)\otimes\_{\;\Z[\t,\t^{-1}]}\Q(\t^{1/2}).
\eeqa
\end{Thm*}

Here $\t$ is the 1-dimensional weight 1 representation of $\T$ and $t=c_1(\t)$. For full details and explanations see Section \ref{torus}. In particular this allows us to make sense of CY$^4$ sheaf counting (for insertions which lift to equivariant cohomology) and $K$-theoretic sheaf counting in \emph{noncompact} equivariant settings, so long as the fixed locus $ M^{\T}$ is compact. Finally we show in \eqref{RRT} that the isomorphism $\tau\_{M^\T}\colon K_0(M^\T)_\Q\to A_*(M^\T)_\Q$ maps the second localisation formula to the first\,$\times\sqrt\td(T^{\vir}_M)$. Thus the equivariant holomorphic Euler characteristic $\chi\_t$ (the $\T$-character of $R\Gamma$) can be described by
$$
\chi\_{e^t}\!\left(\frac{\widehat\cO^{\;\vir}_{\!M^\T}}{\sqrt{\mathfrak e\_{\T}}\;(N^{\vir})}\right)\=\int_{\big[M^\T\big]^{\vir}}\frac{\sqrt\td\(T_M^{\vir}\)\big|_{M^\T}}{\sqrt{e_\T}\,(N^{\vir})}\,.
$$
Furthermore Kiem and Park \cite{KP} have now proved a version of cosection localisation for $[M]^{\vir},\,\Ohat$, writing them as the pushforward of classes supported on the zero locus of an isotropic section of the obstruction sheaf.

There are already a number of papers making calculations of DT$^4$ invariants by assuming the Borisov-Joyce virtual cycle takes a nice algebro-geometric form in special cases \cite{CK, CK3, CKM1, CKM2,CMT1,CMT2,CT1, CT2, DSY, Ne, NP}. This paper gives those references a theoretical foundation, justifying many of their results.
\medskip

\noindent\textbf{Acknowledgements.} Cao-Leung \cite{CL} were the first to highlight the relevance of maximal isotropic subbundles and the Edidin-Graham class to DT$^4$ theory. We are very grateful to Yalong Cao for the suggestion \cite{Cao} that the virtual cycle might therefore be algebraic.

We thank Sara Filippini for introducing us, Nick Addington, Dave Anderson, Noah Arbesfeld, Arkadij Bojko, Chris Brav, Dennis Borisov, Young-Hoon Kiem, Andrei Okounkov, Rahul Pandharipande, Mauro Porta, J{\o}rgen Rennemo, Feng Qu, Yukinobu Toda, Burt Totaro, Ming Zhang and two conscientious referees for  useful comments and conversations.

We particularly thank Dominic Joyce for generous help with \cite{BBBJ, PTVV}, Martijn Kool for many suggestions --- especially his insistence that we work out the $K$-theoretic story and explain the twist \eqref{NOtwist} --- and Davesh Maulik for a crucial observation about orientations on $N^{\vir}$ in \makebox{Section \ref{torus}.}

J.O. thanks the Fields Institute for excellent working conditions and acknowledges support of a KIAS Individual Grant MG063002.
R.P.T. acknowledges support from a Royal Society research professorship and EPSRC grant EP/R013349/1.\medskip

\noindent\textbf{Notation.} We use $E^*$ for the dual of a vector bundle $E$, reserving $E^\vee$ for the derived dual of a complex or object of the derived category.

Given a codimension $d$ regular embedding $\iota\colon Z\into Y$ of schemes, we use $\iota^!$ to denote both the Gysin map $A_*(Y)\to A_{*-d}(Z)$ and the refined Gysin map $A_*(Y')\to A_{*-d}(Z')$ for any scheme $Y'\to Y$ and $Z':=Y'\times\_YZ$. These are denoted $\iota^*$ and $\iota^!$ respectively in \cite[Section 6.2]{F}.

For determinants, duals and signs we adopt the conventions of \cite{KM}. For vector bundles $A,B$ of ranks $n,m$ we use the standard isomorphism
\begin{eqnarray}\label{A+B}
\det(A\oplus B) &\cong& \det A\otimes\det B, \\ \nonumber
a_1\wedge\ldots\wedge a_n\wedge b_1\wedge\ldots\wedge b_m &\mapsto& (a_1\wedge\ldots\wedge a_n)\otimes(b_1\wedge\ldots\wedge b_m),
\end{eqnarray}
which we write as ${\bf a\wedge b}\mapsto{\bf a\otimes b}$. Combined with the isomorphism $A\oplus B\cong B\oplus A,\ (a,b)\mapsto(b,a)$ we find that the identification
\beq{ABBA}
\det B\otimes\det A\ \cong\ \det(B\oplus A)\ \cong\ \det(A\oplus B)\ \cong\ \det A\otimes\det B
\eeq
is given by the Koszul sign rule\footnote{In \cite{KM} this is made explicit by decorating $\det A$ by the integer $n=\rk(A)$; we silently remember $\rk(A)$ mod $2$ without recording it.} ${\bf b}\otimes{\bf a}\mapsto(-1)^{mn}\,{\bf a}\otimes{\bf b}$. For line bundles $L,M$ we identify $(L\otimes M)^*\cong M^*\otimes L^*$ in the obvious way \cite[page 31]{KM}. Combining these two conventions leads to the pairing
\beq{pairA}
\det A\otimes\det A^*\To\cO, \quad (a_1\wedge\dots\wedge a_n)\otimes(a_n^*\wedge\dots\wedge a_1^*)\Mapsto 1,
\eeq
where $\{a_i\}_{i=1}^n$ is a local basis of sections of $A$ and $\{a_i^*\}_{i=1}^n$ is the dual basis. Beware this pairing privileges $\det A$ over $\det A^*$; if we swap their roles and apply \eqref{ABBA} we get another map $\det A\otimes\det A^*\to\cO$ which differs from \eqref{pairA} by the sign $(-1)^n$.

These conventions ensure the commutativity of the diagram
\beq{AB*BA}
\xymatrix@R=18pt{
\ \ \,\det A\otimes\det B\otimes\det B^*\otimes\det A^* \ar[d]_{\eqref{pairA}_B}\ar@{=}[r]& \det\!\((A\oplus B)\oplus(A\oplus B)^*\)\, \ar[d]^{\eqref{pairA}_{A\oplus B}} \\
\ \det A\otimes\det A^* \ar[r]^-{\eqref{pairA}_A}& \cO.\!}
\eeq
We write \eqref{pairA} as ${\bf a}\otimes{\bf a}^*\mapsto 1$. That is, we have set
\beq{adual}
{\bf a}^*\ :=\ a_n^*\wedge\dots\wedge a_1^*.
\eeq
The alternative convention of setting ${\bf a}^*$ to be $a_1^*\wedge\dots\wedge a_n^*$ has the advantage of eliminating the unpleasant sign in Definition \ref{orr} below, but at the (worse) expense of making \eqref{AB*BA} commute only up to a sign.

Finally we note that where we refer to specific numbering in published papers, we use the arXiv version where possible.

\section{(Special) orthogonal bundles}
In this paper a central role will be played by orthogonal bundles, their maximal isotropic subbundles, their maximal positive definite real subbundles, and the relationsip between the two. Throughout we work over a complex quasi-projective scheme $Y$. 

\subsection{$O(r)$ and $SO(r)$ bundles}
By an orthogonal bundle we mean a pair $(E,q)$, where $E$ is a rank $r$ vector bundle over $Y$ and
$$
q\ \colon\,E\otimes E\To\cO_Y
$$
is a nondegenerate quadratic form. By the Gram-Schmidt process (which uses square roots) we can pick an \'etale-local orthonormal basis for $(E,q)$,
\beq{o.n.}
e_1,\dots,e_r \quad\text{such that}\quad q(e_i,e_j)=\delta_{ij}.
\eeq
Transition functions from one such normal form to another lie in $O(r,\C)$, giving a 1-1 correspondence between isomorphism classes of pairs $(E,q)$ and $O(r,\C)$ principal bundles in the \'etale topology. 

The quadratic form induces an isomorphism $q\colon E\xrightarrow{\,\sim\,}E^*$ and thus an isomorphism $\det q\colon\det E\to\det E^*$, giving
\beq{or2}
(\det E\;)^{\otimes2}\ \xrightarrow[\raisebox{.5ex}{$\sim$}]{\ 1\otimes\det q\ }\ \det E\otimes\det E^*\ \rt\sim\ \cO_Y,
\eeq
where the final arrow is \eqref{pairA}.
%$$
%(e_1\wedge\dots\wedge e_r)\otimes(f_r\wedge\dots\wedge f_1)\Mapsto \det\!\(f_i(e_j)\)_{1\le i,j\le r\,}.
%$$
%That is, if $\{e_i\}$ is a local basis for $E$ with dual basis $\{f_i\}$ for $E^*$ then \eqref{pairA} takes $(e_1\wedge\dots\wedge e_r)\otimes(f_r\wedge\dots\wedge f_1)\mapsto1$.
Choosing $\{f_i\}$ to be the dual basis to the \'etale local orthonormal trivialisation \eqref{o.n.} we have $q(e_i)=f_i$, so \eqref{or2} takes
$$
(e_1\wedge\dots\wedge e_r)^{\otimes2}\,\Mapsto\,(e_1\wedge\dots\wedge e_r)\otimes(f_1\wedge\dots\wedge f_r)\,\Mapsto\,(-1)^{\frac{r(r-1)}2}.
$$
We define an \emph{orientation} to be a global choice of $\pm e_1\wedge\dots\wedge e_r$, when one exists. Notice this is a $\Z/2$ choice, not a $\C^*$ choice.

\begin{Def}\label{orr} An orientation on $(E,q)$ is a trivialisation
\beq{OR}
o\ \colon\,\cO_Y\To\Lambda^rE
\eeq
whose square maps to $(-1)^{\frac{r(r-1)}2}$ under \eqref{or2}.
\end{Def}

Transition functions from one orthonormal basis to another, preserving $o$, lie in $SO(r,\C)$. Therefore an orientation is equivalent to a reduction of structure group of our \'etale principal frame bundle from $O(r,\C)$ to $SO(r,\C)$, giving a 1-1 correspondence between isomorphism classes of $SO(r,\C)$ principal bundles and oriented orthogonal bundles $(E,q,o)$. 

Note that if $(E,q)$ is a Zariski locally trivial\footnote{By this we mean $E$ admits $q$-orthonormal frames Zariski locally. Such bundles $(E,q)$ are in 1-1 correspondence with Zariski locally trivial $O(r,\C)$ bundles, and admit two reductions to $SO(r,\C)$ bundles. Beware that \cite{EG:Ch} use a different definition of Zariski local triviality, probably for characteristic 2 reasons that we can ignore.} orthogonal bundle then its $\Z/2$-bundle of local orientations is Zariski locally trivial and hence trivial. Thus $(E,q)$ is automatically orientable.

\subsection{Maximal positive definite real subbundle}
%Consider $Y$ as a topological space in the Euclidean topology, over which $E$ defines a topological $O(r,\C)$ bundle.
Let $E_\R\subset E$ denote a maximal real subbundle on which $q$ is real and positive definite. For instance we can use a partition of unity argument to patch together the real spans of the local orthonormal frames \eqref{o.n.}. We get a real orthogonal splitting
$$
E\ \cong\ E_\R\oplus iE_\R\=E_\R\otimes\_\R\C.
$$
Conversely, given a positive definite real quadratic form $q\_\R$ on $E_\R$, we get a negative definite quadratic form on $iE_\R$, and their direct sum is a nondegenerate complex quadratic form $q=q\_\R\otimes\C$ on $E$. Thus, topologically, $O(r,\C)$ and $O(r,\R)$ bundles are equivalent, reflecting the homotopy equivalence $O(r,\R)\subset O(r,\C)$.

Since $SO(r,\R)\subset SO(r,\C)$ is also a homotopy equivalence, an orientation on $E$ in the sense of Definition \ref{orr} must be equivalent to an orientation on $E_\R$ in the usual sense. Obviously a trivialisation $o\_{\;\R}$ of $\Lambda^r_\R E_\R$ induces a trivialisation $o\_{\;\R}\!\otimes\_\R\!\C$ of $\Lambda^rE\cong(\Lambda^r_\R E_\R)\otimes\_\R\C$. In our local orthonormal basis \eqref{o.n.} this takes $e_1\wedge\_\R\dots\wedge\_\R e_r$ to $e_1\wedge\dots\wedge e_r$, which satisfies Definition \ref{orr}. Conversely, given an orientation $o$ in the sense of \eqref{OR}, its real part $o\_{\;\R}$ under $\Lambda^rE\cong\Lambda^r_\R E_\R\,\oplus\,i\Lambda^r_\R E_\R$ gives a classical orientation on $E_\R$ --- an element of $\Lambda_\R^rE_\R$ satisfying $o_{\;\R}^{\otimes2}\mapsto(-1)^{\frac{r(r-1)}2}$ under
$$
\Lambda_\R^rE_\R\otimes\Lambda_\R^rE_\R\rt{1\otimes \Lambda_\R q\_\R}\Lambda_\R^rE_\R\otimes\Lambda_\R^rE^*_\R\rt{\ \eqref{pairA}_\R\ }\cO_Y.
$$

\subsection{Maximal isotropic subbundles} Write $r$ as either $2n$ or $2n+1$. We call an algebraic subbundle $\Lambda\subset(E,q)$ \emph{isotropic} if $q|_\Lambda\equiv0$ and \emph{maximal isotropic} if it also has maximal rank $n$. The quadratic form gives a surjection $q\colon E/\Lambda\to\Ld$. When $r=2n$ and $\Lambda$ is maximal isotropic, we thus get a short exact sequence
\beq{LL*}
0\To\Lambda\To E\To\Ld\To0.
\eeq
Zariski locally we may pick a basis of sections $e_1,\dots,e_n$ for $\Lambda$ with dual basis $f_1,\dots,f_n$ for $\Ld$. By a version of the Gram-Schmidt process (without square roots!)\footnote{Replace any lift of $f_1$ to $E$ by $f_1-\frac12q(f_1,f_1)e_1$ to make it isotropic. Do the same for $f_2$, then replace it by $f_2-q(f_1,f_2)e_1$ to ensure it is orthogonal to $f_1$. Etc.\label{GStriv}} we can lift the $f_i$ to $E$ in such a way that their span $\Ld\subset E$ is also maximal isotropic. Thus, Zariski locally, \eqref{LL*} splits and $E$ is trivial with basis of sections $e_1,\dots,e_n,f_1,\dots,f_n$ satisfying
\beq{efij}
q(e_i,e_j)\=0\=q(f_i,f_j), \qquad q(e_i,f_j)\=\delta_{ij}.
\eeq
When $r=2n+1$, instead $E/\Lambda$ sits inside an exact sequence
$$
0\To\Lambda^\perp/\Lambda\To E/\Lambda\To\Ld\To0,
$$
where $\Lambda^\perp/\Lambda$ is a line bundle with a nondegenerate quadratic form inherited from $q$. Working \'etale locally this time, we may choose an orthonormal section $e$ of $\Lambda^\perp/\Lambda$ and split as before to get an \'etale local basis of sections $e_1,\dots,e_n,f_1,\dots,f_n,e$ satisfying \eqref{efij} and
\beq{ije}
q(e,e_i)\=0\=q(e,f_i), \qquad\ q(e,e)=1.
\eeq
Thus we get normal forms \eqref{efij} and \eqref{ije} for $(E,q)$ which differ from the orthonormal form \eqref{o.n.}. (They correspond to writing the first $2n$ terms of the quadratic form $\sum_{i=1}^rx_i^2$ as $\sum_{i=1}^nx_iy_i$ by a change of basis.)

\subsection*{Signs of maximal isotropics}
Fix $r=2n$. Continuing with the same local basis $\{e_i,f_i\}$ from \eqref{efij}, let $\{E_i,F_i\}$ denote the dual basis of $E^*$. Then $q\colon E\to E^*$ maps $e_i\mapsto F_i,\,f_i\mapsto E_i$ so \eqref{or2} takes $(e_1\wedge f_1\wedge\dots\wedge e_n\wedge f_n)^{\otimes2}$ to
$$
(e_1\wedge f_1\wedge\dots\wedge e_n\wedge f_n)\otimes(F_1\wedge E_1\wedge\dots\wedge F_n\wedge E_n)\hspace{3cm}\vspace{-2mm}
$$$$
\hspace{1cm}\=(e_1\wedge f_1\wedge\dots\wedge e_n\wedge f_n)\otimes(F_n\wedge E_n\wedge\dots\wedge F_1\wedge E_1)\stackrel{\eqref{pairA}}{\Mapsto}1.
$$
So in this normal form, an orientation $o$ is one of $\pm i^ne_1\wedge f_1\wedge\dots e_n\wedge f_n$ whose square maps to $(-1)^n=(-1)^{\frac{r(r-1)}2}$ under \eqref{or2}.

\begin{Def}\label{Lor} Given a maximal isotropic subbundle $\Lambda\subset(E,q,o)$, define its sign $(-1)^{|\Lambda|}$  by $o=(-1)^{|\Lambda|}(-i)^ne_1\wedge f_1\wedge\dots\wedge e_n\wedge f_n$.

We call $\Lambda$ \emph{positive} if $(-1)^{|\Lambda|}=+1$ and \emph{negative} if $(-1)^{|\Lambda|}=-1$.
\end{Def}

Notice $\Lambda$ is isomorphic (as a real bundle) to $E_\R$ via the composition
\beq{LamER}
\Lambda\Into E\ \cong\ E_\R\oplus iE_\R\To E_\R,
\eeq
since $q$ is negative definite on $iE_\R$ and zero on $\Lambda$ so their intersection is $\{0\}$; see also \cite[Proposition 2.9]{CK}. The choice of sign in Definition \ref{Lor} ensures that $\Lambda$ is positive if and only if its real orientation (induced by its complex structure) agrees with the real orientation $o\_{\;\R}$ on $E_\R$ under \eqref{LamER}; see Proposition \ref{realiso} below.

It is easy to check that picking a different local basis $e_i$ for $\Lambda$ we get the same notion of sign. Replacing $o$ by $-o\,$ swaps signs. The sets of positive and negative maximal isotropic subspaces of $\C^{2n}$ (with its standard quadratic form) form the two connected components of the orthogonal Grassmannian $\operatorname{OGr}(n,2n)$.\medskip

Put another way, given a maximal isotropic $\Lambda\subset(E,q)$, the exact sequence \eqref{LL*} gives a canonical isomorphism
$$
\Lambda^{2n}E\ \cong\ \Lambda^n(\Lambda)\otimes\Lambda^n(\Ld)\ \cong\ \cO_Y,
$$
taking $(e_1\wedge\dots\wedge e_n)\otimes(f_n\wedge\dots\wedge f_1)$ in the central term to 1 in $\cO_Y$ according to our convention \eqref{pairA}. The $2n$-forms corresponding to $\pm i^n\in\Gamma(\cO_Y)$ are orientations in the sense of Definition \ref{orr}. Choosing
\beq{specify}
(-i)^n\ \in\ \Gamma(\cO_Y)\ \cong\ \Gamma(\Lambda^{2n}E)
\eeq
defines an orientation %\footnote{{\color{red}It comes from noting $e_1\wedge f_1\wedge\dots\wedge e_n\wedge f_n=e_1\wedge\dots\wedge e_n\wedge f_n\wedge\dots\wedge f_1$.}} 
$(E,q,o)$ with respect to which $\Lambda$ is positive.
In particular the existence of a maximal isotropic reduces the structure group of $(E,q)$ from $O(2n,\C)$ to $SO(2n,\C)$. (In fact more is true: $(E,q)$ is Zariski locally trivial, as we showed in \eqref{efij}.) \medskip

So on $\C^2$ with its standard quadratic form $z_1^2+z_2^2$ and orientation ${1\choose0}\wedge{0\choose 1}$ our conventions are that $\Lambda_+=\langle{1\choose-i}\rangle$ is positive and $\Lambda_-=\langle{1\choose i}\rangle$ is negative. Let $\R^2\subset\C^2$ be the standard maximal positive definite real subspace Im$\,z_i=0$. From the orientation on $\C^2$ it inherits the real orientation ${1\choose0}\wedge\_\R{0\choose 1}$.

Under \eqref{LamER} we see ${1\choose-i}$ and $i{1\choose-i}$ map to ${1\choose0}$ and ${0\choose 1}$ respectively. So \eqref{LamER} preserves orientations. Taking products of this standard model proves the following.

\begin{Prop}\label{realiso} Fix an $SO(2n,\C)$ bundle $(E,q,o)$. A maximal isotropic $\Lambda\subset E$ is positive if and only if the orientation it induces on $E_\R$ is $o\_{\;\R}.\hfill\qed$
\end{Prop}

\section{Localised Edidin-Graham class}\label{lEG}

Fix a quasi-projective scheme $Y$. 
Since $SO(r,\C)\supset SO(r,\R)$ is a homotopy equivalence, the Euler class $e(E_\R)\in H^r(Y,\Z)$ provides a topological characteristic class for $SO(r,\C)$ bundles $(E,q,o)$ over $Y$. We call it a ``\emph{square root Euler class}" $\sqrt e\;(E)$ since\footnote{An oriented basis $e_1,\dots,e_r$ for $E_\R$ gives a real basis $e_1,\dots,e_r,ie_1,\dots,ie_r$ for $E=E_\R\oplus iE_\R$. The induced orientation differs by $(-1)^{\frac12r(r-1)}$ from the natural real orientation induced by the complex structure on $E$, which has oriented real basis $e_1,ie_1,\dots,e_r,ie_r$.}
$$
e(E)\=(-1)^{\frac12r(r-1)}e(E_\R)^2.
$$
For the rest of this Section we fix $r=2n$. It is then natural to ask if $e(E_\R)$ lifts to the Chow cohomology group $A^n(Y,\Z)$. Field and Totaro \cite{Fi} proved it does not in general, but Edidin-Graham \cite[Theorem 3]{EG:Ch} proved that it \emph{does} lift if we use $\ZZ$ coefficients.

In place of the maximal real positive definite subbundle $E_\R\subset E$, Edidin-Graham use maximal isotropic subbundles $\Lambda\subset E$. Since these need not exist on $Y$ in general, they pull back to a bundle $\wt Y\to Y$ on which they do (cf. the splitting principle). Pulling back to $\wt Y$ is what loses the 2-primary information, forcing us to invert 2.

\subsection{Edidin-Graham square root Euler class} \label{EdidinGraham}
Given an $SO(2n,\C)$ bundle $(E,q,o)$ we let $\rho\colon\wt Y\to Y$ be the bundle of length $n-1$ isotropic flags in $E$ \cite[Section 6]{EG:Ch}. So as a set,
\beq{cover}
\wt Y\ :=\ \big\{(y,E_1\subset\dots\subset E_{n-1}\subset E_y)\ \colon\,y\in Y,\ \dim E_i=i,\ q|\_{E_i}\equiv0\big\}.
\eeq
Let $\cE_i\subset\rho^*E$ denote the tautological isotropic subbundle of rank $i$. Then $\cE_{n-1}^\perp/\cE_{n-1}$ inherits an orthogonal structure from $q$.  Via the composition 
$$
\rho^*\Lambda^{2n}E\,\cong\,\Lambda^2\(\cE_{n-1}^\perp/\cE_{n-1}\)\otimes\Lambda^{n-1}\cE_{n-1}\otimes\Lambda^{n-1}(\cE_{n-1}^*)\xrightarrow[\raisebox{.5ex}{$\sim$}]{\eqref{pairA}}\Lambda^2\(\cE_{n-1}^\perp/\cE_{n-1}\)
$$
the trivialisation $i^{n-1}\rho^*(o)$ of the left hand side induces an orientation on the right hand side. Thus $\cE_{n-1}^\perp/\cE_{n-1}$ is an $SO(2,\C)$ bundle, and so splits as a direct sum of isotropic line bundles $L\oplus L^{-1}$, with $L$ positive and $L^{-1}$ negative. (Locally this is just the observation that the quadratic form $xy$ on $\C^2$ admits precisely two isotropic lines $x=0$ and $y=0$; globally we use the orientation to choose the positive line and call it $L$.) 

Pulling $L\subset\cE_{n-1}^\perp/\cE_{n-1}$ back to $\cE_{n-1}^\perp$ defines a positive maximal isotropic{\footnote{Picking local bases $e_1,...,e_{n-1}$ for $\cE_{n-1}$ and $e_n$ for $L$ with the dual bases $f_1,...,f_{n-1}$ for $\cE^*_{n-1}$ and $f_n$ for $L^*$, \eqref{pairA} takes $-i\;e_1\wedge f_1\wedge\dots\wedge e_n\wedge f_n$ to $-ie_n\wedge f_n$. The positivity of $L$ implies then $i^{n-1}\rho^*(o)=-i\;e_1\wedge f_1\wedge\dots\wedge e_n\wedge f_n$, which tells us $\Lambda_\rho$ is positive.}
$$
\Lambda_\rho\ :=\ L+\cE_{n-1}\ \subset\ \cE_{n-1}^\perp\ \subset\ \rho^*E
$$
and Edidin-Graham prove that $c_n(\Lambda_\rho)$ descends to $Y$ if we invert 2.

More precisely, there exists a distinguished class with degree $2^{n-1}$ \mbox{over $Y,$}
\beq{hdef}
h\ \in\ A^{n(n-1)}\(\wt Y,\Z\)\ \text{ with }\ \rho_*\;h\=2^{n-1}
\eeq
by \cite[Proposition 5]{EG:Ch}. Using this Edidin-Graham define
\beq{pushy}
\sqrt e\;(E)\ :=\ \frac1{2^{n-1}}\rho_*\(h\cup c_n(\Lambda_\rho)\)\ \in\ A^n\(Y,\textstyle{\Z\big[\frac12\big]}\)
\eeq
and show it satisfies
$$
\rho^*\sqrt e\;(E)\=c_n(\Lambda_\rho).
$$
Moreover it is the unique class with this property since $\rho^*$ is injective with left inverse $2^{1-n}\rho_*(h\cup\ \cdot\ )$ on $A^*\(\ \cdot\ ,\Z\big[\frac12\big]\)$. The exact sequence $0\to\Lambda_\rho\to\rho^*E\to\Lambda_\rho^*\to0$ gives $\rho^*\;c_{2n}(E)=c_n(\Lambda_\rho)c_n(\Lambda_\rho^*)=(-1)^nc_n(\Lambda_\rho)^2$, so
$$
e(E)\=(-1)^n\(\sqrt e\;(E)\)^2.
$$
Moreover, since the composition $\Lambda_\rho\subset\rho^*E\to\rho^*E_\R$ preserves orientations, $c_n(\Lambda_\rho)=e(E_\R)$ in cohomology, so
$$
\sqrt e\;(E)\=e(E_\R)\,\text{ in }\,H^{2n}\(Y,\textstyle{\Z\big[\frac12\big]}\).
$$
Finally if $E$ (rather than its pullback $\rho^*E$) admits a maximal isotropic $\Lambda\subset E$ then by \cite[Theorem 1(c)]{EG:Ch} we have
\beq{cL}
c_n(\Lambda)\=(-1)^{|\Lambda|}\sqrt e\;(E).
\eeq
As a consequence, given two oriented orthogonal bundles $(E_i,q_i,o_i)$ we can compute $\sqrt e\(E_1\oplus E_2)$ using the positive maximal isotropic $\Lambda_{\rho_1}\oplus\Lambda_{\rho_2}$ on the bundle $\wt Y:=\wt Y_1\times_Y\wt Y_2$ over $Y$. Here $\rho_i\colon\wt Y_i\to Y$ is the cover associated to $E_i$. Since $e\;(\Lambda_{\rho_1}\oplus\Lambda_{\rho_2})=e\;(\Lambda_{\rho_1})\;e\;(\Lambda_{\rho_2})$ on $\wt Y$ we find a Whitney sum formula
\beq{Whitney}
\sqrt e\;(E_1\oplus E_2)\=\sqrt e\;(E_1)\sqrt e\;(E_2).
\eeq

Given an isotropic subbundle $K\subset E$ of an orthogonal bundle $(E,q)$, a standard operation is to take its reduction
\beq{reduc}
K^\perp/K.
\eeq
Locally $E=K^\perp/K\oplus(K\oplus K^*)$, and we can give the orthogonal bundle $K\oplus K^*$ the standard orientation of \eqref{specify} so that $K\into K\oplus K^*$ is a positive maximal isotropic. Thus an orientation on $E$ then induces an orientation on $K^\perp/K$. Let $\rho\colon\wt Y\to Y$ be the bundle of isotropic flags for $K^\perp/K$, with positive maximal isotropic $\Lambda_\rho\subset K^\perp/K$. Then we get an induced positive maximal isotropic
\beq{17}
\Lambda\ :=\ \Lambda_\rho\times\_{K^\perp/K}K^\perp\ \subset\ \rho^*E
\eeq
sitting in an exact sequence $0\to K\to\Lambda\to\Lambda_\rho\to 0$. This gives $c_n(\Lambda)=c_k(K)\;c_{n-k}(\Lambda_\rho)$ where $k:=\rk K$, and so
\beq{Kred}
\sqrt e\;(E)\=\sqrt e\(K^\perp/K\)\;e\;(K).
\eeq

\subsection{Localisation by an isotropic section}\label{lociso}
Let $(E,q,o)$ be an oriented $SO(2n, \CC)$-bundle over $Y$ and suppose $s\in\Gamma(E)$ is an isotropic section: $q(s,s)=0$.
Let $i\colon Z(s)\into Y$ denote its zero scheme. In this section, we will construct a \emph{localised} square root Euler class --- an operator\footnote{In fact it can be constructed as a \emph{bivariant class} \cite[Chapter 17]{F} $\sqrt e\;(E,s)\in A^n\(i\colon Z(s)\to Y, \ZZ\)$ whose composition with $i_*$ gives the bivariant class $\sqrt e\;(E)\in A^n\(\id\colon Y\to Y, \ZZ\)$, but we will not strictly need this language.}
\begin{align*}
\sqrt e\;(E, s)\ \colon\, A_*\(Y,\ZZ\) \To  A_{*-n}\(Z(s),\ZZ\) 
\end{align*}
such that
$$
i_*\circ\sqrt e\;(E,s) \=  \sqrt e\;(E)\cap(\ \cdot\ ).
$$

\subsection*{Special case} To begin with we suppose that $E$ admits a maximal isotropic $\Lambda\subset E$. Here we can construct a localised class with integer coefficients. By \cite{KO} it coincides, over $\Q$ at least, with the one constructed by Polishchuk-Vaintrob \cite{PV:A} by different methods.

We have the exact sequence
\beq{LLs}
0\To\Lambda\To E\rt\pi\Ld\To0.
\eeq
By \eqref{cL} we know $(-1)^{n+|\Lambda|}\sqrt e\;(E)=(-1)^nc_n(\Lambda)=c_n(\Ld)$, which is represented by the Fulton-MacPherson intersection of the graph of $s^*:=\pi(s)\in\Gamma(\Ld)$ with the 0-section of $\Lambda^*$.

In turn this is described by first linearising the graph $\Gamma_{\!s\ast}\subset\Ld$ about the zero locus $Z^*$ of $s^*$, replacing it with the cone
$$
C_{Z^*/Y}\ \subset\ \Ld\big|_{Z^*}.
$$
This is the limit\footnote{As a subscheme of the total space $\Ld$. That is, we are taking the unique limit in the Hilbert scheme of $\Ld$.} of the graphs $\Gamma_{ts\ast}\subset\Ld$ as $t\to\infty$; see \cite[Remark 5.1.1]{F} for instance.

Intersecting $C_{Z^*/Y}$ with the zero section of $\Ld|_{Z^*}$ defines the Fulton-MacPherson localisation of $c_n(\Ld)$ to $Z^*$. But we would like to localise it further to $Z(s)\subset Z^*$ by using the ``\emph{other half}" of the section $s$.

Note that on restriction to $Z^*$ the section $s\in\Gamma(E)$ factors through $\Lambda|_{Z^*}$, thus defining a map of bundles over $Z^*$
\beq{function}
\wt s\ \colon\,\Ld\big|_{Z^*}\To\cO_{Z^*}
\eeq
with zero scheme $Z(s)\subset Z^*$. In the language of \cite{KL}, \eqref{function} is a \emph{cosection} of $\Ld$. By Lemma \ref{Lem:Zero} below, the fact that $s$ is isotropic forces $\wt s$ to be identically zero on $C_{Z^*/Y}\subset\Ld|_{Z^*}$. Therefore, by \cite[Proposition 1.3]{KL}, the intersection of $C_{Z^*/Y}$ with the zero section $0_{\Ld|_{Z^*}}$ can be \emph{cosection localised} to the zero locus $Z(s)$ of $\wt s$ by an operator
\beqa
0^{\;!,\,\mathrm{loc}}_{\Ld,\,\wt s}\ \colon\, A_*\(C_{Z^*/Y}\)\To A_{*-n}\(Z(s)\)
\eeqa
sitting in the following commutative diagram with the Fulton-MacPherson Gysin operator $0^{\;!}_{\Ld}$,
$$
\xymatrix@R=18pt@C=30pt{
A_*(Y) \ar[r]& A_*\(C_{Z^*/Y}\) \ar[r]^{0^{\;!,\,\mathrm{loc}}_{\Ld,\,\wt s}}\ar[d]& A_{*-n}\(Z(s)\) \ar[d]\ar[dr]^{i_*} \\
& A_*\(\Ld\big|_{Z^*}\) \ar[r]^{0^{\;!}_{\Ld}}& A_{*-n}(Z^*) \ar[r]& A_{*-n}(Y).\!}
$$
Here the vertical arrows, and the arrows in the triangle, are the obvious pushforward maps. Letting the first horizontal arrow be the specialisation map $W\mapsto C_{W\cap\; Z^*/W}$ of \cite[Proposition 5.2]{F}, the composition right-down-right-right $A_*(Y)\to A_{*-n}(Y)$ is cap product with $c_n(\Lambda^*)=(-1)^{n+|\Lambda|}\sqrt e\,(E)$ by \cite[Example 6.3.4]{F}.
Therefore denoting the horizontal composition across the top of the diagram by $(-1)^{n+|\Lambda|}\sqrt e\;(E,s,\Lambda)$ defines an operator
\beq{eLam}
\sqrt e\;(E,s,\Lambda)\ \colon\, A_*(Y)\To A_{*-n}\(Z(s)\)
\eeq
such that
\beq{I*}
i_*\circ \sqrt e\;(E,s,\Lambda)\=\sqrt e\;(E)\cap(\ \cdot\ ).
\eeq

\begin{Lemma} \label{Lem:Zero}
The cosection $\wt s$ of \eqref{function} is zero on $C_{Z^* / Y}\subset\Ld|_{Z^*}$.
\end{Lemma}

\begin{proof}
We may work locally, where \eqref{LLs} splits as in \eqref{efij} so that $E\cong\Lambda\oplus\Ld$ with the canonical quadratic form. Therefore $s=s\_\Lambda\oplus s^*$ also splits, and the isotropic condition $q(s,s)=0$ becomes
$$
\langle s^*,s\_\Lambda\rangle\=0.
$$
Multiplying by $t$ and considering $s\_\Lambda$ as a (fibrewise linear) function $\wt s\_\Lambda$ on the total space of $\Ld$, this says that
\beq{tiden}
\wt s\_\Lambda\big|_{\Gamma_{\!ts\ast}}\ \equiv\ 0.
\eeq
Since $\wt s\_\Lambda$ restricts on $\Ld|_{Z^*}$ to the function $\wt s$ of \eqref{function}, taking 
$\lim_{t\to\infty}$\;\eqref{tiden} gives $\wt s\;|\_{C_{Z^*/Y}}\equiv0$.
\end{proof}

\subsection*{General case} We can now define the localised square root Euler class in general by working on the cover $\rho\colon\wt Y\to Y$ \eqref{cover}. As before, to descend back to $Y$ via \eqref{pushy} we have to invert 2.

\begin{Def}\label{defdef} Given an isotropic section $s\in\Gamma(E)$ of an $SO(2n,\C$) bundle $(E,q,o)$ we define the localised square root Euler class
$$
\sqrt e\;(E,s)\ \colon\,A_*\(Y,\ZZ\)\To A_{*-n}\(Z(s),\ZZ\)
$$
by
\beq{rhoh}
\sqrt e\;(E,s)\ :=\ \frac1{2^{n-1}}\rho_*\(h\cup\sqrt e\;(\rho^*E,\rho^*s,\Lambda_\rho)\).
\eeq
\end{Def}

By \eqref{I*} and \eqref{pushy} this satisfies
\beq{push}
i_*\circ \sqrt e\;(E,s)\=\sqrt e\;(E)\cap(\ \cdot\ ).
\eeq

When $E$ admits a positive maximal isotropic $\Lambda\subset E$, the operators \eqref{eLam} and \eqref{rhoh} become equal on inverting 2,
\beq{EsEsL}
\sqrt e\;(E,s)\=\sqrt e\;(E,s,\Lambda).
\eeq
This can be shown by combining Lemmas \ref{newlem} and \ref{shr} below with the deformation of $\Gamma_{\!s}\subset E$ to $C_{Z(s)/Y}\subset E$ through $(\Gamma_{\!\;ts})_{\;t\in\C}$. Since we do not need it we omit the details, but see \cite[Theorem 5.2]{KP} for a complete proof.

%If $E$ admits a positive maximal isotropic $\Lambda\subset E$ we will later show in Proposition \ref{Younghoon} that Definition \ref{defdef} recovers the special case \eqref{sqrtdef} after inverting 2:
%$$
%\sqrt e\;(E,s)\=\sqrt e\;(E,s,\Lambda).
%$$

\subsection{Localisation by an isotropic cone}
\label{Sect:IsoCone}

We continue with an $SO(2n,\C)$ bundle $(E,q,o)$ over $Y$. Given an isotropic section we have described how to localise $\sqrt e\;(E)$ to $Z(s)$. In our application, $Z(s)$ will provide a local model for the moduli space of sheaves on a Calabi-Yau 4-fold. We will not be able to see $E,\,Y$ or $s$, but the obstruction theory on the moduli space will enable us to see the limiting data\footnote{Note this is a \emph{different} cone from the one $C_{Z^*/Y}$ considered in the previous Section.} $C_{Z(s)/Y}\subset E|_{Z(s)}$. So we would like to make sense of the idea that the cone $C_{Z(s)/Y}$ should be isotropic in $E|_{Z(s)}$, and recover the localised operator $\sqrt e\;(E,s)$ from this data alone.\medskip

So suppose given the data of a subscheme $Z\subset Y$ and a cone $C\subset E|_Z$ supported on $Z$. We call $C$ isotropic if $q$, thought of as a function on the total space of $E$ (quadratic on the fibres), vanishes on the subscheme $C$.

Letting $p\colon E\to Y$ be the projection map from the total space, and denoting the tautological section of $p^*E$ by
$$
\tau\_E\ \in\ \Gamma(E,p^*E),
$$
we see that $C$ is isotropic if and only if $\tau\_E|\_C$ is an isotropic section of $p^*E|_C$.

\begin{Def}\label{IsoCone}
For an isotropic cone $C\subset E|_Z$ we define the {\em square root Gysin map} by 
\beq{sqGy}
\sqrt{0_E^{\;!}}\ :=\ \sqrt e\(p^*E\big|_C\;,\,\tau\_E\big|_C\)\ \colon\, A_*\(C,\ZZ\) \To
A_{*-n} \(Z, \ZZ\), 
\eeq
noting that the support $Z$ of the cone is the zero locus of $\tau\_E\big|_C$.
\end{Def}

In the special case that the isotropic cone $C\subset E$ factors through a maximal isotropic subbundle $\Lambda\subset E$, this operator is familiar. It will follow from Lemma \ref{Schwantz} below that it is
$$
\sqrt{0_E^{\;!}}\=(-1)^{|\Lambda|}0_\Lambda^{\;!}\ \colon\,A_*\(C,\ZZ\) \To
A_{*-n} \(Z, \ZZ\).
$$
First we need a preliminary result giving an expression for $\surd\;0_E^{\;!}$ that does not use cosection localisation. It is a square-rooted version of \cite[Proposition 3.3]{F}. Using the zero section and projection
$$
\xymatrix@C=30pt{
Z\ \ar@{^(->}[r]<.4ex>^{0_C}& \,C, \ar@{->>}[l]<0.3ex>^{\pi}}
$$
let's pretend for a minute that $\pi$ is proper so that $\pi_*\;0_{C*}=\id$ and $0_{C*}$ does not lose any information. This is useful because after pushing forward, our cosection localised operator becomes the usual Edidin-Graham class, so
\beq{fant}
\sqrt{0_E^{\;!}}\,a\=\pi_*\;0_{C*}\(\sqrt e\(\pi^*E,\tau\_E\big|_C\)a\)\ 
\stackrel{\eqref{push}}=\ \pi_*\(\sqrt e\;(\pi^*E)\cap a\).
\eeq
We can turn this fantasy into reality by replacing $\pi\colon C\to Z$ by its projective completion\footnote{To define $\overline C=\PP(C\oplus\cO_Z)$ write $C=\Spec A\udot$ for some positively graded algebra $A\udot$. Take its graded tensor product with $B\udot:=\bigoplus_{i\ge0}\cO_Z$ and set $\overline C:=\Proj(A\udot\otimes B\udot)$.\label{fnC}} $\overline\pi\,\colon\overline C\to Z$ and extending the pair $(\pi^*E,\tau\_E|\_C)$ of an orthogonal bundle and an isotropic section from $C$ to $\overline C$.

To do this we give $\overline\pi^*(E\oplus\cO_Z\oplus\cO_Z)$ the quadratic form $\overline\pi^*q\oplus(xy)$, in the obvious notation. This makes $\cO_{\overline C}(-1)$ an isotropic subbundle orthogonal to $\{0\}\oplus\cO_Z\oplus\{0\}$. Thus on $\overline C$ we get the data of
\begin{enumerate}
\item[(i)] an orthogonal bundle $\overline E:=\cO(-1)^\perp/\cO(-1)$,
\item[(ii)] an isotropic section $\overline\tau$ of $\overline E$ given by the image of the section $(0,-1,0)$ of $\overline\pi^*(E\oplus\cO_Z\oplus\cO_Z)$, such that
\item[(iii)] the zero locus of $\overline\tau$ is the 0-section $0_{\overline C}\;$, and
\item[(iv)] on the cone $j\colon C\subset\overline C$ the pair $\(\;\overline E,\overline\tau\)$ restricts to $\(\pi^*E,\tau\_E|\_C\)$.
\end{enumerate}
This is enough to revive our argument to get a formula for $\surd\;0_E^{\;!}$ like \eqref{fant} in terms of a global (rather than cosection localised) Edidin-Graham class.

\begin{Lemma}\label{glbl} Using the above notation, given $a\in A_*\(C,\ZZ\)$ choose any $\overline a\in A_*\(\overline C,\ZZ\)$ such that $j^*\;\overline a=a$. Then
\beq{globale}
\sqrt{0_E^{\;!}}\,a\=\overline\pi_*\big[\sqrt e\;(\overline E)\cap\overline a\big].
\eeq
\end{Lemma}

\begin{proof}
Since the zero locus $0_{\overline C}$ of $\overline\tau$ lies in the image of $j$ it is immediate from the construction of the localised class that
$$
\sqrt e\(\;\overline E,\overline\tau\)(\overline a)\=\sqrt e\;(j^*\overline E,j^*\overline\tau)(j^*\overline a).
$$
Since this is $\sqrt e\(\pi^*E,\tau\_E|\_C\)(a)=\sqrt{0_E^{\;!}}\,a$ we obtain
\[
\sqrt{0_E^{\;!}}\,a\=
\overline\pi_*\;0_{\overline C*}\sqrt e\(\overline E,\overline\tau\)(\overline a)\ 
\stackrel{\eqref{push}}=\ \overline\pi_*\big[\sqrt e\;(\overline E)\cap\overline a\big].\qedhere
\]
\end{proof}

Our first application is the following.

\begin{Lemma}\label{Schwantz} Suppose $\Lambda\subset E$ is a maximal isotropic subbundle. Thinking of it as an isotropic cone supported on $Y$, we have
\beq{Kevin34}
\sqrt{0_E^{\;!}}\=(-1)^{|\Lambda|}\;0_\Lambda^{\;!}\ \colon\,A_*\(\Lambda,\ZZ\) \To
A_{*-n} \(Y, \ZZ\).
\eeq
\end{Lemma}

\begin{proof}
We suppress the $\ZZ$ coefficients throughout. Applying the construction above to the isotropic cone $C=\Lambda\subset E$ gives us the data (i)-(iv) and a diagram of maps
$$
\xymatrix@=20pt{
\Lambda\ \ar[dr]<0.5ex>^\pi\ar@{^(->}[rr]^j&& \,\overline\Lambda \ar[dl]<-0.5ex>_{\overline\pi} \\
& Y.\! \ar@{^(->}[ul]<0.5ex>^{0\_\Lambda\!\!}\ar@{_(->}[ur]<-0.5ex>_{0_{\;\overline\Lambda}}}
$$
Moreover in this situation we can construct a maximal isotropic subbundle of $\overline E$ over $\overline\Lambda$ that restricts on $\Lambda\subset\overline\Lambda$ to $\pi^*\Lambda\subset\pi^*E$. We start with the maximal isotropic subbundle
$$
\overline\pi^*(\Lambda\oplus\cO_Y)\ \Into\ \overline\pi^*E\oplus\cO_{\overline\Lambda}\oplus\cO_{\overline\Lambda\,}.
$$
The tautological line bundle $\cO(-1):=\cO_{\PP(\Lambda\oplus\cO_Y)}(-1)$ 
is a line subbundle of both, and $\overline\pi^*(\Lambda\oplus\cO_Y)$ lies in its orthogonal $\cO(-1)^\perp$ in $\overline\pi^*E\oplus\cO_{\overline\Lambda}\oplus\cO_{\overline\Lambda}$. So dividing both by $\cO(-1)$ gives, by the relative Euler sequence, the following maximal isotropic subbundle of the orthogonal bundle $\overline E$,
\beq{oyluah}
T_{\overline\pi}(-1)\ \Into\ \overline E\=\cO(-1)^\perp\big/\cO(-1)\ \text{ on }\ \overline\Lambda.
\eeq
The section $(0,-1)$ of $\overline\pi^*(\Lambda\oplus\cO_Y)$ projects, under quotienting by $\cO(-1)$, to a  section $\overline\tau$ of $T_{\overline\pi}(-1)$. Its image in $\overline E$ is the isotropic section $\overline\tau$ of (ii).
It restricts over $\Lambda\subset\overline\Lambda$ to the Euler vector field on $\Lambda\to Y$, i.e. the tautological section $\tau\_\Lambda$ of $\pi^*\Lambda\cong T_\pi$. 

Furthermore $\overline\tau$ cuts out the zero section $0_{\overline\Lambda}\subset\overline\Lambda$, and transversally when thought of as a section of $T_{\overline\pi}(-1)$. Thus
\beq{oiler}
0\_{\overline\Lambda*}0_{\overline\Lambda}^{\;!}\=e\(T_{\overline\pi}(-1)\)\cap(\ \cdot\ ).
\eeq
Therefore, choosing any class $\overline a\in A_*(\overline\Lambda)$ such that $j^*\overline a=a$,
\begin{eqnarray*}
\sqrt{0_E^{\;!}}\,a &\stackrel{\eqref{globale}}=&
\overline\pi_*\big[\sqrt e\;(\overline E)\cap\overline a\big]\ \stackrel{\eqref{cL}}=\ \pm\;\overline\pi_*\(e\(T_{\overline\pi}(-1)\)\cap\overline a\) \\ \nonumber &\stackrel{\eqref{oiler}}=& \pm\;\overline\pi_*\;0\_{\overline\Lambda*}0_{\overline\Lambda}^{\;!}\,\overline a\ =\ \pm\;0_{\overline\Lambda}^{\;!}\,\overline a
\=\pm\;0_{j^*\overline\Lambda}^{\;!}(j^*\overline a)=\pm\;0_{\Lambda}^{\;!}a,
\end{eqnarray*}
where $\pm$ is the sign $(-1)^{|\Lambda|}$.
%To motivate the proof, pretend for a minute that $\pi$ is proper. Then $\pi_*\;0_{\Lambda*}=\id$, so applying $0_{\Lambda*}$ loses no information. This is useful because after pushing forward, our cosection localised operator becomes the usual Edidin-Graham class. So for any $a\in A_*(\Lambda)$ we could conclude
%$$
%\sqrt{0_E^{\;!}}\,a\=\pi_*\;0_{\Lambda*}\sqrt e\(\pi^*E,\tau\_E\big|_\Lambda\)\;a\ \stackrel{\eqref{push}}=\ 
%\pi_*\sqrt e\;(\pi^*E)\;a\ \stackrel{\eqref{cL}}=\ \pm\;\pi_*\(e(\Lambda)\cap a)\).
%$$
%Since the tautological section $\tau\_\Lambda=\tau\_E\big|_\Lambda$ of $\pi^*\Lambda$ cuts out the zero section $0_\Lambda\subset\Lambda$ we know $e(\Lambda)\cap a=0_{\Lambda*}0_{\Lambda}^{\;!}a$, giving the desired result
%$$
%\sqrt{0_E^{\;!}}\,a\=\pm\;\pi_*\;0_{\Lambda*}0_{\Lambda}^{\;!}a\=\pm\;0_{\Lambda}^{\;!}a.
%$$
%This pretend proof becomes an actual proof once we compactify $\Lambda$ inside its projective completion $\overline\Lambda:=\PP(\Lambda\oplus\cO_Y)$, and produce extensions over $\overline\Lambda$ of
%\begin{enumerate}
%\item the orthogonal bundle $\pi^*E$,
%\item the maximal isotropic subbundle $\pi^*\Lambda$, and
%\item the section $\tau\_\Lambda$ transversely cutting out the 0-section $0_\Lambda$.
%\end{enumerate}
%
%We use the self-explanatory notation
%$$
%\xymatrix@=20pt{
%\Lambda\ \ar[dr]<0.5ex>^\pi\ar@{^(->}[rr]^j&& \,\overline\Lambda \ar[dl]<-0.5ex>_{\overline\pi} \\
%& Y.\! \ar@{^(->}[ul]<0.5ex>^{0\_\Lambda\!\!}\ar@{_(->}[ur]<-0.5ex>_{0_{\;\overline\Lambda}}}
%$$
\end{proof}

Recall our definition \eqref{sqGy} $\surd\;{0_E^{\;!}}:=\sqrt e\(p^*E\big|_C\;,\,\tau\_E\big|_C\)$. If $E$ admits a maximal isotropic subbundle $\Lambda\subset E$ there is an obvious alternative definition using the operator $\sqrt e\(p^*E|\_C\;,\,\tau\_E|\_C,p^*\Lambda|\_C\)$ of \eqref{eLam}. Our second application of Lemma \ref{glbl} is to show that they're the same.

\begin{Lemma}\label{newlem} Suppose $E$ admits a maximal isotropic subbundle $\Lambda\subset E$ and an isotropic cone $C\subset E$. Let $\pi\colon C\to Z\subseteq Y$ denote the projection. Then
$$
\sqrt{0_E^{\;!}}\=\sqrt e\(\pi^*E,\,\tau\_E|\_C\;,\pi^*\Lambda\)\ \colon\, A_*\(C,\ZZ\) \To A_{*-n} \(Z, \ZZ\).
$$
\end{Lemma}

\begin{proof}
Lemma \ref{glbl} rests on extending $(\pi^*E,\tau\_E|\_C)$ from $C$ to its projective completion $\overline C$, giving the orthogonal bundle and isotropic section $(\overline{E},\overline\tau)$. There may be no extension $\wt\Lambda\subset\overline E$ of the  maximal isotropic $\pi^*\Lambda\subset\pi^*E$, but there is one if we replace $\overline C$ by a certain blow up (away from $C$).

Namely, $\pi^*\Lambda\subset\pi^*E$ defines a section of the orthogonal Grassmannian bundle $\operatorname{OGr}(\overline E)\to\overline C$ over $C\subset\overline C$. Taking its closure defines the blow up $b\colon\wt C\to\overline C$ on which the universal subbundle  on $\operatorname{OGr}(\overline E)$ restricts to give the maximal isotropic $\wt\Lambda\subset b^*\overline E$.

So now replacing $C\stackrel j\Into\overline C\rt{\overline\pi}Z$ by $C\stackrel\jmath\Into\wt C\rt{\!\wt\pi:=b\circ\overline\pi\!}Z$, the zero section $0_{\overline C}$ by $0_{\wt C}$, the orthogonal bundle and isotropic section $(\overline{E},\overline\tau)$ by $(\wt{E},\wt\tau):=(b^*\overline{E},b^*\overline\tau)$,
and $\overline a$ by any $\wt a$ such that $\jmath^*\;\wt a=a$, the same proof gives
$$
\sqrt e\(\pi^*E,\,\tau\_E|\_C\)(a)\=\wt\pi_*\;0_{\wt C*}\big[\sqrt e\,(\wt E,\wt\tau)(\wt a)\big]\ 
\stackrel{\eqref{push}}=\ \wt\pi_*\big[\sqrt e\;(\wt E)\cap\wt a\big].
$$
But replacing \eqref{push} by \eqref{I*} we similarly get
$$
\sqrt e\(\pi^*E,\,\tau\_E|\_C,\pi^*\Lambda\)(a)\=\wt\pi_*\;0_{\wt C*}\big[\sqrt e\,(\wt E,\wt\tau,\wt\Lambda)(\wt a)\big]\ 
\stackrel{\eqref{I*}}=\ \wt\pi_*\big[\sqrt e\;(\wt E)\cap\wt a\big].
$$
Thus the left hand sides of these two expressions are equal.
\end{proof}

Our third application of Lemma \ref{glbl} is to derive a Whitney sum formula for the localised operators $\surd\;0_E^{\;!}$ by using the global Whitney sum formula \eqref{Whitney} for Edidin-Graham classes. From now on we work over all of $Y$ instead of specialising to the support $Z\subseteq Y$ of a cone, i.e. we set $Z=Y$.

Suppose $C_1,\,C_2$ are isotropic cones in $SO(2n_i,\C)$ bundles $E_1,\,E_2$ and set
$$
C\ :=\ C_1\oplus C_2\ \subset\ E_1\oplus E_2\ =:\ E.
$$
Let $p_i\colon E_i\to Y$ denote the projections from the total spaces. The total space of $E\to Y$ is the total space of $p_1^*E_2\to E_1$. So we may think of the isotropic cone $C\subset E$ as lying in $p_1^*E_2|_{C_1}$, with 0-section $C_1$, and form the composition
\beq{whit}
A_*\(C,\ZZ\)\rt{\surd\;0_{p\ast\!\!\!\_1E_2}^{\;!}}A_{*-n_2}\(C_1,\ZZ\)
\rt{\surd\;0_{E_1}^{\;!}}A_{*-n_1-n_2}\(Y,\ZZ\).
\eeq

\begin{Prop}\label{E1E2}
%Suppose $C_i\subset E_i$ are isotropic cones in $SO(2n_i,\C)$ bundles, for $i=1,2$. Setting $C:=C_1\oplus C_2\subset E_1\oplus E_2=:E$, we have
The composition \eqref{whit} is $\surd\;0_E^{\;!},$
$$
\sqrt{0_{E_1}^{\;!}}\circ\sqrt{0_{p_1^*E_2}^{\;!}}\=\sqrt{0_E^{\;!}}\ \colon\,A_*\(C,\ZZ\)\To A_{*-n_1-n_2}\(Y,\ZZ\).
$$
\end{Prop}

\begin{proof}
Associated to the cones $C_i\subset E_i$ we have the data $(\overline C_i,\overline E_i,\overline \tau_i,j_i)$ from (i)-(iv) above. Thus we can complete $C=C_1\oplus C_2$ with $\overline C_1\times_Y\overline C_2$ instead of $\overline C$. On this we have the orthogonal bundle\footnote{We suppress some obvious pullback maps for the sake of clarity. The functoriality $\sqrt e\;(p^*F)=p^*\sqrt e\;(F)$ of the Edidin-Graham class means this adds no hidden dangers.} $\overline E_1\oplus \overline E_2$ and isotropic section $\overline\tau:=(\overline\tau_1,\overline\tau_2)$ with zero locus the 0-section $0_{\overline C_1\times_Y\overline C_2}$. Applied to these, the argument of Lemma \ref{glbl} then gives
$$
\sqrt{0_E^{\;!}}\,a\=(\overline\pi_1\times\_Y\overline\pi_2)_*\big[\sqrt e\;(\overline E_1\oplus\overline E_2)\cap\overline a\big]
$$
for any $\overline a\in A_*(\overline C_1\times_Y\overline C_2)$ whose pullback $(j_1\times_Yj_2)^*\,\overline a$ to $C_1\oplus C_2$ is $a$.
By \eqref{Whitney} and the projection formula this is
$$
\overline\pi_{1*}\left[\sqrt e\;(\overline E_1)\cap\overline\pi_{2*}\(\sqrt e\;(\overline E_1)\cap\overline a\)\right]\ \stackrel{\eqref{globale}}=\ \overline\pi_{1*}\left[\sqrt e\;(\overline E_1)\cap\sqrt{0^{\;!}_{\overline\pi_1^*E_2}}\,\wt a\right],
$$
where $\wt a$ is the restriction of $\overline a$ to $\overline C_1\times_Y C_2$. Finally applying \eqref{globale} to $\overline\pi_{1*}$ gives
\begin{equation*}
\sqrt{0_E^{\;!}}\,a\=\sqrt{0^{\;!}_{E_1}}\left(j_1^*\sqrt{0^{\;!}_{\overline\pi_1^*E_2}}\,\wt a\right)
\=\sqrt{0^{\;!}_{E_1}}\sqrt{0^{\;!}_{p_1^*E_2}}\,a.\qedhere
\end{equation*}
\end{proof}

A closely variant of this result is the following. Suppose we have\footnote{If $C$ is an isotropic \emph{subbundle} then $K\subset C$ implies $C\subset K^\perp$, but for cones this need not be true, e.g. if $K:=\{x=0\}$ in $C:=\{xy=0\}$ inside $\C^2$ with quadratic form $xy$.}
$$
K\,\subset\,C\,\subset\,K^\perp\,\subset\,(E,q,o)
$$
with $K$ a rank $k$ isotropic subbundle of a rank $2n$ oriented orthogonal bundle $(E,q,o)$ and $C$  an isotropic cone. Then $C$ descends to an isotropic cone
$$
C/K\ \subset\ K^\perp/K
$$
in the reduction \eqref{reduc} of $E$ by $K$. Using the orientation described after \eqref{reduc}, the formula \eqref{Kred} for the Edidin-Graham class of $K^\perp/K$ then gives the following. Note the projection $p\colon C\to C/K$ is flat, so $p^*$ is defined on cycles.

\begin{Prop}\label{KperpK} The square root Gysin classes of $E$ and its reduction $K^\perp/K$ are related by
$$
\sqrt{0_{K^\perp/K}^{\;!}}\=\sqrt{0_E^{\;!}}\circ p^*\ \colon\,A_*\(C/K,\,\ZZ\)\To A_{*+k-n}\(Y,\,\ZZ\).
$$
\end{Prop}

\begin{proof}
This result can be proved by deforming the pair ($C\subset E$) through (isotropic cones in orthogonal bundles) to $C/K\oplus K\subset(K^\perp/K)\oplus(K\oplus K^*)$. Then we can apply  Proposition \ref{E1E2}, which amounts to working in the compactification $\overline{C/K}\times\_Y\overline K$ of $C/K\oplus K$.

Alternatively we can generalise the proof of Proposition \ref{E1E2} by working in the right compactification of the undeformed $C$ from the beginning. It is locally isomorphic to $\overline{C/K}\times\_Y\overline K$ but is globally twisted.

We start with the projective completion $\overline\pi\colon\overline{C/K}\to Y$ of Footnote \ref{fnC} with the data (i)-(iv) over it. In particular we have the isotropic line subbundle $\cO_{\overline{C/K}}(-1)\into\overline\pi^*(K^\perp/K\oplus\cO\oplus\cO)$ and\vspace{-1mm} the orthogonal bundle $\overline{K^\perp/K}:=\cO_{\overline{C/K}}(-1)^\perp/\cO_{\overline{C/K}}(-1)$ extending $K^\perp/K$ over $C/K\subset\overline{C/K}$.
Via the surjection $K^\perp\onto K^\perp/K$ we get the Cartesian diagram
\beq{EKbar}
\xymatrix@=16pt{
q^*\;\cO_{\overline{C/K}}(-1)\ \ar@{^(->}[r]\ar@{->>}[d]& \ \overline\pi^*\(K^\perp\oplus\cO\)\ \ar@{->>}[d]^q\ar@{^(->}[r]& \ \overline\pi^*(E\oplus\cO\oplus\cO) \\
\cO_{\overline{C/K}}(-1) \ \ar@{^(->}[r]& \ \overline\pi^*\(K^\perp/K\oplus\cO\).}
\eeq
Our compactification of $C$ is the projective bundle\vspace{-1mm} $\overline C:=\PP(q^*\;\cO_{\overline{C/K}}(-1))$ over $\overline{C/K}$. The inclusion $C\subset\overline C$ takes $c\in C$ to the point $[c:1]$ in the projectivisation of $K^\perp\oplus\;\cO$ sat over the point $[${\footnotesize{[}}$c\;${\footnotesize{]}}$\,:1]\in\PP(C/K\oplus\;\cO_Y)=\overline{C/K}$. We use the projections
$$
\xymatrix{
\overline C \ar[r]^-{\overline p}& \overline{C/K} \ar[r]^{\overline\pi}& Y. \ar@{<-}@/^.8pc/[ll]+<5pt,-6pt>^-{\Pi\,=\,\overline\pi\;\circ\;\overline p}}
$$
As the projectivisation of a vector bundle, $\overline C:=\PP(q^*\;\cO_{\overline{C/K}}(-1))$ carries a tautological line bundle
\beqa 
\quad\cO_{\overline C}(-1)\Into\overline p^*\(q^*\;\cO_{\overline{C/K}}(-1)\)\Into\Pi^*\(K^\perp\oplus\cO\)\Into\Pi^*(E\oplus\cO\oplus\cO),
\eeqa
where we have applied $\overline p^*$ to \eqref{EKbar}. Since $\cO_{\overline C}(-1)$ is isotropic in $\Pi^*(E\oplus\cO\oplus\cO)$ we may define the orthogonal bundle
$$
\overline E\ :=\ \cO_{\overline C}(-1)^\perp\big/\;\cO_{\overline C}(-1)\ \text{ over }\,\overline C.
$$
The image of the section $(0,-1,0)$ of $\cO_{\overline C}(-1)^\perp\subset\Pi^*(E\oplus\cO\oplus\cO)$ is an isotropic section $\overline\tau$ of $\overline E$ which cuts out precisely the zero section $Y\into C\into\overline C$. Therefore \eqref{globale} applies again, giving, for $a\in A_*(C)$,
\beq{abc}
\sqrt{0^{\;!}_E}\,a\=\Pi_*\big[\sqrt e\(\;\overline E\;\)\cap\overline a\big],
\eeq
for any $\overline a\in A_*(\overline C)$ which restricts to $a$ on $C\subset\overline C$.
We can also define
$$
\overline K\ :=\ \frac{\overline p^*\(q^*\;\cO_{\overline{C/K}}(-1)\)}{\cO_{\overline C}(-1)}\ \subset\ \frac{\cO_{\overline C}(-1)^\perp}{\cO_{\overline C}(-1)}\=\overline E
$$
and check that over $C\subset\overline C$ this gives $K\subset E$. By elementary linear algebra
\beq{Kbarbar}
\overline K^\perp\!\big/\,\overline K\ \cong\ \overline p^*\,\overline{(K^\perp/K)},
\eeq
since being orthogonal to $q^*\;\cO_{\overline{C/K}}(-1)$ is equivalent to being orthogonal to $K$ and $\cO_{\overline{C/K}}(-1)$. Moreover, $\overline C/\overline K=\overline{C/K}$.

So now suppose $a=p^*b$ for $b\in A_*(C/K)$. Choose $\overline b\in A_*\(\overline{C/K}\)$ restricting to $b$ and take $\overline a=\overline p^*\overline b$. Then by \eqref{Kred} and \eqref{Kbarbar}, equation \eqref{abc} becomes
$$
\sqrt{0^{\;!}_E}\,p^*b\=\overline\pi_*\overline p_*\big[\overline p^*\sqrt e\(\overline{K^\perp/K}\)\;e\(\;\overline K\;\)\cap\overline p^*\overline b\,\big].
$$
Now $\overline p_*e\(\;\overline K\;\)$ is a constant which can be computed on any fibre of $\overline p$. Since each fibre is a projective space $\PP^{\;k}$ on which $\overline K$ is $T_{\PP^{\;k}}(-1)$ (cf. \eqref{oyluah}) we find the constant is 1. Therefore, by the projection formula,
\[
\sqrt{0^{\;!}_E}\,p^*b\=\overline\pi_*\big[\sqrt e\(\overline{K^\perp/K}\)\cap\overline b\,\big]\ \stackrel{\eqref{globale}}=\ \sqrt{0_{K^\perp/K}^{\;!}}\,b. \qedhere
\]
\end{proof}

Our final application of Lemma \ref{glbl} is to check that the square root Gysin operator $\surd\;0_E^{\;!}$ commutes with refined Gysin operators $f^{\;!}$. Suppose we have a Cartesian diagram with $f$ a regular embedding of codimension $d$,
$$
\xymatrix@R=16pt{
f'^*C \ar[r]^{f''}\ar[d]_{\pi'} & C \ar[d]^\pi \\
X' \ar[d]\ar[r]^{f'}& Y' \ar[d] \\
X \ar[r]^f& Y,\!}
$$
where $C$ is an isotropic cone in an $SO(2n,\C)$ bundle $(E,q,o)$ over $Y'$. Given $a\in A_*(C)$ we get the cycle $f^{\;!}a$ in the isotropic cone $f'^*C\subset f'^*E$.

\begin{Lemma}\label{shr}
In the above notation we have
\beq{Eshriek}
f^{\;!}\sqrt{0_E^{\;!}}\=\sqrt{0_{f'^*E}^{\;!}}\ f^{\;!}\ \colon\,A_*\(C,\ZZ\)\To A_{*-d-n}\(X',\ZZ\).
\eeq
\end{Lemma}

\begin{proof}
Using the notation of Lemma \ref{glbl} we have projective completions $\overline\pi\colon\overline C\to Y'$ and  $\overline\pi'\colon\overline{f'^*C}\to X'$ of the cones $C,\,f'^*C$, with a map $\overline{f''}$ between them covering $f''$. By \eqref{globale},
$$
\sqrt{0_E^{\;!}}\,a\=\overline\pi_*\big[\sqrt e\;(\overline E)\cap\overline a\big],
$$
where as usual $\overline a$ is any extension of $a$ to $A_*(\overline C)$.

Since we are using $\ZZ$ coefficients we may replace $\overline C$ by the bundle $\rho$ \eqref{cover} over it (and replace $\overline{f'^*C}$ by the corresponding flat basechange) to assume that $\overline E$ admits a positive maximal isotropic $\Lambda$. Thus by \eqref{cL},
$$
\sqrt{0_E^{\;!}}\,a\=\overline\pi_*\big[e(\Lambda)\cap\overline a\big].
$$
By the usual commutativity of Chern classes and refined Gysin maps \cite[Proposition 6.3]{F},
$$
f^{\;!}\big[e(\Lambda)\cap\overline a\big]\=
e\(\;\overline{f''}^*\Lambda\)\cap f^{\;!}\;\overline a.
$$
Now pushing down by $\overline\pi'_*$ and using $\overline\pi'_*f^{\;!}=f^{\;!}\overline\pi_*$ \cite[Theorem 6.2(a)]{F} gives
$$
f^{\;!}\sqrt{0_E^{\;!}}\,a\=\overline\pi'_*\big[e\(\;\overline{f''}^*\Lambda\)\cap f^{\;!}\;\overline a\big],
$$
which by another application of \eqref{cL} and \eqref{globale} is precisely $\sqrt{0_{f'^*E}^{\;!}}\ f^{\;!}a$.
\end{proof}

\section{Moduli of sheaves on CY$^4$ via orthogonal bundles}
\label{CY4M}

Let $(X,\cO_X(1))$ be a smooth projective 4-fold with a fixed trivialisation of $K_X$.
%$$
%H^i(\cO_X)\=0\ \text{ for }\ 1\le i\le3.
%$$
Fix a class $c\in H^*(X,\Q)$ such that $\cO_X(1)$-Gieseker semistable sheaves with  Chern character $c$ are all stable. Then there is a \emph{projective} moduli space $M=M(X,c)$ of stable sheaves $F$ of charge $c$. We let $\LL_M\in D(M)$ denote the truncated cotangent complex in the bounded derived category of coherent sheaves on $M$. We use $\pi$ to denote \emph{any} projection $X\times N\to N$ down $X$.

\subsection{Obstruction theory} \label{Sect:realization}
Let $\cE$ be any universal twisted sheaf on $X \times M$, whose existence is proved in \cite[Propositions 3.3.2, 3.3.4]{Ca}. The twistings cancel in $R\hom(\cE,\cE)$, giving a complex of sheaves in $D(M\times X)$. By \cite[Theorem 4.1]{HT} the truncated Atiyah class of \cite[Equation 4.2]{HT} defines an obstruction theory for $M$,
\beq{ObsTh}
\EE\ :=\ \tau^{[-2,0]}\(R\pi_*\;R\hom(\cE,\cE)[3]\)\rt\At\LL_M,
\eeq
in the sense of \cite[Definition 4.4]{BF}. That is, $h^0(\At)\colon h^0(\EE) \to \Omega_M$ is an isomorphism, and $h^{-1}(\At) \colon h^{-1}(\EE) \onto h^{-1}(\mathbb{L}_{M})$ is a surjection. The obstruction theory  \eqref{ObsTh} has virtual dimension
$$
\vd\ :=\ \rk\EE\=2-\chi(F,F),
$$
for any sheaf $F$ on $X$ of the same Chern character $c$. Relative Serre duality down the map $\pi$ gives an isomorphism
\beq{duel}
\theta\,\colon\,\EE \rt\sim \EE^\vee[2] \,\text{ such that }\, \theta\=\theta^\vee[2]\ \in\ \Hom_{D(M)\!}\(\EE,\EE^\vee[2]\).
\eeq
Since $\EE$ is perfect\footnote{$R\pi_*\;R\hom(\cE,\cE)$ is perfect. From it remove $R^4\pi_*\;\cO_{X\times M}[-4]=\cO_M[-4]$ via trace and $R^0\pi_*\;\cO_{X\times M}=\cO_M$ via the identity map to give $\EE[-3]$, which is therefore also perfect.} of amplitude $[-2,0]$, rather than $[-1,0]$, the obstruction theory $\At$ is not \emph{perfect}. (In general $\Ext^3(F,F)=\Ext^1(F,F)^*$ is nonzero for $F\in M$, so $h^{-2}(\EE)$ can be nonzero.) So we cannot apply \cite{BF, LT} to get a virtual cycle. Instead we will follow the Behrend-Fantechi recipe as far as producing a cone in a vector bundle --- in fact an \emph{isotropic} cone in an \emph{oriented orthogonal} bundle --- then we will replace their intersection with the zero section $0^{\;!}_E$ by the square-rooted analogue $\surd\;0_E^{\;!}$ of Definition \ref{IsoCone}.

To show the cone \eqref{iscone} is isotropic, we do not know if it is sufficient to have the duality \eqref{duel}. Instead our proof will use a lifting of this symmetry to a $(-2)$-shifted symplectic structure \cite{PTVV}, whereupon we can employ the results of \cite{BBBJ, BBJ, BG}. (Contraction with the shifted symplectic form then induces a shifted duality $\LL_M^{\vir}\cong(\LL_M^{\vir})^\vee[2]$ on the virtual cotangent bundle $\LL_M^{\vir}:=\EE$, recovering \eqref{duel}.) For this
we will need an alternative description of the obstruction theory \eqref{ObsTh} using the wonders of derived stacks.

\subsection*{Derived description}
Let $\cM^{\mathrm{der}}$ denote the derived stack \cite{TVa} of stable sheaves of charge $c$, with underlying Artin stack $\cM$ and coarse moduli space $M$. Since stable sheaves are simple the projection $\cM\to M$ is a $B\C^*$-bundle.
Let $\curly E$ denote the universal sheaf on $X\times\cM^{\mathrm{der}}$. The (derived) cotangent bundle of $\cM^{\mathrm{der}}$ is
$$
\LL_{\cM^{\mathrm{der}}}\=\(R\pi_*\;R\hom(\curly E,\curly E)[1]\)^\vee\ \cong\ R\pi_*\;R\hom(\curly E,\curly E)[3],
$$
by relative Serre duality and the fixed trivialisation of the dualising sheaf $K_X$ of $\pi$. Restricting to $\cM\subset\cM^{\mathrm{der}}$ and truncating gives the composition
$$
\tau^{[-2,0]}\(R\pi_*\;R\hom(\curly E,\curly E)[3]\)\To\tau^{[-2,0]}\LL_{\cM}\To\tau^{[-1,0]}\LL_{\cM}.
$$
The latter is the pullback of $\LL_M$ from $M$,\footnote{Since $\cE$ and $\curly E$ differ locally by a line bundle, their derived endomorphisms coincide.} and by \cite[Appendix A]{STV} the result is the pullback of the map $\At\colon\EE\to\LL_M$ of \eqref{ObsTh}.

\subsection*{Normal form for $\EE$}
We begin by getting the virtual cotangent bundle $\EE$ into a normal form. Call a 3-term complex of locally free sheaves $E\udot$ \emph{self-dual} if it has the form
\beq{Kdot}
E\udot\ :=\ \big\{T\rt aE\rt{a^*}T^*\big\},
\eeq
where $(E,q)$ is an orthogonal bundle, inducing the isomorphism $E\cong E^*$ used in forming the map $a^*$ above. Such a complex has an obvious duality
\beq{dueldot}
\hspace{5mm}\xymatrix@R=2pt{
E\udot\! \ar[dd]&& T \ar@{=}[dd]\ar[r]^a& E \ar[dd]^(.4){q}_(.4)\wr\ar[r]^-{a^*}& T^* \ar@{=}[dd]<-.4ex> \\  &= \\
\hspace{-9mm}E_\bullet[2]\,:=\,(E\udot)^\vee[2] && (T^*)^* \ar[r]^-{(a^*)^*}& E^*\! \ar[r]^{a^*}& T^*.\!\!}
\eeq
We use the standard notation for complexes $E\udot$ that $E^i$ appears in degree $i$, with $E_{-i}:=(E^i)^*$ in degree $-i$ in the dual complex $E_\bullet$. Given a map of complexes $\psi^\bullet\colon E\udot\to F\udot$ we denote its dual by $\psi_\bullet:=(\psi\udot)^\vee\colon F_\bullet\to E_\bullet$.

We will show there is a quasi-isomorphism from $\EE$ to a self-dual complex $E\udot$ intertwining the Serre duality map $\theta\colon\EE^\vee[2]\to\EE$ of \eqref{duel} with the duality \eqref{dueldot}. To state this we are careful to distinguish between morphisms in $D(M)$ and genuine maps of complexes. We denote the former by single letters such as $\alpha$, and the latter by $\alpha\udot$.

\begin{Prop}\label{form} There is a self-dual 3-term complex of locally free sheaves $E\udot$ \eqref{Kdot} and an isomorphism $\alpha\colon E\udot\to\EE$ in $D(M)$, such that the following diagram commutes in $D(M)$,
\beqa 
\xymatrix@R=18pt@C=50pt{
E\udot\! \ar[d]^-\wr_-{\alpha}\ar[r]^-{\eqref{dueldot}} & E_\bullet[2] \\
\EE \ar[r]^-\theta & \EE^\vee[2].\!\! \ar[u]^-\wr_{\alpha^\vee[2]}}
\eeqa
Furthermore, given an embedding $M\subset A$ in a smooth scheme with ideal $I$, we may assume that $\At\circ\,\alpha\colon E\udot\to\LL_M$ is represented by a genuine map of complexes $E\udot\to\{I/I^2\to\Omega_A|_M\}$ which is surjective in each degree.
%Given two pairs $(E_i\udot,\alpha_i)$ satisfying the above conditions, there is a third $(E_3\udot,\alpha_3)$ and maps of complexes $\beta_i\udot\colon E_i\udot\to E_3\udot$ such that
%$$
%\xymatrix@R=5pt@C=30pt{& \EE \ar[dl]_{\alpha_1}\ar[dr]^{\alpha_2}\ar[dd]^{\alpha_3}
%\\ E_1\udot \ar[dr]_{\beta\udot_1}&& E_2\udot \ar[dl]^{\beta\udot_2} \\
%& E_3\udot}
%$$
%commutes in $D(M)$.
\end{Prop}

\begin{proof} Choose a locally free resolution $\phi\colon A\udot\xrightarrow{\,\sim\,}\EE$ in $D(M)$. Then $\theta\colon\EE\to\EE^\vee[2]$ becomes the morphism $\phi^\vee[2]\circ\;\theta\circ\;\phi\;\colon A\udot\to A_\bullet[2]$ in $D(M)$.

This need not be a genuine map of complexes, however. So we now further resolve $A\udot$ by a map of complexes $\psi\udot\colon B\udot\to A\udot$, where $B\udot$ is a \emph{sufficiently negative} locally free resolution that there exist \emph{maps of complexes} $f\udot,\,g\udot$ filling in commutative diagrams in $D(M)$,
$$
\xymatrix@R=18pt@C=55pt{
B\udot \ar[d]^-\wr_-{\psi\udot}\ar@{-->}[dr]^-{f\udot} &&
B\udot \ar[d]_{\phi\circ\psi\udot}\ar@{-->}[dr]^{g\udot} \\
A\udot \ar[d]_{\phi}^\wr\ar[r]_-{\phi^\vee[2]\circ\;\theta\circ\;\phi} & A_\bullet[2] & \EE \ar[r]_-\At& L\udot_M, \\
\EE \ar[r]_-\theta& \EE^\vee[2] \ar[u]^\wr_{\phi^\vee[2]}}
$$
where $L\udot_M$ is defined to be the complex $\{I/I^2\rt{d}\Omega_A|_M\}\cong\LL_M$ made from the fixed embedding $M\subset A$ with ideal $I$.
Composing then gives another commutative diagram in $D(M)$,
\beq{commdg}
\xymatrix@R=18pt@C=50pt{
B\udot \ar[d]^-\wr_-{\phi\circ\psi\udot}\ar[r]^-{\psi_\bullet[2]\circ f\udot} & B_\bullet[2]\! \\
\EE \ar[r]^-\theta & \EE^\vee[2],\!\! \ar[u]^\wr_-{\psi_\bullet\circ\phi^\vee[2]}}
\eeq
representing $\theta$ by a genuine map of complexes $\theta_\bullet:=\psi_\bullet[2]\circ f\udot\colon B\udot\to B_\bullet[2]$,
\begin{equation}\label{cxis2}
\xymatrix@R=15pt{
\dots \ar[r]^{d_3}& B^{-2} \ar[r]^{d_2}\ar[d]^{\theta_0}& B^{-1} \ar[d]^{\theta_1}\ar[r]^{d_1}& B^0 \ar[r]^{d_0}\ar[d]^{\theta_2}& \dots \\
\dots \ar[r]^{d_0^*}& B_0 \ar[r]^{d_1^*}& B_1 \ar[r]^{d_2^*}& B_2 \ar[r]^{d_3^*}& \dots}
\end{equation}
Since $\theta_\bullet=\theta\udot[2]$ as morphisms in $D(M)$, we may replace $\theta_\bullet$ by $\frac12(\theta_\bullet+\theta\udot[2])$ to ensure $\theta_i=\theta_{2-i}^*$ while \eqref{commdg} still commutes.

The restriction of $\EE$ to any closed point has cohomology only in degrees $[-2,0]$, so the truncation $\tau^{[-2,0]}$ of \eqref{cxis2} gives canonically quasi-isomorphic complexes of locally free sheaves. That is we replace $B^{-2}$ and $B^0$ by the locally free sheaves coker$\,d_3$ and $\ker d_0$ respectively. Dually we replace $B_2$ by $\ker d_3^*$ and $B_0$ by  coker$\,d_0^*$. The maps $\theta_i$ induce maps of complexes on the truncations. Thus, by applying $\tau^{[-2,0]}$ to \eqref{commdg}, we may assume $B^{-i}=0=B_i$ for $i\not\in[0,2]$ in \eqref{cxis2}. Similarly replacing $g\udot$ by $\tau^{[-2,0]}g\udot$ we also keep the map to $\tau^{[-2,0]}L\udot_M=L\udot_M$.

So we now have a commutative diagram in $D(M)$ like \eqref{commdg} intertwining $\theta$ with the map of complexes
\begin{equation}\label{tu}
\hspace{1cm}\xymatrix@C=30pt{
B^{-2} \ar[r]^{d_2}\ar[d]^{\theta_2^*}& B^{-1} \ar[d]_{\theta_1\!}^{\!=\;\theta_1^*}\ar[r]^{d_1}& B^0\! \ar[d]^{\theta_2}\! \\
B_0 \ar[r]^{d_1^*}& B_1 \ar[r]^{d_2^*}& B_2.\!\!}
\end{equation}
Since $\theta$ is a quasi-isomorphism the total complex of \eqref{tu} is acyclic. In particular $d_2^*\oplus\theta_2$ is onto and
$$
E\ :=\ \ker\big(B_1\oplus B^0\xrightarrow{d_2^*\oplus\;\theta_2}B_2\big)\ \cong\ \mathrm{coker}\big(B^{-2}\xrightarrow{d_2\oplus\;\theta_2^*}B^{-1}\oplus B_0\big)\ \cong\ E^*
$$
defines a bundle with orthogonal structure $E\cong E^*$, i.e. an $O(r,\C)$ bundle.

Now the acyclicity of \eqref{tu} is equivalent to the vertical map of complexes
$$
\xymatrix@R=18pt@C=40pt{
B^{-2} \ar[r]^-{d_2\oplus\;\theta_2^*}& B^{-1}\oplus B_0 \ar[r]^-{d_1\oplus\;0}\ar[d]_{\theta_1\oplus\;-d_1^*}& B^0 \ar[d]^{\theta_2} \\
& B_1 \ar[r]^{d_2^*}& B_2}
$$
being a quasi-isomorphism. Dividing by the injection $d_2\oplus\theta_2^*$ shows the top row is just the complex $E\to B^0$, quasi-isomorphic to the bottom row $B_1\to B_2$. Therefore the vertical quasi-isomorphism of horizontal complexes \eqref{tu} factors through the vertical quasi-isomorphisms of complexes
\beq{BtoE}
\xymatrix@R=16pt{
B^{-2} \ar[r]\ar[d]& B^{-1} \ar[d]\ar[r]& B^0\! \ar@{=}[d] && B\udot \ar[d]<-0.3ex>^-{\eta\udot} \\
B_0 \ar[r]\ar@{=}[d]& E \ar[d]\ar[r]& B^0\! \ar[d] &=:& E\udot\,\cong\ \EE\hspace{-9mm} \ar[d]<-0.3ex>^-{\eta_\bullet} \\
B_0 \ar[r]& B_1 \ar[r]& B_2\! && B_\bullet.\!\!}
\eeq
By their construction the two arrows in $E\udot:=\{B_0\to E\to B^0\}$ are dual to each other, so it is a self-dual complex. 
So setting $\alpha$ to be composition of the inverse of $\eta\udot\colon B\udot\rt{\!\sim\!}E\udot$ with the morphism $\phi\circ\psi\udot\colon B\udot\to\EE$ of \eqref{commdg} gives the result claimed.\medskip

Finally, the map of complexes $g\udot$,
$$
\xymatrix@R=15pt{
B^{-2} \ar[r]& B^{-1} \ar[r]\ar[d]_-{g^1}& B^0 \ar[d]^-{g^0} \\
& I/I^2 \ar[r]_d& \Omega_A|_M}
$$
factors through $E\udot$ by
$$
\xymatrix@R=20pt@C=28pt{
B^{-2} \ar[r]\ar[d]_{\theta_2^*}& B^{-1} \ar[r]\ar[d]_{(1,0)}& B^0\! \ar@{=}[d]&& B\udot\! \ar[d]^{\eta\udot} \\
B_0 \ar[r]^-{0\;\oplus-1}& \frac{B^{-1}\oplus B_0}{B^{-2}} \ar[r]\ar[d]_-{g^1\oplus\;0}& B^0\! \ar[d]^-{g^0} &=& E\udot\! \ar[d] \\
& I/I^2 \ar[r]_d& \Omega_A|_M && L\udot_M.\!\!}
$$
By adding copies of the acyclic self-dual complex
$$
\xymatrix{
\cO_M(N) \ar[r]^-{(1,0)}& \cO_M(N)\oplus\cO_M(-N) \ar[r]^-{(0,1)}& \cO_M(-N), & N\gg0,}
$$
to $E\udot$, mapping to $L\udot_M$ by sections of $I/I^2(N)$, we may assume that $E\to I/I^2$ is a surjection. Since $h^0(E\udot)\to h^0(L\udot_M)$ is also a surjection (in fact an isomorphism) this shows that $E\udot\to L\udot_M$ is a surjection in each degree.
\end{proof}

\subsection*{Determinants} We defined orientations for orthogonal bundles in  Definition \ref{orr}. To get a similar definition for $E\udot$ and $\EE$ we first need to review determinants of complexes. Fix a bounded complex of bundles $A\udot$ over a quasi-projective scheme $Y$, with $A^j$ in degree $j$,
$$
A\udot\=\ldots\to A^{2i-1}\to A^{2i}\to A^{2i+1}\to\ldots.
$$
Following the conventions of \cite{KM} we define its determinant by 
$$
\det A\udot\ :=\ \Lambda^{\mathrm{top}} \Big[\bigoplus\nolimits_i A^{2i}\ \oplus\ \bigoplus\nolimits_j (A^{2j+1})^*\Big].
$$
In \cite{KM} this is shown to extend to give a determinant functor, unique up to canonical isomorphism, from the category of perfect complexes (with morphisms quasi-isomorphisms) to the category of line bundles (with morphisms the isomorphisms) satisfying natural compatibilities. The key is to describe the right isomorphism $\det f\colon\det A\udot\rt\sim\det B\udot$ induced by a quasi-isomorphism $f\colon A\udot\to B\udot$ between complexes of bundles. We may write $f$ as a composition $b\udot\circ(a\udot)^{-1}$ or roof
\beq{roof}
\xymatrix@R=14pt@C=8pt{ &C\udot \ar@{->>}[dl]_{a\udot}\ar@{->>}[dr]^{b\udot} \\ A\udot \ar@{-->}_f[rr]&& B\udot,\!}
\eeq
where $a\udot$ and $b\udot$ are genuine chain \emph{maps of complexes} which are also surjective and quasi-isomorphisms. By symmetry it is then enough to describe $\det a\udot\colon\det C\udot\rt{\!\sim\!}\det A\udot$. As a map of sheaves (in fact line bundles) it is sufficient to describe it \emph{locally}. Here we may assume $a\udot$ is \emph{split} by an injective map of complexes $A\udot\into C\udot$ so it becomes
\beq{CKA}
C\udot\=K\udot\oplus A\udot\rt{(0,1)}A\udot,
\eeq
where $K\udot:=\ker a\udot$ is acyclic. Since we are working locally we can write $K\udot$ as a sum of 2-term complexes of the form
$$
K[i]\rt\sim K[i-1] .
$$
Using \eqref{A+B} it is now sufficient to trivialise the determinant of this by
\beq{AA}
\det\!\(K[i] \rt\sim K[i-1]\) \= \det (K \oplus K^*) \rt\sim \cO_Y, \quad {\bf k}\wedge{\bf k}^* \Mapsto 1.
\eeq
Here, as in \eqref{pairA} and \eqref{adual}, our convention is that if $\{k_i\}_{i=1}^n$ is a local basis of sections of $K$ with dual basis $\{k^*_i\}_{i=1}^n$ of $K^*$ and ${\bf k}:=k_1\wedge\dots\wedge k_n \in \det K$, then its dual is ${\bf k}^*=k_n^*\wedge\dots\wedge k_1^* \in \det K^*$. The resulting trivialisation \eqref{AA} appears in \cite[p32]{KM} for $i$ odd and \cite[p33]{KM} for $i$ even. 

%The functorial property
%$$
%\text{for } \ \ A\udot \rt{f} B\udot \rt{g} C\udot, \ \text{ we have }\ \ \det(\; g\circ f) = \det g \circ \det f
%$$
%also follows from the argument using surjective roof and local splittings.

%\medskip

\subsection*{Pairing}
Fix a complex of bundles $A\udot$ and set ${\bf A}:=\oplus_i \; A^i$. We define
\beq{pairA.}
p\_{\!A\udot}\,\colon\det A\udot \otimes \det (A\udot)^\vee \ \stackrel{\eqref{A+B}}\cong\ \det({\bf A}\oplus{\bf A}^{\!*}) \To \cO_Y, \quad {\bf a}\wedge {\bf a}^* \Mapsto 1.
\eeq
This definition reduces to \eqref{pairA} when $A\udot=A[0]$ is a bundle. Like \eqref{pairA} it privileges $A\udot$ over its dual $(A\udot)^\vee$: if we use \eqref{ABBA} and $p\_{(A\udot)^{\!\vee}}$ to produce another map 
$\det A\udot \otimes \det (A\udot)^\vee\to\cO_Y$ then it differs from \eqref{pairA.} by the sign $(-1)^{\rk A\udot}$.

The pairing \eqref{pairA.} induces a pairing on the derived category of perfect complexes, so we can write $\det\AA\otimes\det\AA^{\!\vee}\to\cO_Y$ for any $\AA\in D^b(\mathrm{Perf}\,Y)$. Again the key point is to handle the roofs \eqref{roof} locally, where we now check that (\ref{CKA}, \ref{AA}) induce an isomorphism $\det\!\({\bf A}\oplus{\bf K}\oplus({\bf A}\oplus{\bf K})^*\)\to\det({\bf A}\oplus{\bf A}^{\!*})$ intertwining $p\_{\!A\udot\oplus K\udot}$ and $p\_{\!A\udot}$. By repeated use of \eqref{A+B}
we get
$$
\xymatrix@=25pt{
\det{\bf A}\otimes\det{\bf A}^{\!*}\otimes\det{\bf K}\otimes\det{\bf K}^*\ar[d]_{\eqref{AA}}\ar[r]^{\eqref{ABBA}} & \det{\bf A}\otimes\det{\bf K}\otimes\det{\bf K}^*\otimes\det{\bf A}^{\!*} \ar[d]^{p\_{\!A\udot\oplus K\udot}} \\ \det{\bf A}\otimes\det{\bf A}^{\!*} \ar[r]_{p_{\!A\udot}}&\cO_Y,\!}
$$
which by inspection acts commutatively as follows,
$$
\xymatrix@=18pt{
\ {\bf a}\otimes{\bf a^*}\otimes{\bf k}\otimes{\bf k^*}\ \ar@{|->}[d]\ar@{|->}[r] & \ {\bf a}\otimes{\bf k}\otimes{\bf k^*}\otimes{\bf a^*}\!\!\! \ar@{|->}[d] \\ \ \ {\bf a}\otimes{\bf a^*}\ \ar@{|->}[r]& \,1.\!}
$$

%where $A$ is the sum $A:=\oplus_i \; A^i$ (beware there is no dual here!). The pairings commute 
%$$
%\xymatrix@=15pt{
%\det A\udot \otimes\det(A\udot)^\vee \ar[dr]\ar@{=}[rr]^\sim&& \det B\udot\otimes\det(B\udot)^\vee \ar[dl] \\ &\cO_Y,\!}
%$$
%for a quasi-isomorphism $A\udot \to B\udot$ and so give a pairing $p\_{\mathbb A}:\det\mathbb A\otimes\det(\mathbb A^\vee)\to\cO_Y$ for any object $\mathbb A$ of the derived category of perfect complexes. Here is a brief explanation.
%Using surjective roof and local splittings as before, we may assume that $A\udot = \{K[i] \to K[i-1]\}$ and $B\udot=0$ hence $\det B\udot = \cO_Y$. Then for ${\bf k} \in \det K$, the determinant functor $\det A\udot \otimes\det(A\udot)^\vee \to \det B\udot \otimes\det(B\udot)^\vee$ takes
%$$
%\({\bf k} \wedge {\bf k}^*\) \ot \({\bf k}^* \wedge {\bf k} \) \ \Mapsto \ 1,
%$$
%whereas the pairing $p\_{A\udot}$ takes
%$$
%\({\bf k} \wedge {\bf k}\) \wedge \({\bf k}^* \wedge {\bf k}^* \) \ \Mapsto \ 1.
%$$
%The commutativity follows from
%$$
%{\bf k} \wedge {\bf k}^* \wedge {\bf k}^* \wedge {\bf k} = {\bf k} \wedge {\bf k} \wedge {\bf k}^* \wedge {\bf k}^*. 
%$$
%%\medskip

\subsection*{Orientations}
The duality $\theta\colon\EE\rt\sim\EE^\vee[2]$ of \eqref{duel} induces an isomorphism
\beq{detEE}
\(\!\det \EE\)^{\otimes2}\ \xrightarrow[\raisebox{.5ex}{$\sim$}]{\ 1\otimes\:\det\theta\ }\ \det \EE\otimes\det \EE^\vee\ \xrightarrow[\raisebox{.5ex}{$\sim$}]{\ p\_{\;\EE} \ }\ \cO_M
\eeq
on the moduli space $M=M(X,c)$ of stable sheaves of Chern character $c$ and virtual dimension $\vd:=\rk\EE$. By analogy with Definition \ref{orr} we define an \emph{orientation} on $\EE$ to be an isomorphism
\beq{orcx}
o\ \colon\,\cO_M\rt\sim\det\EE\,\text{ such that }\,\eqref{detEE}\circ o^{\;\otimes2}\=(-1)^{\frac{\vd(\vd-1)}2}.
\eeq
Using different sign conventions, Team Joyce \cite[Corollary 1.17]{CGJ} has constructed an orientation on $\EE$, canonical up to sign. Throughout this paper we arbitrarily fix one of these two orientations; choosing the other would simply multiply our virtual cycle by $-1$.

By Proposition \ref{form} there exists an orthogonal bundle $(E,q)$ and a quasi-isomorphism
\beq{EEEdot}
\EE\ \cong\ E^\bullet\ :=\ \{T \to E \to T^* \}
\eeq
to a self-dual complex $E\udot$ in degrees $-2,-1,0$, intertwining $\theta$ with
\beq{thedot}
\theta_{E^\bullet}\ :=\ \(\id_T,\: q,\: \id_{T^*}\)\ \colon\,E^\bullet \rt\sim E_\bullet[2].
\eeq

\begin{Prop}\label{por} The choice of orientation \eqref{orcx} defines a canonical orientation \eqref{OR} on $(E,q)$. Thus $E$ is an $SO(r,\C)$ bundle.
\end{Prop}

\begin{proof}
Let $Q$ be the obvious quadratic form on $T\oplus T^*\oplus E^*$ --- the direct sum of the pairing on $T\oplus T^*$ with the inverse $q^{-1}\colon E^*\to E$ of the quadratic form $q\colon E\rt\sim E^*$. Comparing with \eqref{thedot} we find
\beq{theta=Q}
\det\theta_{E\udot}\=\det Q\ \colon\,\det\!\(T\oplus T^*\oplus E^*\)\To\det\!\(E\oplus T\oplus T^*\).
\eeq
Therefore, under this identification, an orientation \eqref{orcx} on $E\udot$ is the same as an orientation on $(T\oplus T^*\oplus E^*,Q)$.

Now $T\oplus T^*$ has a canonical orientation --- the one with respect to which $T$ is a positive maximal isotropic subbundle. By Definition \ref{Lor} it is
\beq{oTis}
o\_{\;T}\ :=\ (-i)^m\,t_1\wedge t_1^*\wedge\dots\wedge t_m\wedge t_m^*\=(-i)^m\,{\bf t}\wedge{\bf t}^*,
\eeq
where $\{t_i\}_{i=1}^m$ is a local basis of sections of $T$ with dual basis $\{t_i^*\}$ for $T^*$. With respect to
$$
\det E\udot\ \cong\ \det\!\(T \oplus T^*\) \otimes \det E^*
$$
we can write any orientation on $E\udot$ as $o\_{\;T}\otimes o^*$, where $o^*$ is an orientation on $E^*$. Thus \eqref{orcx} endows $E^*$ with a canonical orientation $o^*$. Via $(\det q)^{-1}\colon\det E^*\to\det E$ this induces a canonical orientation $o'$ on $E$. 
\end{proof}

Suppose we change the representative $E\udot$ of $\EE$ by adding an acyclic complex $K\udot=\{K\to K\oplus K^*\to K^*\}$,
\beq{addK}
E\udot\oplus K\udot\=\{T\oplus K\to E\oplus K\oplus K^*\to T^*\oplus K^*\},
\eeq
where the quadratic form on $E\oplus K\oplus K^*$ is the direct sum of $q$ and the pairing. Then the orientation on $E\oplus K\oplus K^*$ prescribed by Proposition \ref{por} works out to be
\beq{oror}
o'\otimes o\_K\ \in\ \det E\otimes\det(K\oplus K^*),
\eeq
with $o\_K$ the canonical orientation \eqref{oTis}. That is, under the isomorphism \eqref{pairA} $\det(K\oplus K^*)\rt\sim\cO_M$ we use the trivialisation $(-i)^{\;\rk K}$ rather than 1, just as in \eqref{specify}.

%
%Combining the quasi-isomorphism \eqref{EEEdot} with \eqref{A+B} gives
%\begin{multline}
%\det \EE\ \cong\ \det E\udot\=\det\!\(T \oplus T^*\) \otimes \det E^* \\
%\rt{\id\otimes(\det q)^{-1}}\det\!\(T \oplus T^*\) \otimes \det E^*. \label{EdotE}
%\end{multline}
%Using \eqref{EdotE} to write \eqref{orcx} as
%$$
%o\=o\_{\;T}\otimes o\_E,
%$$
%then $o_E$ is an orientation on $(E,q)$ in the sense of \eqref{OR}.
%
%****fixme***The pairing $p\_T$ of \eqref{pairA.} takes $o_{\;T}\mapsto(-i)^m$. So identifying \eqref{EdotE} with $\det E$ via
%$$
%\det\!\(T \oplus T^*\) \otimes \det E^*
%\xrightarrow[\raisebox{.5ex}{$\sim$}]{\ (-i)^{m}\cdot p\_T\;\otimes\;(\det q)^{-1}\ }\ \det E
%$$
%then an orientation $\eqref{orcx}$ on $\EE$ maps to an orientation on $E$ because 
%$$
%\xymatrix@C=20mm@R=8mm{
%(\det \EE)^{\otimes2} \ar[r]^-{\eqref{detEE}} \ar[d]_-{\eqref{EdotE}^{\otimes2}} & \cO_M \ar[d]^{(-1)^{m}(-1)^{r}} \\
%(\: \det E\: )^{\otimes2} \ar[r]^-{\eqref{or2}} & \cO_M
%}
%$$
%is commutative. Notice that the sign comes from the identity 
%$$
%(-1)^{\frac{\vd(\vd-1)}2}\ =\ (-1)^{m}(-1)^{r}(-1)^{\frac{r(r-1)}2}
%$$
%which compares the signs in Definition \ref{orr} and \eqref{orcx}.

\subsection{Virtual cycle}\label{vcsec} In an act of violence, we denote by $\tau E\udot$ the stupid truncation $\{E\to T^*\}$ of the complex $E\udot$ \eqref{Kdot}. Composing with the obstruction theory \eqref{ObsTh},
\beq{taupot}
\tau E\udot\To E\udot\rt\alpha\EE\rt\At\LL_M
\eeq
defines a \emph{perfect} obstruction theory $\tau E\udot\to\LL_M$ for $M$ since each map induces an isomorphism on $h^0$ and a surjection on $h^{-1}$. (Note that $\tau E\udot\to\LL_M$ is only a morphism in $D(M)$; we do not require it to be described via a map of complexes.)

Therefore by \cite[Section 4]{BF} we get a cone $C_{E\udot}\subset E^*\cong E$ by pulling back the intrinsic normal cone $\mathfrak C_M\subset h^1/h^0\big((\tau E\udot)^\vee\big)=[E^*/T]$,
\beq{recipe}
C_{E\udot}\ :=\ \mathfrak C_M\times\_{[E^*/T]}E^*\ \subset\ E^*\ \cong\ E.
\eeq
The Behrend-Fantechi virtual cycle would then be the Fulton-MacPherson intersection $0_E^{\;!}[C_{E\udot}]\in A_{\vd}(M)$ of $C_{E\udot}$ with the zero section $0_{E}$, but our use of the truncation $\tau$ makes this the wrong thing to do --- the result (in fact even its virtual dimension) is not a quasi-isomorphism invariant of $E\udot\cong\EE$. Instead the following result will allow us to replace $0_E^{\;!}[C_{E\udot}]$ by $\surd\;0_E^{\;!}[C_{E\udot}]$, giving a cycle in $M$ of \emph{half} the virtual dimension.

\begin{Prop}\label{isoprop} The cone $C_{E\udot}\subset E$ is isotropic.
\end{Prop}

\begin{proof}
Since this is a local statement it is sufficient to prove it on any \'etale neighbourhood $U$ of a point $[F]\in M$. The proof is quite lengthy; we first use \cite{BBBJ} to get the nice local form \eqref{FisE} for the obstruction theory \eqref{ObsTh}. Then we compare the resulting self-dual complex $Q\udot$ to our representative $E\udot$ from Proposition \ref{form}.

We work on a ``minimal standard form neighbourhood" of $F\in\cM^{\mathrm{der}}$ provided by \cite[Theorem 2.8]{BBBJ}. This is a result about the structure of derived Artin stacks\footnote{In particular this does not use the shifted symplectic structure.} that gives a derived affine \emph{scheme} $\cU$ and a smooth morphism $\cU\to\cM^{\mathrm{der}}$ of relative dimension $1=\dim\Aut(F)$ such that $\LL_{\;\cU/\cM^{\mathrm{der}}}=\cL$ is a line bundle on $\cU$. So at the level of cdgas, $\cU\to\cM^{\mathrm{der}}$ modifies only the degree 1 piece governing stabilisers.

If we ignore the derived structure, $\cU\to\cM^{\mathrm{der}}$ becomes an atlas $U\to\cM$ for an open neighbourhood of $F\in\cM$. Here $U$ is the scheme underlying $\cU$, and the composition $U\to\cM\to M$ is an \'etale neighbourhood of $[F]\in M$. 

By \cite{PTVV}, $\cM^{\mathrm{der}}$ carries a $(-2)$-shifted symplectic structure. Its pullback to $\cU$ is no longer $(-2)$-shifted symplectic because we have removed $\cL[-1]$ from $\LL_{\cM^{\mathrm{der}}}$ but not yet its shifted dual $\cL^*[3]$. This is done by \cite[Theorem 2.10(c)]{BBBJ}, replacing $\cU$ by another affine derived scheme $\cU'$ with the same underlying affine scheme $U$ but with its cdga replaced by a sub-cdga such that
$$
\LL_{\;\cU/\;\cU'}\ \cong\ \LL_{\;\cU/\cM^{\mathrm{der}}}^\vee[3]\ \cong\ \cL^*[3],
$$
by \cite[Equation 2.14]{BBBJ}. After possibly shrinking the \'etale neighbourhood $U$ of $[F]\in M$ if necessary, \cite{BBBJ} prove that the pullback to $\cU$ of the $(-2)$-shifted symplectic structure on $\cM^{\mathrm{der}}$ descends to a $(-2)$-shifted symplectic structure on $\cU'$ which can be written in the standard Darboux form of \cite[Theorem 5.18(ii) and Example 5.16]{BBJ}; we describe this next.

Since $U$ is also the underlying scheme of $\cU'$ it inherits an obstruction theory from the derived structure of $\cU'$. This is the pullback of the obstruction theory $\At\colon\EE\to\LL_M$ \eqref{ObsTh} to our \'etale neighbourhood $U\to M$ of $[F]\in M$ on which we will show the cone $C_{E\udot}\subset E$ is isotropic. The Darboux form is described in terms of the tangent and obstruction spaces of $\cU'$ at a point lying over $[F]\in M$, which are
$$
h^0\(\LL^\vee_{\cM^{\mathrm{der}}}|\_F\)\=\Ext^1(F,F)\ \text{ and }\ h^1\(\LL^\vee_{\cM^{\mathrm{der}}}|\_F\)\=\Ext^2(F,F).
$$
Let $q$ denote the quadratic form on the latter given by Serre duality. Then by \cite[Example 5.16]{BBJ} there is an open neighbourhood $V$ of $0\in\Ext^1(F,F)$ and an isotropic map
\beq{ext2}
s\,\colon V\To\Ext^2(F,F)\,\text{ such that }\, ds|\_0\,=\,0,\ q(s,s)\,\equiv\,0 \,\text{ and }\,s^{-1}(0)\,\cong\,U
\eeq
compatible with the obstruction theory $\LL_{\cM^{\mathrm{der}}}|\_U\to\LL_U$ in the following sense.
Let $Q\cong Q^*$ denote the trivial orthogonal bundle with fibre $\Ext^2(F,F)$ over $V$, and think of $s$ as an isotropic section of $Q$. Then $\LL_{\cM^{\mathrm{der}}}|\_U\to\LL_U$ is isomorphic to
\beq{FisE}
\xymatrix@R=2pt{\hspace{8mm}
Q\udot \quad\ :=\ \quad \big\{T_V|_U \ar[r]^-{ds}& Q \ar[dd]^-s\ar[r]^-{(ds)^*}& \Omega_V|_U\big\}\!\! \ar@{=}[dd]<-.3ex> && \EE|_U \ar[dd] \\ &&& \cong \\
& I/I^2 \ar[r]^d& \Omega_V|_U && \LL_{\;U}.\!\!}
\eeq
Here $I$ is the ideal of $Z(s)\subset V$ generated by $s$. Finally, this isomorphism intertwines the Serre duality \eqref{duel} with the obvious self-duality $Q\udot\cong Q_\bullet[2]$ \eqref{dueldot} of the above complex.\medskip

Let $E\udot\cong\EE$ be the self-dual complex given to us by Proposition \ref{form}. Since we are working locally, where morphisms in $D(M)$ can be lifted to homotopy classes of maps of complexes of locally free sheaves, the isomorphism \eqref{FisE} and dualities it intertwines can be represented by a map of complexes $f\udot$ such that
\beqa 
\xymatrix@R=18pt@C=45pt{
Q\udot \ar[d]_{f\udot}\ar@{=}[r]^{\eqref{dueldot}}& Q_\bullet[2] \\
E\udot|_U \ar@{=}[r]_{\eqref{dueldot}}& E_\bullet|_U[2] \ar[u]_{f_\bullet[2]}}
\eeqa
commutes \emph{only up to homotopy}.

On restriction to the point $[F]\in U$ the differentials $ds,\,(ds)^*$ in $Q\udot$ vanish, so the three terms of $Q\udot|_{[F]}$ equal its cohomology groups $\Ext^i(F,F)$ and the composition $f_\bullet\circ f\udot$ is precisely $\id_{Q\udot}$. Thus the $f^{-i}$ are injections and the $f_{2-i}$ are surjections at $[F]$ --- and therefore also over $U$, after shrinking if necessary.
In particular $K_\bullet:=\ker f_\bullet[2]$ is a 3-term acyclic complex of vector bundles $\{K_{-2}\to K_{-1}\to K_0\}$. By the definition of $f_\bullet$ it is orthogonal to $f\udot(Q\udot)$ under the identification $E\udot\cong E_\bullet[2]$, giving an orthogonal splitting
\begin{eqnarray}\nonumber
E\udot &=& f\udot(Q\udot)\,\oplus\,\ker\(f_\bullet[2]\) \\ &\cong& Q\udot\,\oplus\,K_\bullet
\label{splitEF}
\end{eqnarray}
over $U$.  The pairing on $E\udot$ makes $K_\bullet$ isomorphic to its shifted dual via the composition
$$
K_\bullet\=\ker f_\bullet[2]\Into E_\bullet[2]\ \cong\ E\udot\To\coker f\udot\=(K_\bullet)^\vee[2],
$$
so $K_\bullet$ is a self-dual complex.

The quasi-isomorphism $Q\udot\cong\EE$ over $U$ gives a cone $C_{Q\udot}\subset Q$ by Behrend-Fantechi's recipe \eqref{recipe}. As in \cite[Remark 5.1.1]{F} it is the flat limit (in the Hilbert scheme of subschemes of the total space of $E$) of the graphs of $ts$,
\beq{limiso}
C_{Q\udot}\=\lim_{t\to\infty}\Gamma_{\!ts}\ \subset\ Q.
\eeq
Considering the quadratic form $q$ on $Q$ to be a function on its total space (quadratic on the fibres), it vanishes on $\Gamma_{\!ts}\subset Q$ since $ts$ is an isotropic section. Taking the limit as $t\to\infty$ shows that $C_{Q\udot}$ is an isotropic cone. (Notice the similarity to the proof of Lemma \ref{Lem:Zero}.)

By the recipe \eqref{recipe} and the splitting \eqref{splitEF} we deduce the Behrend-Fantechi cone $C_{E\udot}\subset E\cong E^*$ induced from the quasi-isomorphism $E\udot\cong\EE$ is
\beq{ansa}
C_{E\udot}\ \cong\ C_{Q\udot}\oplus K^0\ \subset\ Q\oplus K^1.
\eeq
Since $K_\bullet$ is self-dual and exact over $U$,
$$
\im\(K^0\into K^1\)\=\ker\(K^1\onto K^2\)\ \cong\ \ker\(K_{-1}\onto K_0\)
$$
shows $\im(K^0)$ is a maximal isotropic in $K^1$. Therefore the cone \eqref{ansa} is isotropic over $U$.
\end{proof}

By Propositions \ref{form} and \ref{por}, $E$ is an $SO(r,\C)$ bundle on $M$, where $r$ has the same parity as the virtual dimension $\vd$ \eqref{vd}. So Proposition \ref{isoprop} and Definition \ref{IsoCone} finally allow us to define a virtual cycle for $M$.

\begin{Def}\label{vdef} If $\vd$ is odd define $[M]^{\vir}=0$. If $\vd$ is even, define
\beq{virMdef}
[M]^{\vir}\=\sqrt{0_E^{\;!}}\,[C_{E\udot}]\ \in\ A_{\frac12\!\vd}\(M,\ZZ\).
\eeq
\end{Def}

We need to check that this definition is independent of choices: given two self-dual complexes $E\udot,\,F\udot$ with quasi-isomorphisms to $\EE$ intertwining $\theta$ as in Proposition \ref{form}, we would like to show that
\beq{wont}
\sqrt{0_{E}^{\;!}}\,[C_{E\udot}]\=\sqrt{0_F^{\;!}}\,[C_{F\udot}].
\eeq
We may represent the quasi-isomorphism $E\udot\cong\EE\cong F\udot$ by a roof of maps of complexes
\beq{oof}
\xymatrix@=14pt{
& A\udot \ar[dl]_{e\udot}^(.4){\rotatebox[origin=c]{45}{${\!}_\sim$}}\ar[dr]^{f\udot}_(.45){\rotatebox[origin=c]{-30}{${}^\sim$}} \\ E\udot && F\udot.\!\!}
\eeq
Replacing $A\udot$ by its truncation $\tau^{[-2,0]}A\udot$, we may assume it is also a 3-term complex of locally free sheaves. It will not, in general, be self-dual, however. It inherits two representatives
\beqa 
\theta\udot_e\ :=\ e_\bullet[2]\circ e\udot\ \text{ and }\ \theta\udot_f\ :=\ f_\bullet[2]\circ f\udot\ \colon\,A\udot\To A_\bullet[2]
\eeqa
of $\theta$, both of which are self-dual: $\theta\udot_e=(\theta\udot_e)^\vee[2]$ and $\theta\udot_f=(\theta\udot_f)^\vee[2]$. From such data --- a 3-term complex with a self-dual map to its dual$[2]$ --- we defined in \eqref{BtoE} a canonical \emph{self-dual} complex. Applied to $(A\udot,\theta\udot_e)$ and $(A\udot,\theta\udot_f)$ we get self-dual complexes $A\udot_e,\,A\udot_f$ respectively. By \eqref{BtoE},
\beq{Ae}
A\udot_e\=\big\{A_0\To A_e\To A^0\big\},\quad\text{where }\, A_e\ :=\ \frac{A^{-1}\oplus A_0}{A^{-2}}\ \cong\ A_e^*
\eeq
is naturally an orthogonal bundle. Note $A\udot_f$ takes the same form, but the arrows (and so the orthogonal bundle $A_f$) induced by $\theta_f\udot$ will be different.

Though this construction is canonical --- and applied to the self-dual complexes $E\udot,\,F\udot$ it returns them unmodified --- it is \emph{not} functorial due to the appearance of both covariant terms $A^{-i}$ and the contravariant term $A_0$. Thus the roof \eqref{oof} does not induce maps amongst $A_e\udot,\,A\udot_f,\,E\udot,\,F\udot$. Nonetheless Proposition \ref{give} below proves the virtual cycles made from $A_e\udot$ and $E\udot$ are the same. By symmetry this also equates the virtual cycles associated to $A_f\udot$ and $F\udot$. So finally to prove the desired equality \eqref{wont} we must check the virtual cycles associated to $A_e\udot$ and $A_f\udot$ are the same.

This follows from a simple deformation argument. Since $\theta\udot_e$ and $\theta\udot_f$ describe the same morphism in $D^b(M)$, for each $t\in\C$ so does
$$
\theta\udot_t\ :=\ t\theta\udot_e+(1-t)\theta\udot_f,
$$
which is also self-dual. Applying \eqref{BtoE} gives orthogonal bundles $A_t$ and self-dual complexes
\beq{tfam}
A\udot_t\=\big\{A_0\To A_t\To A^0\big\},
\eeq
all of which vary in a flat family over $\C\ni t$. Let $\bf A$ denote the bundle on $M\times\C$ whose restriction to $M\times\{t\}$ is $A_t$. Truncating \eqref{tfam} gives a perfect \emph{relative} obstruction theory \cite[Section 7]{BF} $\{{\bf A}\to A^0\}$ for $M\times\C\to\C$. The resulting absolute perfect obstruction theory has a local model just as in Proposition \ref{isoprop}, given by allowing \eqref{ext2} to vary with $t\in\C$. That is, we replace the ambient space $V$ by $V\times\C$, and the section $s$ by $s_t$. This is isotropic in $\bf A$ because $s_t$ is isotropic in $A_t$ for each fixed $t\in\C$.

Therefore we get an isotropic cone ${\bf C}\subset{\bf A}$ over $M\times\C$ just as in \eqref{limiso}. We apply the square root Gysin map of Definition \ref{IsoCone} defined via cosection localisation. By \cite[Theorem 5.2]{KL} it satisfies
\beq{defo}
\sqrt{0_{A_t}^{\;!}}\,[C_{A_t\udot}]\=\iota_t^{\;!}\left(\sqrt{0_{\bf A}^{\;!}}\,[{\bf C}]\right), \quad t\in\C,
\eeq
where $\iota\_t$ is the inclusion of any point $\{t\}\into\C$. Hence it is independent of $t$ and setting $t=0,1$ shows
$$
\sqrt{0_{A_f}^{\;!}}\big[C_{A_f\udot}\big]\=\sqrt{0_{A_e}^{\;!}}\big[C_{A_e\udot}\big].
$$
So we can finally conclude the well-definedness \eqref{wont} of our virtual cycle by proving the following.

\begin{Prop}\label{give} In the above notation, $\displaystyle{\sqrt{0_{E}^{\;!}}\,[C_{E\udot}]\=\sqrt{0_{A_e}^{\;!}}\big[C_{A_e\udot}\big].}$
\end{Prop}

\begin{proof}
Start with the map of complexes $A\udot\to E\udot$ \eqref{oof}.
By choosing a sufficiently negative vector bundle $B$ and a surjection $B\onto A_0$ and adding the acyclic self-dual complex\footnote{We endow $B\oplus B^*$ with its obvious quadratic form and canonical orientation $o\_B$ \eqref{oTis} with respect to which $B\into B\oplus B^*$ is a positive maximal isotropic.} 
\beq{Badd}
\xymatrix{B \ar[r]^--{(1,0)}& B\oplus B^* \ar[r]^-{(0,1)}& B^*}
\eeq
to $E\udot$ as in \eqref{addK}, we may assume that $A^0\to E^0=T^*$ is an injection. By \eqref{recipe} the addition of \eqref{Badd} replaces $C_{E\udot}\subset E$ by
$$
C_{E\udot}\oplus B\ \subset\ E\oplus B\oplus B^*
$$
to which we can apply Proposition \ref{E1E2}. Since $\surd\;{0_{B\oplus B^*}^{\;!}}\,[B]=0_B^{\;!}[B]$ is just the zero section by Lemma \ref{Schwantz}, the virtual cycle \eqref{virMdef} is unaffected:
\beq{BBB}
\sqrt{0_{E\oplus B\oplus B^*}^{\;!}}\,\big[C_{E\udot}\oplus B\big]\=
\sqrt{0_{E}^{\;!}}\,\big[C_{E\udot}\big].
\eeq

So without loss of generality $A\udot\to E\udot$ and its dual $E_\bullet[2]\to A_\bullet[2]$ take the form
$$
\xymatrix@R=13pt{
A^{-2} \ar[r]\ar[d]& A^{-1} \ar[r]\ar[d]& A^0 \\
T \ar[r]\ar@{->>}[d]& E \ar[d]\ar[r]& T^*\! \ar[d]\ar@{<-^)}[u]-<0pt,9pt><-.2ex> \\
A_0 \ar[r]& A_1 \ar[r]& A_2.\!}
$$
Set $K:=\ker(T\onto A^0)$. Since $E\udot\to A_\bullet[2]$ is a quasi-isomorphism, the composition $K\into T\to E$ is also an injection, with isotropic image. We get the commutative diagrams with exact columns
$$
\xymatrix@R=13pt{
K \ar@{=}[r]& K &&& A_0 \ar@{=}[d]\ar[r]& \frac{A_0\oplus A^{-1}}{A^{-2}} \ar[r]& A^0 \\
T \ar@{->>}[d]\ar@{<-^)}[u]-<0pt,9pt><-.2ex>\ar[r]& \frac{T\oplus E}{T} \ar@{->>}[d]\ar@{<-^)}[u]-<0pt,9pt><-.2ex>\ar[r]& T^* \ar@{=}[d]<-.6ex> &\text{and}& A_0 \ar[r]& \frac{A_0\oplus E}T \ar[r]\ar@{<-^)}[u]-<0pt,12pt><-.2ex>\ar@{->>}[d]& T^* \ar@{<-^)}[u]-<0pt,9pt><-.2ex>\ar@{->>}[d] \\
A_0 \ar[r]& \frac{A_0\oplus E}T \ar[r] & T^* &&& K^* \ar@{=}[r]& K^*.\!\!}
$$
Letting $K^\perp$ denote the orthogonal to $K\subset E$, by \eqref{Ae} the diagrams become
\beq{dgss}
\xymatrix@R=13pt{
K \ar@{=}[r]& K &&& A_0 \ar@{=}[d]\ar[r]& A_e \ar[r]& A^0 \\
T \ar@{->>}[d]\ar@{<-^)}[u]-<0pt,9pt><-.2ex>\ar[r]& E \ar@{->>}[d]\ar@{<-^)}[u]-<0pt,9pt><-.2ex>\ar[r]& T^* \ar@{=}[d]<-.6ex> &\text{and}& A_0 \ar[r]& (K^\perp)^* \ar[r]\ar@{<-^)}[u]-<0pt,9pt><-.2ex>\ar@{->>}[d]& T^* \ar@{<-^)}[u]-<0pt,9pt><-.2ex>\ar@{->>}[d] \\
A_0 \ar[r]& (K^\perp)^* \ar[r] & T^* &&& K^* \ar@{=}[r]& K^*.\!\!}
\eeq
Let $T\udot$ denote the complex $A_0\to(K^\perp)^*\to T^*$. The diagrams on the left and right describe quasi-isomorphisms $E\udot\onto T\udot\hookleftarrow A_e\udot$ respectively, corresponding to the reduction $K^\perp/K\leftarrow\hspace{-3mm}\leftarrow K^\perp\into E$ of $E$ by $K$.

Letting $\tau E\udot$ denote the stupid truncation $\{E\to T^*\}$ we get the perfect obstruction theory
$$
\tau E\udot\To E\udot\rt\sim\EE\rt\At\LL_M
$$
of \eqref{taupot} and similarly for $\tau A\udot_e$ and $\tau T\udot$. This removes the first columns of the diagrams \eqref{dgss}. We then dualise to find, by \eqref{recipe}, the exact sequences
$$
\xymatrix@R=13pt{
K^* &&& A_e & \ C_{A\udot} \ar@{_(->}[l]<.15ex> \\
E \ar@{->>}[u]& \ C_{E\udot} \ar@{_(->}[l]<.15ex>\ar@{=}[d] &\text{and}& K^\perp \ar@{->>}[u]\ar@{<-_)}[d]+<0pt,10pt><.15ex>& \ C_{T\udot} \ar@{<-_)}[d]+<0pt,10pt><.15ex> \ar@{->>}[u]\ar@{_(->}[l]<.15ex> \\
K^\perp \ar@{^(->}[u]<-.2ex> & \ C_{T\udot} \ar@{_(->}[l]<.15ex>&& \!K \ar@{=}[r]& K.}
$$
The upshot is that $C_{E\udot}\subset E$ lies in $K^\perp\subset E$ and contains $K$; quotienting gives $C_{A_e\udot}=C_{E\udot}/K\subset K^\perp/K$. Thus applying Proposition \ref{KperpK} to $[C_{E\udot}]$ gives
\beq{cones}
\sqrt{0_{E}^{\;!}}\,[C_{E\udot}]\=\sqrt{0_{A_e}^{\;!}}\big[C_{A_e\udot}\big].\qedhere
\eeq
\end{proof}

\subsection*{Deformation invariance} Suppose $X$ moves in a smooth projective family $\cX\to C$ over a smooth connected curve $C$. Then we can repeat our working relative to $C$. We get a coarse moduli space $\cM\to C$ with a relative obstruction theory $\At\colon\EE\to\LL_{\cM/C}$ as in \eqref{ObsTh}. We can resolve it by a self-dual 3-term complex $E\udot=\{T\to E\to T^*\}$ over all of $\cM$. Thus by \cite[Section 7]{BF} we get a cone $C_{E\udot}\subset E^*\cong E$.

Choices of orientation on the fibres of $\cM\to C$ form a $\Z/2$-local system on $C$, essentially because the proof of their existence in \cite{CGJ} depends only on the (locally trivial) topology of the fibres of $\cX\to C$. Replacing $C$ by a $\Z/2$-cover we may assume this has trivial monodromy. Therefore a choice of orientation on one fibre induces one on all of them, thus orienting $E$ as in Proposition \ref{por}.

Now the key point is that the Darboux theorems \cite{BBBJ, BG, BBJ} also work with parameters (in fact \cite{BG} explicitly works over a base) allowing us to write $\cM$ locally as the zeros of an isotropic section of an orthogonal bundle such that the derivative of this description recovers the relative obstruction theory. That is, Proposition \ref{isoprop} goes through in this family setting to show that $C_{E\udot}\subset E$ is isotropic.

The argument \eqref{defo} then produces a cycle $[\cM]^{\vir}\in A_*\(\cM,\ZZ\)$ with  the deformation invariance property that
\beq{defint}
[\cM_t]^{\vir}\=\iota_t^![\cM]^{\vir}
\eeq
for any point $\iota\_t\colon\{t\}\into C$.

\subsection{Generalisations}\label{genrul} We sketch some extensions of the above theory.

\subsection*{Fixed determinant} In Definition \ref{vdef} we allowed our Calabi-Yau 4-fold to have nonzero $H^i(\cO_X)$ for $i=1,2,3$. However when $H^2(\cO_X)\ne0$ it is easy to see that $[M]^{\vir}=0$. In this case we can 
restrict attention to sheaves of strictly positive rank $r>0$ and define a kind of reduced virtual class on the subscheme $M_L\subset M$ of stable sheaves with determinant a fixed line bundle $L$ on $X$. By \cite[Definition 3.1]{STV} there is a derived determinant map to the derived Picard stack,
$$
\cM^{\mathrm{der}}\rt{R\det}R\Pic(X).
$$
Letting $\cM^{\mathrm{der}}_L$ denote the fibre over $\{L\}$, it is a derived Deligne-Mumford stack with stabilisers $\Z/r\Z$ at every point. \'Etale locally, tensoring with $r$th roots of degree 0 line bundles splits $R\det$, giving a local product structure $\cM^{\mathrm{der}}\cong\cM^{\mathrm{der}}_L\times R\Pic(X)$. By \cite[Proposition 3.2]{STV} this induces the splitting of derived cotangent complexes 
$$
R\pi_*\;R\hom (\curly E, \curly E)[3]\ \cong\ R\pi_*\;R\hom (\curly E, \curly E)\_0[3]\,\oplus\,R\pi_*\;\cO_{X\times\cM^{\mathrm{der}}},
$$
splitting the trace map
$$
\tr\ \colon\,R\pi_*\;R\hom (\curly E, \curly E) \To R\pi_*\;\cO_{X\times\cM^{\mathrm{der}}}
$$
by $\frac1r\id_{\curly E}$. We expect that it is obvious to experts\footnote{Related results are explained in \cite[Remarks 3.1, 3.2]{PTVV} for $X$ a K3 surface, as Mauro Porta pointed out to us.} that the $(-2)$-shifted symplectic structure splits accordingly, thus endowing $\cM^{\mathrm{der}}_L$ --- and thus the underlying Deligne-Mumford stack $\cM_L$ and its coarse moduli scheme $M_L$ --- with a $(-2)$-shifted symplectic structure. Consider the restriction map
$$
R\pi_*\;R\hom (\curly E, \curly E)\_0[3]\=\LL_{\cM^{\mathrm{der}}_L}\To\LL_{\cM_L}.
$$
Since the latter is the pullback of $\LL_{M_L}$ this gives an amplitude $[-2,0]$ obstruction theory
\beq{EE}
\EE\ :=\ R\pi_*\;R\Hom(\cE,\cE)\_0[3]\To\LL_{M_L}
\eeq
replacing \eqref{ObsTh}. Here $\cE$ is again any choice of twisted universal sheaf on $X\times M_L$. The shifted symplectic isomorphism $\LL_{\cM^{\mathrm{der}}_L}\cong\LL_{\cM^{\mathrm{der}}_L}^\vee[2]$ becomes the relative Serre duality $\theta\colon \EE\to\EE^\vee[2]$ on \eqref{EE}. Thus we have all the same ingedients as before. Resolving $\EE$ by a self-dual 3-term locally free complex $E\udot$ we get an isotropic cone $C_{E\udot}\subset E$, and a virtual cycle
\beq{redvir}
[M_L]^{\vir}\ :=\ \surd\;0_E^{\;!}\;[C_{E\udot}]\ \in\ A_{\frac12\!\vd-h^1(\cO_X)+\frac12h^2(\cO_X)}\(M_L,\ZZ\).
\eeq
Deformation invariance of this class holds only for families $\cX\to C$ over which the line bundle $L$ (on the central fibre $\cX_0=X$) extends to a line bundle $\cL$ over all of $\cX$. This allows us to form a family moduli space of fixed determinant sheaves and repeat the working relative to $C$ to give the deformation invariance \eqref{defint}.

\subsection*{Quasi-projective Calabi-Yau 4-folds} Though we have worked with projective $X$ out of habit, all of our results --- except deformation invariance --- work for moduli spaces of \emph{compactly supported} stable sheaves on a quasi-projective Calabi-Yau 4-fold $X$. The paper \cite{BD} extends the $(-2)$-shifted symplectic structure of \cite{PTVV} to this setting, and the obstruction theory $\EE\to\LL_M$ is identical to the projective case. Therefore \cite{BBBJ} applies to the stack of stable compactly supported sheaves and we can deduce Proposition \ref{isoprop} as before: the cone $C_{E\udot}\subset E$ is isotropic. The orientation result \cite{CGJ} on $\EE$ is extended to the quasi-projective case in \cite{Bo1}. Thus we have all the same ingredients to produce the virtual cycle in this setting.

\subsection*{$(-2)$-shifted symplectic schemes}
Though we have not worked in such generality, our results apply to any quasi-projective scheme $M$ with a $(-2)$-shifted symplectic structure (in the sense of \cite{PTVV}) and an orientation. Quasi-projectivity ensures $M$ has enough locally free sheaves to allow us to produce a self-dual 3-term locally free resolution of the virtual cotangent bundle. Then the key point is that the papers \cite{BG,BBJ} only use the $(-2)$-shifted symplectic structure, so Proposition \ref{isoprop} works in this generality, producing an isotropic cone $C_{E\udot}\subset E$ to which we can apply $\surd\;0_E^{\;!}$ to define the virtual cycle.

\begin{Thm}
Let $M$ be a quasi-projective $(-2)$-shifted symplectic derived scheme with an orientation in the sense of \eqref{orcx}. If $\,\vd$ is odd, set $[M]^{\vir}=0$. If $\;\vd$ is even then
$$
[M]^{\vir} := \sqrt{0_E^{\;!}}\,[C_{E\udot}]\ \in\ A_{\frac12\!\vd}\( M,\ZZ\)
$$
is independent of the choice of self-dual resolution $E\udot$.

Moreover if $M$ is projective then $[M]^{\vir}$ is deformation invariant in the sense that \eqref{defint} holds.
\end{Thm}

In particular we can get virtual cycles on projective moduli spaces of simple complexes in the derived category $D(X)$ of sheaves on any projective Calabi-Yau 4-fold. This includes the moduli space $P_n(X,\beta)$ of \emph{stable pairs} $(F,s)$ of \cite{PT}, where
\begin{itemize}
\item $F$ is a pure 1-dimensional sheaf on $X$,
\item $F$ has curve class $[F]=\beta\in H_2(X,\Z)$ and $\chi(F)=n$,
\item $s\in H^0(F)$ has 0-dimensional cokernel.
\end{itemize}
The same applies to the moduli space $J(X,c)$ of \emph{Joyce-Song pairs} $(F,s)$ of \cite{JS}, where
\begin{itemize}
\item $F$ is a semistable sheaf on $X$ of Chern character $c$,
\item $n\gg0$ is sufficiently large that $H^{\ge1}(F(n))=0$ for all semistable $F$,
\item $s\in H^0(F(n))$ factors through no semi-destabilising subsheaf of $F$.
\end{itemize}
Both moduli spaces $P_n(X,\beta),\,J(X,c)$ are naturally projective schemes\footnote{The section rigidifies the pairs $(F,s)$; there are no semistables and no automorphisms.} with obstruction theory governed by the self-dual 3-term complex
$$
R\Hom\_X(I\udot,I\udot)\_0,
$$
where $I\udot\in D(X)$ denotes the complexes $\{\cO_X\rt sF\}$ and $\{\cO_X(-n)\rt sF\}$ respectively. By \cite{CGJ} $P_n(X,\beta),\,J(X,c)$ admit orientations. Choosing one, our theory endows them with virtual cycles.

\begin{Thm} If $\chi(c,c)\in2\Z$ we have algebraic virtual cycles
\beqa
[P_n(X,\beta)]^{\vir} &\in& A_n\(P_n(X,\beta),\ZZ\), \\
{}[J(X,c)]^{\vir} &\in& A_{\;\chi(c\;(n))-\frac12\chi(c,c)}\(J(X,c),\ZZ\)
\eeqa
which, by the sequel \cite{OT2}, map to the integral classes of \cite{BJ} in homology. They are deformation invariant in the sense of \eqref{defint}.
\end{Thm}

The first class puts the Calabi-Yau 4-fold conjectures of \cite{CMT2, CT1, CT2, CK3, CKM1, CKM2} on a firmer footing. 
 
When stability$\,=\,$semistability for sheaves $F$ of Chern character $c$, the Joyce-Song moduli space $J(X,c)$ is a projective bundle over $M=M(X,c)$ with fibre $\PP^{\;\chi(F(n))-1}$, so we expect a simple relationship between $[J(X,c)]^{\vir}$ and $[M]^{\vir}$, just as in the Calabi-Yau 3-fold case. More generally --- again by analogy with the Calabi-Yau 3-fold case --- one could hope for a universal wall-crossing formula like \cite[Theorem 5.27]{JS}. This would allow the definition of invariants valued in $\Q$ counting strictly semistable sheaves $F$ by using virtual counts (valued in $\Z$) of the Joyce-Song pairs $I\udot$. This will be addressed in forthcoming work of Dominic Joyce.

\section{$K$-theoretic virtual class}\label{Ksec}
Again let $M$ be a quasi-projective $(-2)$-shifted symplectic derived scheme with an orientation. As usual we have in mind the moduli space of compactly supported stable sheaves --- or stable pairs or Joyce-Song pairs --- of fixed total Chern character $c$ on a smooth quasi-projective 4-fold $X$ with $K_X\cong\cO_X$. 
%$H^i(\cO_X)=0$ for $1\le i\le3$
In this Section we will endow $M$ with a virtual structure sheaf with good properties.

For this we will need a $K$-theoretic analogue of the Edidin-Graham class of an $SO(2n,\C)$ bundle $(E,q,o)$. There are candidate classes in the literature in various special cases \cite{And, CLL, Chiodo, EG:Ch, KO, OS, PV:A}. We find a class which works in maximal generality --- with no assumption of Zariski local triviality, the existence of a maximal isotropic subbundle, or the existence of a spin bundle --- and which uses Definition \ref{Lor} to fix the sign ambiguity present in many of these classes. To do so we find only that we need to invert 2 in the coefficients of $K$-theory.

We make crucial use of Anderson's class\footnote{Anderson works in \emph{connective $K$-theory} \cite[Appendix A]{And}; we specialise his Bott class $\beta$ to $-1$ to work in $K^0$. He also restricts to Zariski locally trivial $SO(2n,\C)$ bundles.} \cite[Appendix B]{And}. We define a variant $\e$ of it with the desirable property that $(-1)^n\e\;^2$ is the $K$-theoretic Euler class; see \eqref{squares}. Our class and his are supported in codimension $n$, agree in codimension $n+1$, but differ in codimension $n+2$; see \eqref{usvA}.

\subsection{$K$-theoretic Edidin-Graham class}
On a quasi-projective scheme $Y$, let $K^0(Y)$ and $K_0(Y)$ denote the Grothendieck groups of vector bundles and coherent sheaves respectively. The former is a ring under (derived) tensor product, and the latter is a module over it.
On inverting 2 we will find that any line bundle has a distinguished square root in $K^0\(Y,\ZZ\)=K^0(Y)\otimes\ZZ$.

Using the Catalan numbers $C_i:=\tfrac1{i+1}\binom{2i}i\in\Z$, set $a_i:=2^{1-2i}C_{i-1}\in\ZZ$. This is minus the $i$th Taylor coefficient of
$$
(1-x)^{\frac12}\=1-\tfrac12x-\tfrac18x^2-\dots\=1-\sum_{i\ge1}a_ix^i\ \in\ \ZZ[\![x]\!].\vspace{-2mm}
$$
Squaring gives the identity
\beq{xhalf}
\left(1-\sum_{i=1}^{n-1}a_ix^i\right)^{\!2}\ =\ (1-x)+2a_nx^n+O(x^{n+1}),
\eeq
where $O(x^{n+1})$ denotes the product of $x^{n+1}$ and a polynomial in $x$.

For $i>0$ let $K^0(Y)^i\subset K^0(Y)$ denote the subring of elements which may be written as a formal difference of sheaves supported in codimension $i$. Since $K^0(Y)^i\cdot K^0(Y)^j\subseteq K^0(Y)^{i+j}$ this subring is nilpotent.

\begin{Lemma}\label{sqrts}
For $E\in 1+K^0(Y)^1$, the finite sum
\beq{sqdf}
\sqrt E\ :=\ 1-\sum_{i\ge1}a_i(1-E)^{\otimes i}\ \in\ K^0\(Y,\ZZ\)
\eeq
is the unique class in $1+K^0(Y)^1$ which squares to $E$. In particular, $\sqrt{E F}=\sqrt E\sqrt F$.
\end{Lemma}

\begin{proof}
That $(\sqrt E\;)^2=E$ is trivial from \eqref{xhalf} for $n>\dim Y$. For the uniqueness statement, suppose $M\in 1+K^0(Y)^1$ is another class with $M^2=E$. Working in $K^0\(Y,\ZZ\)$ at all times, define $M_n$ by
\beq{ton}
M\=1-\sum_{i=1}^{n-1}a_i(1-E)^{\otimes i}+M_n.
\eeq
Assume inductively that codim$\,M_n\ge n$. (The base case $n=1$ is trivial.)
Squaring gives
\beqa
M^{\otimes2}\=E &=& \!\left(1-\sum_{i=1}^{n-1}a_i(1-E)^{\otimes i}\right)^{\!\otimes2}+2M_n+O(n+1) \\
&\!\stackrel{\eqref{xhalf}}=\!& E+2a_n(1-E)^{\otimes n}+2M_n+O(n+1),
\eeqa
where $O(n+1)$ means a class supported in codimension at least $n+1$. Thus
$$
M_n\=-a_n(1-E)^{\otimes n}+O(n+1),
$$
which by \eqref{ton} shows codim$\,M_{n+1}\ge n+1$. For $n\ge\dim Y$ this shows that $M_{n+1}=0$ and $M=\sqrt E$.
\end{proof}

\begin{Rmk} In particular any line bundle $L$ has a canonical square root
$$
\sqrt L\ \in\ K^0\(Y,\ZZ\),
$$
since $L=1+(L-1)\in 1+K^0(Y)^1$. It is unique in $1+K^0(Y)^1$, so if $M$ is any line bundle on $Y$ such that $M^{\otimes2}=L$ then $M=\sqrt L$ in $K^0\(Y,\ZZ\)$.
\end{Rmk}

The $K$-theory of complex vector bundles is an ``oriented cohomology theory" --- it admits a notion of Chern classes. The $K$-theoretic first Chern class of a line bundle $L$ is
\beq{Kc1}
\c_1(L)\ :=\ 1-L^*\ \in\ K^0(Y)^1,
\eeq
which is the class $[\cO_D]$ of the structure sheaf of any divisor $D\in|L|$. The splitting principle and Whitney sum formula for bundles
\beq{Kc2}
\c(E_1\oplus E_2,t)\=\c(E_1,t)\;\c(E_2,t),\quad\text{where }\,\c(E,t)=1+\sum_{i=1}^{\rk E}t^i\c_i(E),
\eeq
then determine the Chern classes of any object of $K^0(Y)$. In particular the $r$th Chern class $\c_r(E)$ (or Euler class $\mathfrak e(E)$) of a rank $r$ bundle $E$ is
\begin{eqnarray}\nonumber
K^0(Y)^r\ \ni\ \c_r(E) &:=& \Lambda\udot E^*\ :=\ \sum_{i=0}^r(-1)^i\Lambda^iE^* \\
&=& 0_E^*\;0\_{E*}[\cO_Y]\=0_E^*\big[\cO_{\Gamma_{\!s}}\big], \label{cnK}
\end{eqnarray}
where $0_E\colon Y\into E$ is the 0-section, $0_E^*$ is the derived pullback, and $\Gamma_{\!s}\subset E$ is the graph of any section $s\in\Gamma(E)$. If $s$ is transverse to $0_E$ with zero locus $Z$ then this gives $\c_r(E)=[\cO_Z]$. We also have
\beq{crE*}
\c_r(E^*)\ :=\ \Lambda\udot E\=(-1)^{\rk E}\c_r(E)\cdot\det E.
\eeq

\subsection*{Special case} We first define our $K$-theoretic Edidin-Graham class in the presence of a maximal isotropic subbundle.

\begin{Def}\label{KEGdef} Let $(E,q,o)$ be an $SO(2n,\C)$ bundle admitting a maximal isotropic subbundle $\Lambda\subset E$. We define\footnote{In fancier language this is the Atiyah-Bott-Shapiro KO-theory Euler class of the $SO(n,\R)$ bundle $E_\R$ underlying $E$. When $\sqrt{\det\Lambda}$ a genuine line bundle, $\Lambda\udot(\Lambda^*)\cdot\sqrt{\det\Lambda}$ is a spin bundle for $E_\R$ --- as used by Polishchuk-Vaintrob \cite{PV:A} --- and defines an ABS class in KO-theory with integer coefficients.}
\beq{KEG}
\e(E)\ :=\ (-1)^{|\Lambda|}\;\c_n(\Lambda)\cdot\sqrt{\det\Lambda}\ \in\ K^0\(Y,\ZZ\).
\eeq
\end{Def}

\noindent Note that by \eqref{crE*},
\beq{dualeps}
(-1)^{n+|\Lambda|}\c_n(\Lambda^*)\cdot\sqrt{\det\Lambda^*}\=(-1)^{|\Lambda|}\c_n(\Lambda)\cdot\sqrt{\det\Lambda}\=\e(E).
\eeq
So if $E=\Lambda\oplus\Lambda^*$ and we choose $\Lambda^*$ as our maximal isotropic, of sign $(-1)^{|\Lambda|+n}$, we get the same class. More generally we have the following.

\begin{Prop}\label{well} Definition \ref{KEGdef} is well-defined: $\e(E)$ \eqref{KEG} is independent of the choice of $\Lambda\subset E$ and is really a $K$-theoretic square root Euler class:
\beq{squares}
\e(E)^2\=(-1)^n\c_{2n}(E).
\eeq
\end{Prop}

\begin{proof}
The exact sequence $0\to\Lambda\to E\to\Lambda^*\to0$ and \eqref{dualeps} give
$$
\e(E)^2\=(-1)^n\c_n(\Lambda)\;\c_n(\Lambda^*)=(-1)^n\mathfrak e\;(E).
$$

As noted in Footnote \ref{GStriv} the existence of $\Lambda\subset E$ means $(E,q,o)$ is Zariski-locally trivial, so Anderson's $K$-theoretic Edidin-Graham class $\epsilon(E)$ is defined \cite[Appendix B]{And}. To show $\e(E)$ is well-defined we relate it to $\epsilon(E)$, using some identities we learned from \cite[Appendix B]{And}.

From \eqref{Kc1} it is immediate that for any line bundle $L$,
$$
\c\(L,-\tfrac12\)\otimes L\=\c\(L^*,-\tfrac12\).
$$
By the splitting principle and Whitney sum formula \eqref{Kc2} this gives
$$
\c\(\Lambda,-\tfrac12\)\otimes\det\Lambda\=\c\(\Lambda^*,-\tfrac12\).
$$
Combined with the exact sequence $0\to\Lambda\to E\to\Lambda^*\to0$ this gives
$$
\c\(E,-\tfrac12\)\otimes\det\Lambda\=\c\(\Lambda,-\tfrac12\)\c\(\Lambda^*,-\tfrac12\)\otimes\det\Lambda\=\c\(\Lambda^*,-\tfrac12\)^2.
$$
Since both lie in $1+K^0\(Y,\ZZ\)^1$ we deduce that
$$
\sqrt{\c\(E,-\tfrac12\)}\,\sqrt{\det\Lambda}\=\c\(\Lambda^*,-\tfrac12\).
$$
Now Anderson's class is $\epsilon(E):=(-1)^{|\Lambda|}\;\c\(\Lambda^*,-\tfrac12\)\c_n(\Lambda)$ and is independent of $\Lambda$ by \cite[Appendix B, Theorem 3]{And}. Therefore
\beq{usvA}
\e(E)\=\sqrt{\c\(E,-\tfrac12\)^{-1}}\,\epsilon(E)
\eeq
is also independent of $\Lambda$.
\end{proof}

\subsection*{General case} Fix any \'etale locally trivial $SO(2n,\C)$ bundle $(E,q,o)$. We work, as usual, on the bundle $\rho\colon\wt Y\to Y$ of \eqref{cover}, where $\rho^*E$ admits a canonical positive maximal isotropic $\Lambda_\rho\subset\rho^*E$. Thus $\e(\rho^*E)$ is defined by Proposition \ref{well}. We show it is the pullback by $\rho^*$ of a class we will define to be $\e(E)$.

\begin{Prop}\label{welld}
$\e(\rho^*E)=\rho^*\rho_*\(\e(\rho^*E)\)$. Furthermore, if $E$ admits a maximal isotropic then $\e(E)$ defined by \eqref{KEG} equals $\rho_*\e(\rho^*E)$.
\end{Prop}

\begin{proof}
We use the Cartesian diagram of flat maps
$$
\xymatrix@=10pt{
& \wt Y\times\_Y\wt Y \ar[dl]_-{\rho\_1}\ar[dr]^-{\rho\_2} \\
\wt Y \ar[dr]_{r\_2}&& \wt Y \ar[dl]^{r\_1} \\ & \,Y,\!}
$$
where $r\_1=\rho=r\_2$. Let $P\colon\wt Y\times\_Y\wt Y\to Y$ be $r\_2\circ\rho\_1=r\_1\circ\rho_2$. Then $P^*E$ has two positive maximal isotropic subbundles $\rho_1^*\Lambda_\rho$ and $\rho_2^*\Lambda_\rho$, so by Proposition \ref{well},
$$
\c_n(\rho_1^*\Lambda_\rho)\cdot\sqrt{\det(\rho_1^*\Lambda_\rho)}\=
\c_n(\rho_2^*\Lambda_\rho)\cdot\sqrt{\det(\rho_2^*\Lambda_\rho)}
$$
in $K^0\(\wt Y\times_Y\wt Y,\ZZ\)$. Equivalently,
\beq{12}
\rho_1^*\(\e(\rho^*E)\)\=\rho_2^*\(\e(\rho^*E)\).
\eeq
Since $\rho\_1$, like $\rho$, is an iterated bundle of smooth quadrics, we have $\rho\_{1*\;}\rho_1^*\cO_{\wt Y}=\cO_{\wt Y}$ in the derived category and so also in $K$-theory. (As usual we use \emph{derived} pushforward.) Therefore by the projection formula, \eqref{12} and flat basechange,
$$
\e(\rho^*E)\=\rho\_{1*}\rho_1^*\(\e(\rho^*E)\)\=\rho\_{1*}\rho_2^*\(\e(\rho^*E)\)\=r_2^*r\_{1*}\(\e(\rho^*E)\).
$$
That is, $\e(\rho^*E)=\rho^*\rho_*\(\e(\rho^*E)\)$, as required.\medskip

If $\Lambda\subset E$ is a maximal isotropic subbundle then by Proposition \ref{well} the class $\e(\rho^*E)$ --- defined by $\Lambda_\rho$ --- is the same as the class \eqref{KEG} defined by $\rho^*\Lambda$. That is, $\e(\rho^*E)=\rho^*\e(E)$. Applying the projection formula gives $\rho_*\e(\rho^*E)=\e(E)$.
\end{proof}

It follows that the following is well-defined, and gives the same as Definition \ref{KEGdef} when $E$ admits a maximal isotropic.

\begin{Def}\label{KEGdfn} For an $SO(2n,\C)$ bundle $(E,q,o)$ define
$$
\e(E)\ :=\ \rho_*\(\e(\rho^*E)\)\ \in\ K^0\(Y,\ZZ\).
$$
By \eqref{squares} this satisfies $\e(E)^2=(-1)^n\c_{2n}(E)$.
\end{Def}

We note the curious contrast with the cohomological Edidin-Graham class, where we had to invert 2 in order to descend from $\wt Y$ to $Y$. In $K$-theory the descent works over the integers, but we had to invert 2 earlier to define the class \eqref{KEG} even in the presence of a maximal isotropic subbundle.

\subsection{Localisation by an isotropic section}
Fix an isotropic section $s$ of our $SO(2n, \CC)$-bundle $(E,q,o)$, with
zero scheme $i\colon Z(s)\into Y$. We will construct a \emph{localised} $K$-theoretic square root Euler class
$$
\e(E, s)\ \colon\, K_0\(Y,\ZZ\)\To K_0\(Z(s),\ZZ\)
$$
whose pushforward $i_*\circ\e(E,s)$ is tensor product with $\e(E)\in K^0\(Y,\ZZ\)$.

\subsection*{Special case} To begin with we suppose that $E$ admits a maximal isotropic $\Lambda\subset E$. We follow closely what we did in Chow in Section \ref{lociso}, adapting to $K_0$. Via the exact sequence
$$
0\To\Lambda\To E\rt\pi\Ld\To0
$$
we set $s^*:=\pi(s)\in\Gamma(\Ld)$. Then by \eqref{cnK} we have
$$
\c_n(\Ld)\=0^*_{\Lambda\raisebox{1pt}{\ast\,}}\cO\_{\Gamma_{\!s\ast}\,}\ \in\ K^0(Y).
$$
By deformation to the normal cone we can deform the graph $\Gamma_{\!s\ast}\subset\Ld$ to its linearisation $C_{Z^*/Y}\subset\Ld|_{Z^*}$ about the zero locus $j\colon Z^*\into Y$ of $s^*$. Using the $K$-theoretic specialisation map $K_0(Y)\to K_0(C_{Z^*/Y})$ \cite[p\,352]{F} we get the localisation
$$
\xymatrix{
K_0(Y) \ar[r]& K_0(C_{Z^*/Y}) \ar[rr]^-{0_{\Lambda\ast|_{Z\ast}}^*}&& K_0(Z^*)},
$$
whose pushforword to $K_0(Y)$ is $\c_n(\Lambda^*)\;\otimes\ $.

%Therefore, after passing from $K^0$ to $K_0$, we deduce that the structure sheaves of $\Gamma_{\!s\ast}$ and $C_{Z\ast/Y}$ have the same class, giving the localisation
%\beq{interm}
%\xymatrix{K_0(Z^*)\ \ni\ \(0_{\Lambda\ast|_{Z\ast}}\)^*\,\cO\_{C_{Z\ast/Y}\,}\ar@{|->}[r]^-{j_*}& \ \c_n(\Ld)\ \in\ K_0(Y).}
%\eeq
%%$$
%%\c_n(\Ld)\=j_*\Big((0_{\Lambda\ast|_{Z\ast}})^*\,\cO\_{C_{Z\ast/Y}\,}\Big)\ \in\ K_0(Y).
%%$$
To localise further to $Z(s)\subset Z^*$ by using the ``\emph{other half}" of the section $s$, we use Kiem-Li's $K$-theoretic cosection localisation \cite{KL:K}. By Lemma \ref{Lem:Zero} the cosection $\wt s$ \eqref{function} defined by $s$ is zero on $C_{Z^* / Y}\subset\Ld|_{Z^*}$, so \cite[Theorem 4.1]{KL:K} defines the arrow across the top of the commutative diagram
%$$
%0^{\;!,\,\mathrm{loc}}_{\Ld,\,\wt s}\ \colon\, K_0\(C_{Z^*/Y}\)\To K_0\(Z(s)\)
%$$
\beq{eq1}
\xymatrix@R=18pt@C=30pt{
K_0(Y) \ar[r]& K_0\(C_{Z^*/Y}\) \ar[r]^{0^{*,\,\mathrm{loc}}_{\Lambda\ast,\,\wt s}}\ar[d]& K_0\(Z(s)\) \ar[d]\ar[dr]^{i_*} \\
& K_0\(\Ld\big|_{Z^*}\) \ar[r]^{0^*_{\Lambda\ast}}& K_0(Z^*) \ar[r]^{j_*}& K_0(Y)}
\eeq
in which the vertical maps are the obvious pushforwards. Composing the first two arrows on the top row defines an operator
$$
\c_n(\Ld,s)\ \colon\,K_0(Y)\To K_0(Z(s))
$$
%$$
%\c_n(\Ld,s)\ :=\ 0^{\;!,\,\mathrm{loc}}_{\Lambda\ast,\,\wt s}\,\big[\cO\_{C_{Z\ast/Y}}\big]\ \in\ K_0(Z(s))
%$$
such that $i_*\circ\c_n(\Ld,s)$ is tensor product with $\c_n(\Ld)\in K^0(Y)$. Thus we get a localised $K$-theoretic Edidin-Graham operator
\begin{align}\label{eq2}
\e(E,s,\Lambda)\ :=\ (-1)^{n+|\Lambda|}&\c_n(\Ld,s)\cdot\sqrt{\det\Lambda^*}
\end{align}
from $K_0\(Y,\ZZ\)\to K_0\(Z(s),\ZZ\)$
such that $i_*\circ\e(E,s,\Lambda)$ is tensor product with $\e(E)\in K^0\(Y,\ZZ\)$.

\subsection*{General case} We can now define the localised $K$-theoretic Edidin-Graham class in general by using the cover $\rho\colon\wt Y\to Y$ \eqref{cover}.

\begin{Def}\label{def1} Given an isotropic section $s\in\Gamma(E)$ of an $SO(2n,\C$) bundle $(E,q,o)$ we define the localised operator
$$
\e(E,s)\ :=\ \rho_*\(\e(\rho^*E,\rho^*s,\Lambda_\rho)\)\ \colon\, K_0\(Y,\ZZ\)\To K_0\(Z(s),\ZZ\).
$$
\end{Def}

By construction its pushforward to $Y$ is tensor product with $\rho_*\e(\rho^*E)\in K^0(Y)$, so by Proposition \ref{welld} and Definition \ref{KEGdfn} it follows that
\beq{pushK}
i_*\circ\e(E,s)\=\e(E)\;\otimes\ \colon\,K_0\(Y,\ZZ\)\To K_0\(Y,\ZZ\).
\eeq

This allows us to define a $K$-theoretic square root version of the intersection between an isotropic cone $C\subset E$ (supported over $Z\subset Y$) and the zero section $0_E\colon Y\into E$. As usual we let $\tau\_E$ denote the tautological (isotropic) section of $\pi^*E$ on $\pi\colon C\to Y$. Its zero locus is $Z$.

\begin{Def}\label{def2} Given an isotropic cone $C\subset E$ we define
$$
\sqrt{0_E^*}\ :=\ \e(\pi^*E,\tau\_E)\ \colon\,K_0\(C,\ZZ\)\To K_0\(Z,\ZZ\).
$$
\end{Def}

There are obvious $K$-theoretic versions of Lemma \ref{glbl},
\beqa 
\sqrt{0_E^*}\=\overline\pi_*\big[\e\(\overline E\)\otimes(\ \cdot\ ) \big],
\eeqa
and --- when $C$ lies in a maximal isotropic subbundle $\Lambda\subset E$ --- Lemma \ref{Schwantz},
\beq{Klam}
\sqrt{0_E^*}\=(-1)^{|\Lambda|}\sqrt{\det\Lambda}\cdot0_\Lambda^*.
\eeq
The proofs are almost identical on replacing each $e\cap(\ \cdot\ )$ by $\mathfrak e\;\otimes(\ \cdot\ )$, equation \eqref{oiler} by $0\_{\overline\Lambda*}0^*_{\overline\Lambda*}=\Lambda\udot(T_{\overline\pi}(-1))^*\otimes(\ \cdot\ )$ and using \eqref{pushK} in place of \eqref{push}.

Similarly the proof of Lemma \ref{shr} goes through without change to give
\beq{Kshriek}
f^{\;!}\sqrt{0_E^*}\=\sqrt{0_{f'^*E}^*}\ f^{\;!}\ \colon K_0\(C,\ZZ\)\To K_0\(X',\ZZ\).
\eeq
Here $f^{\;!}$ is the $K$-theoretic refined Gysin map of \cite[Section 3]{AP} and \cite[Section 2.1]{YP}. By construction this commutes with tensor product and product with $K$-theoretic Chern classes. It also commutes with proper pushforward by \cite[Lemma 3.1]{AP}, so the result $\overline\pi'_*f^{\;!}=f^{\;!}\overline\pi_*$ used in the proof of Lemma \ref{shr} also holds in $K$-theory.
%$$
%f^{\;!}\overline\pi_*\big[\mathfrak e(\Lambda)\cap(\ \cdot\ )\big]\=
%\overline\pi'_*\big[\mathfrak e(\overline{f''}^*\Lambda)\cap f^{\;!}(\ \cdot\ )\big],
%$$

\subsection{Virtual structure sheaf}
Let $M$ be a quasi-projective $(-2)$-shifted symplectic derived scheme with an orientation. By Proposition \ref{form} its obstruction theory admits a self-dual resolution $E\udot=\{T\to E\to T^*\}\to\LL_M$. By Proposition \ref{por}, $E$ is an $SO(r,\C)$ bundle, where $r$ has the same parity as the virtual dimension $\vd$ \eqref{vd}. If this is odd we define the virtual structure sheaf to be zero. Suppose now it is even. By \eqref{recipe} and Proposition \ref{isoprop} we get an isotropic cone $C_{E\udot}\subset E$.

\begin{Def}\label{vss} We define the twisted virtual structure sheaf of $M$ to be
\beq{Ohatdef}
\Ohat\ :=\ \sqrt{0_E^*}\,\big[\cO_{C_{E\udot}}\big]\cdot\sqrt{\det T^*}\ \in\ K_0\(M,\ZZ\).
\eeq
\end{Def}

Proposition \ref{independence} below shows the twist by $\sqrt{\det T^*}$ ensures independence from the choice of self-dual complex $E\udot$. (Since the other terms in \eqref{Ohatdef} involve only the stupid truncation $\{E\to T^*\}$ of $E\udot$, we should expect some contribution from the omitted term $T$, and this turns out to be it.)

In special situations --- one being the local Calabi-Yau 4-fold case of Section \ref{KY} --- $E$ admits a maximal isotropic subbundle $\Lambda\subset E$ such that $\At\colon E\udot\to\LL_M$ factors through the following quotient of $E\udot$,
\beq{BFhalf}
\tfrac12E\udot\ :=\ \{\Lambda^*\To T^*\}.
\eeq
That is, in this situation $\{T\To\Lambda\}$ can be though of as a virtual tangent bundle for $M$; taking determinants gives
$$
K^{\vir}_M\ :=\ \det T^*\otimes\det\Lambda.
$$
Then $C_{E\udot}$ lies in $\Lambda\subset E$ and, by \eqref{Klam}, the class \eqref{Ohatdef} can be written in the perhaps more suggestive form
\beq{NOtwist}
(-1)^{|\Lambda|}\;0_\Lambda^*\big[\cO_{C_{E\udot}}\big]\cdot\sqrt{K_M^{\vir}}
\eeq
in this situation. 
The last term is the ``Nekrasov-Okounkov twist" \cite{NO} that has proved profitable to use in $K$-theoretic DT$^3$ theory. By contrast it is \emph{necessary} in DT$^4$ theory --- only the \emph{twisted} virtual structure sheaf is well-defined, as was already anticipated in the papers \cite{Ne, NP, CKM1}.

\begin{Prop}\label{independence}
Definition \ref{vss} is independent of the choice of self-dual resolution $E\udot$ of the obstruction theory $\EE$. If $M$ is projective then $\Ohat$ is deformation invariant: in the setting of \eqref{defint} we have $\widehat\cO^{\;\vir}_{\!M_t}=\iota_t^!\;\Ohat$.
\end{Prop}

\begin{proof}
The proof follows that of the cycle version \eqref{wont} very closely. That required invariance under three moves.
\begin{enumerate}
\item The deformation invariance \eqref{defo} was proved by \cite[Theorem 5.2]{KL}. The precise $K$-theoretic analogue is given by \cite[Proposition 5.5]{KL:K}.
\item The invariance \eqref{BBB} under the addition of the acyclic complex $B\udot$ \eqref{Badd} to $E\udot$ was proved by Proposition \ref{E1E2} and Lemma \ref{glbl}.
Replacing $e$ by $\mathfrak e$ and $\surd\;0_E^{\;!}$ by $\surd\;0_E^*$ in the proof of Proposition \ref{E1E2} gives the precise $K$-theoretic analogue. Lemma \ref{Schwantz} has $K$-theoretic analogue \eqref{Klam}, which introduces the twist $\sqrt{\det B}$. But this is cancelled by the $\sqrt{\det T^*}$ twist in \eqref{Ohatdef}, since we have replaced $T^*$ by $T^*\oplus B^*$.
\item Finally, we used Proposition \ref{KperpK} to pass from the deformation complex $E\udot$ made from the orthogonal bundle $E$ to the complex $A\udot_e$ made from $A_e=K^\perp/K$ as in \eqref{cones}. 
\end{enumerate}
To get the $K$-theoretic analogue of the last step we first replace $e$ by $\mathfrak e$ in \eqref{Kred}. That is, fix an isotropic subbundle $K\subset E$ and orient $K^\perp/K$ as described in \eqref{reduc}. Then, in the notation of \eqref{17}, the exact sequence $0\to K\to\Lambda\to\Lambda_\rho\to0$ gives
$$
\Lambda\udot(\Lambda^*)\otimes\sqrt{\det\Lambda}\=\Lambda\udot(\Lambda^*_\rho)\otimes\sqrt{\det\Lambda_\rho}\otimes\Lambda\udot(K^*)\otimes\sqrt{\det K},
$$
which is
\beq{K3}
\e(E)\=\e\(K^\perp/K\)\;\mathfrak e(K)\cdot\sqrt{\det K}.
\eeq
Then the proof of Proposition \ref{KperpK} goes through as before on replacing
$e$ by $\mathfrak e$ and $\surd\;0_E^{\;!}$ by $\surd\;0_E^*$ to give
\beq{KKperpK}
\sqrt{\det K}\cdot\sqrt{0_{K^\perp/K}^*}\=\sqrt{0_E^*}\,\circ p^*.
\eeq
Again the $\sqrt{\det T^*}$ twist in \eqref{Ohatdef} cancels the $\sqrt{\det K}$ since in passing from $E\udot$ to $A_e\udot$ we replaced $T^*$ by $A^0=(T/K)^*$.
\end{proof}

%The extensions of $[M]^{\vir}$ of Section \ref{genrul} all apply equally to $\Ohat$.
%\begin{itemize} \item If $H^2(\cO_X)\ne0$ then $\Ohat=0$ but we may define a reduced virtual structure sheaf $\widehat\cO^{\;\vir}_{\!M_L}$ on the moduli space $M_L$ of stable rank $r>0$ sheaves of fixed determinant $L$.
%\item We may allow $X$ to be quasi-projective at the expense of losing deformation invariance. \item We may allow $M$ to be a moduli space of stable pairs or Joyce-Song pairs instead of a moduli space of stable sheaves.
%\end{itemize}

\begin{Rmk} It is natural to ask if there are analogues of Siebert's formula  \cite[Theorem 4.6]{Sie} for the virtual cycle --- or the $K$-theoretic analogue \cite[Theorem 4.2]{Th:K} for the virtual structure sheaf --- in this square-rooted setting. When the cone $C_{E\udot}$ is contained in a maximal isotropic subbundle $\Lambda\subset E$ we are in a Behrend-Fantechi setting (cf. \eqref{BFhalf}), so we can deduce
\begin{align*}
[M]^{\vir}&\=(-1)^{|\Lambda|}\left[ c\(\Lambda - T \) \cap c\_F\(M\) \right]_{\frac12\!\vd}, \\
\Ohat&\=(-1)^{|\Lambda|}\sqrt{\det\Lambda}\cdot\sqrt{\det T^*}\cdot\left[\Lambda\udot\(\Lambda^*-T^*\)\otimes\Lambda_M\right]_{\t=1}.
\end{align*}
Here $c\_F(M)\in A_*(M)$ is Fulton's class \cite[Example 4.2.6]{F} while $\Lambda_M\in K_0(M)[\t]$ is the $K$-theoretic analogue of \cite{Th:K}. In general the situation is more complicated.
\end{Rmk}
%
%For a scheme with the perfect obstruction theory $(Y,F\udot)$ Siebert \cite[Theorem 4.6]{Si} proved
%$$
%[Y]^{\vir} \ = \ \left\{ c\((F\udot)^\vee\)^{-1} \cap c_F\(Y\) \right\}_{\vd} \ \in\ A_{\vd}\(Y,\Z\)
%$$
%where $c_F\;(\, \cdot\, )$ denotes the Fulton's canonical class \cite[Definition 4.3]{Si}, $\vd=\rank F\udot$ and $\{\,\cdot\,\}_d$ denotes the degree $d$ part.
%
%Similarly for our moduli space $M$ with the obstruction theory \eqref{Kdot} one may ask if we have
%$$
%[M]^{\vir} \ = \ (-1)^{|\Lambda|}\left\{ c\(\Lambda - T \) \cap c_F\(M\) \right\}_{\frac12\!\vd} \ \in\ A_{\frac12\!\vd}\(M,\ZZ\)
%$$
%for a maximal isotropic subbundle $\Lambda$ of $E$. Unfortunately this holds {\em only when} the cone $C_{E\udot}\subset E$ is contained in $\Lambda$. However using the intersection $C_{E\udot}\cap\Lambda$ in general, we can modify the formula. A short note on this is being prepared. }

\section{Virtual Riemann-Roch}\label{RRsec}

Using rational coefficients, in this Section we will relate $\Ohat$  \eqref{Ohatdef} in $K$-theory to $[M]^{\vir}$ \eqref{virMdef} in Chow homology via the isomorphism
\beq{taumap}
\tau\_M\ \colon\,K_0(M)_\Q\rt\sim A_*(M)_\Q
\eeq
that holds for any quasi-projective scheme $M$ \cite[Corollary 18.3.2]{F}. Here $\tau\_M=\ch(\ \cdot\ )\cap\(\td(T_M)\cap[M]\)$ when $M$ is smooth. More generally choose an embedding $i\colon M\into P$ in a smooth variety and set
$$
\tau\_M(F)\ :=\ \ch^P_M(F\udot)\cap\(\td(T_P)\cap[P]\)\ \in\ A_*(M)_\Q,
$$
where $F\udot\to i_*F$ is a locally free resolution and $\ch^P_M(F\udot)\in A_*(M\to P)_\Q$ is its localised Chern character of \cite[Theorem 18.1]{F}.

First we recall the Riemann-Roch theorems of Fantechi-G\"ottsche \cite[Lemma 3.5]{FG}
and Ciocan-Fontanine-Kapranov \cite[Theorem 4.4.1]{CFK}. Suppose $M$ has a perfect obstruction theory $E\udot\to\LL_M$ with dual the virtual tangent bundle $E_\bullet$. We get a Behrend-Fantechi virtual cycle $[M]^{\vir}:=0^{\;!}_{E_1}[C_{E\udot}]$ and virtual structure sheaf $\cO^{\vir}_M:=0^*_{E_1}\big[\cO_{C_{E\udot}}\big]$. They prove
\beq{FGCK}
\tau\_M\(\cO^{\vir}_M\)\=\td(E_\bullet)\cap[M]^{\vir}.
\eeq

\medskip In our setting we fix instead a 3-term self-dual obstruction theory
$$
E\udot\=\big\{T\to E\to T^*\big\}\To\LL_M
$$
of even rank and with a fixed orientation. From this we get the virtual cycle $[M]^{\vir}:=\surd\;0^{\;!}_E\,[C_{E\udot}]$ of \eqref{virMdef} and the twisted virtual structure sheaf 
$\Ohat:=\surd\;0_E^*\,\big[\cO_{C_{E\udot}}\big]\cdot\surd\!\det T^*$
of \eqref{Ohatdef}. Since any Todd class $\td\in1+K^0(M)^1$ its square root is uniquely defined and multiplicative as in Lemma \ref{sqrts}. The analogue of \eqref{FGCK} is the following.

\begin{Thm}\label{virRR}
In $A_*(M)_\Q$ we have $\tau\_M\(\Ohat\)\,=\,\sqrt\td(E_\bullet)\cap[M]^{\vir}$.
\end{Thm}

\begin{proof}
We first recall the definition \eqref{Ohatdef} of $\Ohat$. Using the stupid truncation $\{E\to T^*\}$ of $E\udot$ as a perfect obstruction theory for $M$ gives an isotropic cone $\pi\colon C_{E\udot}\to M$ in the $SO(2n,\C)$ bundle $E$ by Proposition \ref{isoprop}. Unravelling Definitions \ref{def2} and \ref{def1} and equations (\ref{eq2}, \ref{eq1}) we see $\Ohat$ is the image of $(-1)^n[\cO_{C_{E\udot}}]\cdot\pi^*\surd\!\det T^*$ along the top row of the diagram
$$
\xymatrix@C=9pt@R=20pt{
K_0(C_{E\udot}) \ar[d]_\tau\ar[rrr]^{\rho^*}&&& K_0(\wt C) \ar[d]_\tau\ar[rr]&& K_0\(C_{Z^*/\wt C}\) \ar[d]_\tau\ar[rrr]^{\sqrt{\det\Lambda^*_\rho}\,\cdot}_-{0^{*,\,\mathrm{loc}}_{\Lambda\!\raisebox{1pt}{\ast}_{\!\!\rho},\,\wt\tau|_{Z\ast}}}&&& K_0(\wt M) \ar[d]_\tau\ar[rr]^-{\rho_*}&& K_0(M) \ar[d]_\tau \\
A_*(C_{E\udot}) \ar[rrr]^-{\td(T_\rho)\cdot\rho^*}&&& A_*(\wt C) \ar[rr]&& A_*\(C_{Z^*/\wt C}\) \ar[rrr]^{\sigma\cap}_-{0^{\;!,\,\mathrm{loc}}_{\Lambda\!\raisebox{1pt}{\ast}_{\!\!\rho},\,\wt\tau|_{Z\ast}}}&&& A_*(\wt M) \ar[rr]^-{\rho_*}&& A_*(M).\!}
$$
We explain the notation and maps. We use $\rho$ to denote any basechange of the flat map $\rho\colon\wt M\to M$ of \eqref{cover} such as $\wt C:=\rho^*C_{E\udot}\to C_{E\udot}$. Recall $\Lambda_\rho$ denotes the canonical maximal isotropic in the pullback of $E$. We suppress some pullback maps for clarity, so $\Lambda_\rho\subset E$ makes sense on $\wt M$ or $\wt C$ (which one should be clear from the context). Since $\wt C$ is a subscheme of (the pullback of) $E$, it inherits a tautological section $\wt\tau$ of $E$. Projecting this to $E/\Lambda_\rho\cong\Lambda_\rho^*$ gives the section $\wt\tau^*$ whose zero locus we denote $Z^*\subset\wt C$. The second horizontal arrows are the specialisation maps defined by the deformation of $\wt C$ to the normal cone of $Z^*\subset\wt C$,
\beqa %{defnC}
\xymatrix{\wt C\ \ar@{~>}[r]& \ C_{Z^*\!\;/\;\wt C}\,.}
\eeqa
This normal cone embeds in (the pullback of) $\Lambda_\rho^*|_{Z^*}$ on which $\wt\tau$ factors through $\Lambda_\rho\subset E$ and so defines a cosection $\wt\tau|_{Z^*}\colon\Lambda_\rho^*|_{Z^*}\to\cO_{Z^*}$. The third horizontal arrows use this to define the pullback $0^*_{\Lambda\raisebox{1pt}{\ast}_{\!\!\!\rho}}$\vspace{-1mm} and intersection $0^{\;!}_{\Lambda\raisebox{1pt}{\ast}_{\!\!\!\rho}}$ cosection-localised to the zero locus $\wt M$ of $\wt\tau|_{Z^*}$ by \cite{KL:K} and \cite{KL} respectively. Finally, we use $\Q$ coefficients throughout, and the vertical maps are the $\tau$ maps of \eqref{taumap}.

The first square of the diagram commutes by \cite[Theorem 18.2(3)]{F} (since $\rho$ is flat), the second by \cite[Example 18.3.8]{F} and the fourth by the Grothendieck-Riemann-Roch theorem of \cite[Theorem 18.2(1)]{F} (since $\rho$ is proper). By \cite[Equation 5.21]{KL:K} the third commutes on setting
$$
\sigma\ :=\ \ch\sqrt{\det\Lambda^*}\!\!\!\;_\rho\ \td(-\Lambda^*_\rho).
$$
Since $\td(F^*)=\td(F)\;\ch(\det F^*)$ for any bundle $F$ we have the identity
\beq{FF*}
\sqrt{\td}(F\oplus F^*)=\td(F)\,\ch\sqrt{\det F^*}.
\eeq
It follows that
$$
\sigma\=\(\sqrt\td\,E\)^{-1},
$$
so it is pulled back from $M$.

By a similar unravelling of Definitions \ref{vdef} and \ref{IsoCone} and equations (\ref{rhoh}, \ref{eLam}), we get $[M]^{\vir}$ by starting with\footnote{The twist by $\sqrt\td(E)$ is there to cancel the $\sigma$ term.} $(-1)^n\sqrt\td(E)\cap\rho^*[C_{E\udot}]$ in the second group $A_*(\wt C)$ and moving along the bottom row of the diagram as far as the fourth group $A_*(\wt M)$. Then we cap with the class $\frac1{2^{n-1}}h$ of \eqref{hdef} before applying $\rho_*$ to give $[M]^{\vir}\in A_*(M)$. This is the same as starting with
\beq{startin}
(-1)^n\sqrt\td(E)\cap[C_{E\udot}]
\eeq
in the first group and moving along to the last one, because $\td(T_\rho)$ plays the same role as $\frac1{2^{n-1}}h$ on classes pulled back by $\rho^*$:
$$
\rho_*\(\td(T_\rho)\cap\rho^*a\)\=a\=\frac1{2^{n-1}}\rho_*\(h\cap\rho^*a\).
$$
Here we have used the projection formula and $\rho_*\;\td(T_\rho)=1$ by Grothendieck-Riemann-Roch, and we are applying the formula to $(-1)^n\;0^{\;!,\,\mathrm{loc}}_{\Lambda\raisebox{1pt}{\ast}_{\!\!\!\rho},\wt\tau|_{Z\ast}}\big[C_{Z^*/\wt C}\big]$, which is $\rho^*[M]^{\vir}$ by Lemma \ref{newlem}.

So we now chase $(-1)^n[\cO_{C_{E\udot}}]\surd\!\det T^*$ through the diagram. Applying the key result \cite[Proposition 3.1]{FG} to the stupidly truncated perfect obstruction theory $\tau E^\bullet$ gives
\beqa %{FGresult}
\tau\_{C_{E\udot}}\big[\cO_{C_{E\udot}}\big]\=\pi^*\td(T)\cap[C_{E\udot}],
\eeqa
Combined with the module property \cite[Theorem 18.2(2)]{F} of $\tau$ this shows the first vertical map gives
$$
(-1)^n\ch\(\sqrt{\det T^*}\,\)\;\td(T)[C_{E\udot}]\ \in\ A_*(C_{E\udot}).
$$
This is $\ch\(\sqrt{\det T^*}\,\)\;\td(T)\sqrt\td(E)^{-1}\cap$\eqref{startin}, so across the bottom of the diagram it maps to
\[
\ch\(\sqrt{\det T^*}\,\)\;\td(T)\sqrt\td(E)^{-1}\cap[M]^{\vir}\ \in\ A_*(M).
\]
Applying \eqref{FF*} to $F=T$ shows this is
\[
\sqrt\td(T\oplus T^*)\sqrt\td(E)^{-1}\cap[M]^{\vir}\=\sqrt\td(E_\bullet)\cap[M]^{\vir}. \qedhere
\]
\end{proof}

%Applying $\rho^*$ to \eqref{FGresult}, both sides extend naturally to the total space of the deformation to the normal cone \eqref{defnC}; on the left we take $\tau$ of the structure sheaf of the total space, on the right the pullback of $\td(T)$ capped with the fundamental class of the total space. Then $\rho^*$\eqref{FGresult} says the resulting classes are equal after applying the refined Gysin morphism $\iota_t^!$ (where $\iota_t\colon\{t\}\into\C$ is the point $t$ of the base of the deformation to the normal cone). Thus by \cite[Proposition 11.1]{F} they are also equal after applying $\iota_0^!$. Since $\tau$ commutes with specialisation \cite[Example 18.3.8]{F} this gives
%$$
%\tau\_{C_{Z^*/\wt C}}\big[\cO_{C_{Z^*/\wt C}}\big]\=\td(T)\cap\big[C_{Z^*/\wt C}\big].
%$$

\section{Torus localisation}\label{torus}
Let $\T:=\C^*$ be the one dimensional algebraic torus. 
If $\T$ acts on a scheme $Y$ then it is elementary to do everything in this paper in $\T$-equivariant Chow (co)homology, for instance by replacing $Y$ by the $Y$-bundle $Y\times\_\T(\C^{N+1}\take\{0\})$ over the finite dimensional approximation $\PP^N$ to the classifying space $B\T$ (and taking the limit as $N\to\infty$). In particular we have $\T$-equivariant versions of the (localised) Edidin-Graham classes and square root Gysin operators, satisfying the identities proved in Section \ref{lEG}.

So now suppose $\T$ acts on a quasi-projective Calabi-Yau 4-fold $(X,\cO_X(1))$ \emph{preserving the holomorphic 4-form}. Thus it acts on any moduli space  $M$ of compactly supported sheaves --- or stable pairs, or Joyce-Song pairs --- on $X$ by pull back of sheaves. There is a lifting of the $\T$ action to the (twisted) universal sheaf $\cE$ by \cite[Proposition 4.2]{Ri}. Thus the complex of (untwisted) sheaves $\EE$ \eqref{ObsTh} is also $\T$-equivariant, as is the obstruction theory $\At\colon\EE\to\LL_M$ given by the Atiyah class of $\cE$ \cite[Theorem 4.3]{Ri}. The orientation on $\EE$ is a $\Z/2$ choice; since $\T$ is connected it preserves it.

Since each step of Proposition \ref{form} (resolution, truncation, Serre duality, etc) can be done $\T$-equivariantly, we get a $\T$-equivariant self-dual 3-term complex of locally free sheaves $E\udot$ and a map in $D(M)$,
\beq{surjec}
E\udot\=\big\{T\to E\to T^*\big\}\To\LL_M
\eeq
resolving $\At\colon\EE\to\LL_M$, with $E=(E,q,o)$ a $\T$-equivariant $SO(2n,\C)$ bundle.\footnote{We assume that $\vd$ \eqref{vd} is even since otherwise $[M]^{\vir}=0$.} The cone $C_{E\udot}\subset E$ is also $\T$-equivariant, so we get a lift of the virtual cycle \eqref{virMdef} to equivariant homology,
\beqa 
[M]^{\vir}\ :=\ \sqrt{0^{\;!}_E}\,\big[C_{E\udot}\big]\ \in\ A_{\frac12\!\vd}^{\T}\(M,\ZZ\).
\eeqa
The restriction of $E\udot$ to $\iota\colon M^\T\into M$ splits into fixed and moving parts,
$$
\iota^*E\udot\=E\udot_f\,\oplus\,(N^{\vir})^\vee,
$$
which are also self-dual 3-term complexes of locally free sheaves,
\begin{eqnarray}\nonumber
E\udot_f &=& \big\{T^f\to E^f\to(T^f)^*\big\}, \\ \label{efemtftm}
(N^{\vir})^\vee &=& \big\{T^m\to E^m\to(T^m)^*\big\}.
\end{eqnarray}
The duality on $N^{\vir}$ pairs weights $w>0$ with weights $-w<0$, so $r:=\rk N^{\vir}$ is even\footnote{It is important to note --- especially when interpreting the formulae (\ref{fixvc}, \ref{eNvir}) --- that $r$ may not be constant but can vary from one connected component of $M^\T$ to another.} and $\det N^{\vir}$ is trivial. (We thank Davesh Maulik for this observation.) Thus it is orientable; choosing one of its two orientations means the orientation on $E\udot$ induces one on $E\udot_f$. Though it does not matter which we choose (ultimately the signs will cancel in the localisation formula), a canonical choice is to apply the convention \eqref{oTis} to the \emph{positive} weight subbundles of $(E^m)^*,\,T^m$ \emph{and} $(T^m)^*$. That is, we recall that by \eqref{theta=Q} an orientation on $N^{\vir}$ is the same as an orientation on
$$
T^m\oplus(T^m)^*\oplus(E^m)^*\=N^{>0}\oplus(N^{>0})^*,
$$
where $N^{>0}:=T^{>0}\oplus(T^{<0})^*\oplus(E^{<0})^*$. We use $o\_{N^{>0}}$ \eqref{oTis}, with respect to which $N^{>0}$ is a positive maximal isotropic subbundle. Therefore the maximal isotropic subbundle
$$
T^m\oplus0\oplus(E^{<0})^*\ \subset\ T^m\oplus(T^m)^*\oplus(E^m)^*
$$
has sign $(-1)^{\rk T^{<0}}$. So applying Proposition \ref{por} to $o\_{N^{>0}}$ induces the orientation
$$
(-1)^{\rk T^{<0}}o\_{E^{>0}}\ \text{ on }\ E^m\,=\,E^{>0}\oplus(E^{>0})^*
$$
with respect to which the maximal isotropic subbundle $(E^{<0})^*=E^{>0}\subset E^m$ the sign\footnote{This fixes a sign error in the published version of this paper spotted by Arkadij Bojko in the course of his work \cite{Bo2} on the Calabi-Yau fourfold wall-crossing formula. It also shows that his resolution-free substitute $e\_{\T}\((N^{\vir})^{>0}\)=e\_{\T}\(T^{>0}\oplus(T^{<0})^*\)/e\_{\T}(E^{>0})$ for $\sqrt e\_{\T}(N^{\vir})$ indeed equals \eqref{eNvir}.\label{ABo}} $(-1)^{\rk T^{<0}}$. 

By \cite[Proposition 1]{GP} the induced map $E\udot_f\to\LL_{M^\T}$ is an obstruction theory. In Footnote \ref{F28} below we observe that $C_{E_f\udot}$ is contained in the fixed part of $C_{E\udot}|_{M^\T}$, which is isotropic in $E|_{M^\T}$. Thus $C_{E_f\udot}$ is isotropic in $E^f$. (Alternatively one could prove that the fixed part of the $(-2)$-shifted symplectic form on $M$ induces one on $M^\T$ and then invoke Proposition \ref{isoprop}.) Together with the orientation on $E\udot_f$ this therefore defines a virtual cycle on the fixed locus by \eqref{virMdef},
\beq{fixvc}
\big[M^\T\big]^{\vir}\=\sqrt{0_{E^f}^{\;!}}\,\big[C_{E_f\udot}\big]\ \in\ A_{\frac12(\vd-r)}^\T\(M^\T,\ZZ\).
\eeq
%More precisely, with the space $P$ in the proof, we will see that
%$$
%C_{E_f\udot}\=\frac{C_{M^\T/P^\T}\oplus T^f}{T_P|^f_{M^{\T}}}\ \subset\ \frac{C_{M/P}|_{M^{\T}}\oplus T|_{M^{\T}}}{T_P|_{M^{\T}}}\=C_{E\udot}|_{M^{\T}},
%$$
%cf. the Cartesian diagram \eqref{Cart} and Equation \eqref{Cemp}. 
The complex dimension $\frac12(\vd-r)$ of this cycle can vary from one connected component of $M^{\T}$ to another along with that of
\beq{eNvir}
\sqrt{e\_{\T}}\;(N^{\vir})\ :=\ \frac{e\_{\T}(T^m)}{\sqrt{e\_{\T}}\;(E^m)}\ \in\ A^{r/2}_\T\(M^\T,\Q\)\big[t^{-1}\big].
\eeq
Here we have used the fact that $T^m$ and $E^m$ split into weight spaces \emph{all of whose weights are nonzero}, so the equivariant Chow cohomology classes $e\_{\T}(T^m),\,\sqrt{e\_{\T}}\;(E^m)$ and $\sqrt{e\_{\T}}(N^{\vir})$ are all invertible once we localise by inverting $t=c_1(\t)$, where $\t$ is the standard one dimensional representation of $\T$. Using \eqref{Kred} or Footnote \ref{ABo} it is elementary to show that \eqref{eNvir} is independent of choices;
since we do not strictly need this we leave the details to the reader.

\begin{Thm}\label{localChow} We can localise $[M]^{\vir}$ to $\iota\colon M^\T\into M$ by
$$
[M]^{\vir}\=\iota_*\,\frac{\big[M^\T\big]^{\vir}}{\sqrt{e\_{\T}}\;(N^{\vir})}\ \in\ A^\T_{\frac12\!\vd}(M,\Q)\big[t^{-1}\big].
$$
\end{Thm}

\begin{proof}
We follow \cite[Section 3]{GP}. The usual construction of the moduli space $M$ as a GIT quotient of an open set in a Quot scheme (of an equivariant compactification of $X$) can be done $\T$-equivariantly since $\cO_X(1)$ is a $\T$-linearised line bundle. Thus GIT endows $M$ with a $\T$-linearised ample line bundle $\cO_M(1)$. For large $N\gg0$ there is a $\T$-equivariant finite dimensional linear system $V\subset  H^0(\cO_M(N))$ inducing a $\T$-equivariant embedding $M\into P:=\PP(V^*)$ into a smooth $\T$-variety $P$. The $\T$-fixed locus $i:P^\T\into P$ is also smooth and so regularly embedded (but with codimension $c$ which can vary from one connected component to another). It fits into a Cartesian diagram
\beq{Cart}
\xymatrix@=18pt{
M^\T\ \ar[d]<-.8ex>\ar@{^(->}[r]^{\iota}&\,M \ar[d] \\
P^\T\ \ar@{^(->}[r]^i&\,P.}
\eeq
Furthermore, by Proposition \ref{form} we may assume that the map $E\udot\to\LL_M$ of \eqref{surjec} is represented by a $\T$-equivariant surjection of complexes
$$
\xymatrix@R=15pt{T \ar[r]& E \ar[r]\ar@{->>}[d]& T^* \ar@{->>}[d]<-0.5ex> \\
& I/I^2 \ar[r]^-d& \Omega_P|_M.\!}
$$
In particular, $T_P|_M\to T$ is injective over $M$.

By \cite[Theorem 6.3.5]{Kr}, for instance,
\beq{loc1}
\iota_*\ \colon\,A^\T_*(M^\T,\Q)\big[t^{-1}\big]\rt\sim
A^\T_*(M,\Q)\big[t^{-1}\big]
\eeq
is an isomorphism, so we may write $[M]^{\vir}=\iota_*\;a$ for some $a\in A_*(M^\T)[t^{-1}]$. Then by \cite[Corollary 6.3]{F},
\beq{shriek1}
i^{\;!}[M]^{\vir}\=i^{\;!}\iota_*\;a\=e\_\T(N_{P^\T/P})\cap a,
\eeq
where $i^{\;!}$ is the refined Gysin map associated to the Cartesian diagram \eqref{Cart}. Splitting $N_{P^\T/P}$ into weight spaces, all its weights are nonzero, so $e\_\T(N_{P^\T/P})$ is invertible. Thus we get the standard localisation formula
\beq{stand1}
\iota_*\left(\frac{i^{\;!}[M]^{\vir}}{e\_\T(N_{P^\T/P})}\right)\=\iota_*\;a\=[M]^{\vir}.
\eeq
We split $T_P|_{P^\T}=T^f_P\oplus T^m_P$ into fixed and moving parts, where $T_P^m=N_{P^\T/P}$. Substituting in the definition \eqref{virMdef} of $[M]^{\vir}$ now gives
\beq{nowg}
[M]^{\vir}\=\iota_*\,\frac{i^{\;!}\sqrt{0_E^{\;!}}\,[C_{E\udot}]}{e\_{\T}(T^m_P)}\ \stackrel{\eqref{Eshriek}}=\ \iota_*\,\frac{\sqrt{0_{\iota^*E}^{\;!}}\,i^{\;!}[C_{E\udot}]}{e\_{\T}(T^m_P)}\,,
\eeq
where the isotropic cone $C_{E\udot}\subset E$ is defined from the intrinsic normal cone $\mathfrak C_M$ by \eqref{recipe}. Since $\mathfrak C_M$ is by definition the stack quotient $C_{M/P}\big/\;T_P|_M$ we see from \eqref{recipe} that
\beq{Cemp}
C_{E\udot}\=\frac{C_{M/P}\oplus T}{T_P|_M}\,,
\eeq
where we take the diagonal action of $T_P|_M$. Expanding our Cartesian diagram to
$$
\xymatrix@=18pt{
\iota^*C_{M/P} \ar[r]\ar[d]<-.5ex>& C_{M/P} \ar[d] \\
M^\T\,\;\ar[d]<-.5ex>\ar@{^(->}[r]^{\iota}&\,M \ar[d] \\
P^\T\ \ar@{^(->}[r]^i&\,P,}
$$
Vistoli's rational equivalence \cite{Vi} implies
\beq{Vist}
i^{\;!}\big[C_{M/P}\big]\=\big[C_{M^\T/P^\T}\big]\quad\text{in }\,A_*^\T\(\iota^*C_{M/P}\),
\eeq
as in \cite[Equation 15]{GP}, or the equation before Proposition 3.3 in \cite{BF}. 
Thus
\beq{Vis2}
i^{\;!}\big[C_{M/P}\oplus T\big]\=\big[C_{M^\T/P^\T}\oplus\iota^*T\big],
\eeq
by \cite[Theorem 6.2(b)]{F}. We now quotient by $T_P|_M$ using the diagram
\beq{affn}
\xymatrix@R=20pt{
A_{*+t}^\T(C_{M/P}\oplus T) \ar[r]^-{i^{\;!}}& A_{*+t-c}^\T(\iota^*C_{M/P}\oplus\iota^*T) \\
A_*^\T(C_{E\udot}) \ar[u]^{q^*}_(.45)\wr\ar[r]^-{i^{\;!}}& A_{*-c}^\T(\iota^*C_{E\udot}). \ar[u]_{q^*}^(.45){\wr}}
\eeq
Here $q\colon C_{M/P}\;\oplus\,T\to C_{E\udot}$ is an affine bundle of dimension $t:=\rk T_P|_M$, giving the vertical isomorphisms by \cite[Corollary 3.6.4]{Kr}. Since $q$ is flat, the diagram commutes by \cite[Theorem 6.2(b)]{F}. By \eqref{Cemp} and \eqref{Vis2} we find
\beq{wrongcycle}
q^*i^{\;!}\big[C_{E\udot}\big]\=\big[C_{M^\T/P^\T}\oplus\iota^*T\big].
\eeq
We would like to deform $C_{M^\T/P^\T}\oplus\;\iota^*T$ to something invariant \vspace{-1mm}under the action of $T_P|_{M^\T}=T_P^f\oplus T_P^m$ in order to apply $(q^*)^{-1}$ to both sides.\footnote{Alternatively one can work entirely ``upstairs" on $C_{M^\T/P^\T}\oplus\;\iota^*T$ via an analogue of \cite[Lemma 1]{GP}.} To do this we may assume, without loss of generality, that the inclusion $T_P^m\into T^m$ is \emph{split}. This is an application of the Jouanolou trick \cite{J}, as used in \cite[Section 1.1]{GT} for instance, replacing $M^\T$ by a \emph{affine} variety on which all extensions split automatically. That is, there is an affine variety 
$$
\wt{M^\T}\To M^\T
$$
which is an \emph{affine bundle} over $M^\T$. It is therefore a homotopy equivalence inducing an isomorphism on Chow groups \cite[Corollary 3.6.4]{Kr}, so any relation we prove in the Chow group upstairs also holds downstairs on $M^\T$.

The splitting $T^m\onto T_P^m$ induces a map $f\colon T^m\to\iota^*C_{M/P}$. Deforming the inclusion induced by the Cartesian diagram \eqref{Cart}
$$
C_{M^\T/P^\T}\oplus T^f\oplus T^m\Into\iota^*C_{M/P}\oplus T^f\oplus T^m
$$
by $t.f$, for $t\in\C$, gives a rational equivalence from \eqref{wrongcycle} at $t=0$  to a $T_P|_{M^\T}$-invariant cycle at $t=1$. Thus \eqref{wrongcycle} becomes
\beq{vis3}
q^*i^{\;!}\big[C_{E\udot}\big]\=q^*\left[\frac{C_{M^\T/P^\T}\oplus\iota^*T}{T_P|_{M^\T}}\right],
\eeq
where on the right the $T^m$ in $\iota^*T$ embeds in $\iota^*C_{M/P}\;\oplus\;\iota^*T$ via $(f,1)$ instead of $(0,1)$. Thus the same relation is true with $q^*$ removed. So by Proposition \ref{E1E2}, \eqref{nowg} has now become
$$
\iota_*\,\frac{\sqrt{0_{E^f}^{\;!}}\left[\frac{C_{M^\T/P^\T}\oplus T^f}{T_P^f}\right]\sqrt{0_{E^m}^{\;!}}\left[\frac{T^m}{T_P^m}\right]}{e\_{\T}(T^m_P)}\=
\iota_*\,\frac{\sqrt{0_{E^f}^{\;!}}\big[C_{E\udot_f}\big]\sqrt{e\_{\T}}(E^m)\frac{e\_{\T}(T^m_P)}{e\_{\T}(T^m)}}{e\_{\T}(T^m_P)}\,.
$$
On the right we have used Proposition \ref{KperpK} to write
$$
\sqrt{0_{E^m}^{\;!}}\,[K]\=\sqrt{0_{K^\perp/K}^{\;!}}\,\big[0_{K^\perp/K}\big]
\=\sqrt e\_\T\;(K^\perp/K)\ \stackrel{\eqref{Kred}}=\ \sqrt e\_\T\;(E^m)\big/e\_\T\;(K),
$$
where $K=T^m/T^m_P$ injects into $E^m$ by taking moving parts of the inclusion\footnote{\label{F28}Notice the fixed part of the left hand side is $C_{E\udot_f}$ by the same argument as in \eqref{Cemp}. So $C_{E\udot_f}$ is contained in the fixed part of $C_{E\udot}|_{M^\T}$ which is isotropic in $E|_{M^\T}$. This shows that $C_{E\udot_f}\subset E^f$ is isotropic, as used earlier.}
$$
\frac{C_{M^{\T}/P^{\T}}\oplus T|_{M^{\T}}}{T_P|_{M^{\T}}}\ \subset\ \frac{C_{M/P}|_{M^{\T}}\oplus T|_{M^{\T}}}{T_P|_{M^{\T}}}\ \stackrel{\eqref{Cemp}}=\ C_{E\udot}|_{M^{\T}}\ \subset\ E\;|_{M^{\T}}.
$$
%where the first inclusion is obtained by the Cartesian diagram \eqref{Cart}, inducing $C_{M^{\T}/P^{\T}}\subset C_{M/P}|_{M^{\T}}$.} 
Thus by \eqref{fixvc},
\[
[M]^{\vir}\=\iota_*\,\frac{\big[M^\T\big]^{\vir}\sqrt{e\_{\T}}(E^m)}{e\_{\T}(T^m)}\=\iota_*\,\frac{\big[M^\T\big]^{\vir}}{\sqrt{e\_{\T}}(N^{\vir})}\,.\qedhere
\]
\end{proof}

We would like to do something similar with $\Ohat$ in the equivariant $K$-theory groups $K^0_\T,\,K_0^\T$. But our definition \eqref{sqdf} of $\surd L$ does not immediately work equivariantly because $1-L$ is usually not nilpotent in $K^0_{\T}(Y)$, so \eqref{sqdf} becomes an infinite series. (For instance $1-\t$ is not nilpotent, where $\t$ is the weight one irreducible representation of $\T$.) We thank Andrei Okounkov for pointing this out, and Noah Arbesfeld for a suggestion on how best to fix it.

Note first that given a $\T$-equivariant line bundle on a \emph{$\T$-fixed scheme} such as $M^\T$, it takes the form $L'\otimes\t^w$ for some $\T$-fixed line bundle $L'$ and locally constant integer-valued weight $w$. Thus we may use \eqref{sqdf} to define its square root as
\beq{sqdf2}
\sqrt{L'}\otimes\t^{w/2}\ \in\ K^0_\T(M^\T)\otimes\_{\;\Z[\t,\t^{-1}]}\Q\(\t^{1/2}\).
\eeq
Combining this with the localisation formula will allow us to construct an \emph{operator} $\surd L\,\star$ on a localised version of $K_0^{\T}$ (rather than an element of a localisation of $K^0_\T$). So for any quasi-projective $\T$-scheme $Y$, set
$$
K^{\T}_0(Y)\loc\ :=\ K_0^\T(Y)\otimes\_{\;\Z[\t,\t^{-1}]}\Q\(\t^{1/2}\).
$$
Denoting the fixed locus by $\iota\colon Y^{\T}\into Y$, we have the $K$-theoretic localisation theorem \cite[Theorem 3.3(a)]{EG:Lo},
\beqa
\iota_*\ \colon\ K_0^\T\(Y^\T\)\loc\rt\sim K_0^\T(Y)\loc\;,
\eeqa
and its inverse $\iota_*^{-1}$. Given a $\T$-equivariant line bundle $L$ this allows us to define an operator
$$
\sqrt L\,\star\ \colon\,K^{\T}_0(Y)\loc\To K^{\T}_0(Y)\loc
$$
by specifying it to be \eqref{sqdf2} on $Y^\T$, where $\iota^*L\cong L'\cdot\t^w$. 
That is, we define
\beq{stardef}
\sqrt L\,\star F\ :=\ \iota_*\left(\sqrt{L'}\cdot\t^{w/2}\otimes\iota_*^{-1}(F)\right).
\eeq
%By the projection formula this has the following compatibility with \eqref{sqdf2},
%\beq{compatt}
%\iota^*\circ\sqrt L\,\star(\ \cdot\ )\=\sqrt{\iota^*L}\,\otimes\;(\ \cdot\ ).
%\eeq
We check this reduces to our previous definition \eqref{sqdf} in the non-equivariant limit $\t^{1/2}\to 1$. Let $L_1\in K^0(Y)$ denote the line bundle $L|_{\t=1}$ given by forgetting about the $\T$ action.

\begin{Prop}
If $F\in K_0^\T\(Y\)\loc$ has no pole at $\t^{1/2}=1$ then setting $F_1:=\lim_{\;\t^{1/2}\to1}(F)\in K_0(Y)\otimes\Q$,
$$
\lim_{\t^{1/2}\to1}\sqrt L\,\star F \,\text{ exists and equals }\, \sqrt{L_1}\otimes F_1\ \in\ K_0(Y)\otimes\Q.
$$
\end{Prop}

\begin{proof}
Let $Sq\_k(x)$ denote the approximation to the square root given by truncating the power series \eqref{sqdf},
$$
Sq\_k(x)\ :=\ 1-\sum\nolimits_{i=1}^ka_i(1-x)^i\ \in\ \Q[x].
$$
Applied to $\iota^*L$ this approximates $\surd L'\cdot\t^{w/2}$ with error
$$
E_k\ :=\ Sq\_k(L'\cdot\t^w)-\sqrt{L'}\cdot\t^{w/2}\ \in\ K^0\(Y^\T\)\otimes\Q[\t^{\pm1/2}].
$$
Working in $K_0^\T(Y)\loc$, by \eqref{stardef} we have
\begin{eqnarray}
\sqrt L\,\star F &=& \iota_*\left(\sqrt{L'}\cdot\t^{w/2}\otimes\iota_*^{-1}(F)\right) \nonumber \\
&=& \iota_*\left(Sq\_k(\iota^*L)\otimes\iota_*^{-1}(F)\right)-\iota_*\(E_k\otimes\iota_*^{-1}(F)\) \nonumber \\ 
&=&Sq\_k(L)\otimes F-\iota_*\(E_k\otimes\iota_*^{-1}(F)\), \label{sqke}
\end{eqnarray}
by the projection formula. We will show that, for $d\in\N$ specified below and then any fixed $k\ge\dim Y^\T+d$,
\begin{itemize}
\item[(a)] $E_k\in(1-\t^{1/2})^{d+1}K^0\(Y^\T\)\otimes\Q[\t^{\pm1/2}]$ and
\item[(b)] $(1-\t)^d\,\iota_*^{-1}F\in K_0(Y^\T)\loc$ has no pole at $\t^{1/2}=1$.
\end{itemize}
Thus $(1-\t^{1/2})$ divides $E_k\otimes\iota_*^{-1}(F)$, so $\lim_{\t^{1/2}\to1}$ of either side of \eqref{sqke} exists and equals $Sq\_k(L_1)\otimes F_1$. Since $k\ge\dim Y^\T$ this is $\surd L_1\otimes F_1$, as required.\medskip

We first prove (a). By \eqref{xhalf} we see that $(1-x)^{k+1}$ divides
$$
Sq\_k(x)^2-x\=\(Sq\_k(x)-x^{1/2}\)\(Sq\_k(x)+x^{1/2}\)\ \in\ \Q[x^{1/2}].
$$
Since the second factor is nonzero at $x^{1/2}=1$ and $\Q[x^{1/2}]$ is a principal ideal domain, we deduce that $\(1-x^{1/2}\)^{k+1}$ divides the first, so
\beq{sq12}
Sq\_k(x)-x^{1/2}\ \in\ \(1-x^{1/2}\)^{k+1}\Q[x^{1/2}].
\eeq
Taking the branch of $\sqrt x$ which is $1$ at $x=1$ and Taylor expanding the identity $\sqrt{xy}\equiv\sqrt x\sqrt y$ about $(x,y)=(1,1)$ shows that the formal power series $Sq(x):=1-\sum_{i=1}^\infty a_i(1-x)^i\in\Q[\![1-x]\!]$ satisfies
$$
Sq(xy)\=Sq(x)\;Sq(y)\ \in\ \Q[\![1-x,1-y]\!].
$$
Truncating, it follows that 
$$
Sq\_k(xy)-Sq\_k(x)\;Sq\_k(y)\ \in\ \m^{k+1}\;\Q[x,y],
$$
where $\m:=(1-x,1-y)\subset\Q[x,y]$ is the maximal ideal at $(1,1)$. Hence there is a polynomial $f(x,y)$ such that
\beq{111}
Sq\_k(L'\cdot\t^w)-Sq\_k(L')\;Sq\_k(\t^w)\=f(1-L',1-\t),
\eeq
where $f$ is a sum of monomials all of which have degree $\ge k+1$. All of these monomials are divisible by $(1-\t)^{k+1-\dim Y^\T}$ because $(1-L')^{\dim Y^\T+1}=0$. So combining \eqref{111} and \eqref{sq12} gives, for $k\ge\dim Y^\T+d$,
$$
Sq\_k(L'\cdot\t^w)-Sq\_k(L')\;\t^{w/2}\ \in\ (1-\t^{1/2})^{d+1}K^0\(Y^\T\)\otimes\Q[\t^{\pm1/2}].
$$
Since the left hand side is $E_k$ this proves (a). \medskip

To prove (b) we use a $\T$-equivariant embedding of $Y$ in a smooth projective $\T$-scheme $P$ with fixed locus $i\colon P^\T\into P$ as in \eqref{Cart}. Let $N:=N_{P^\T/P}|_{Y^\T}$. The refined Gysin map $i^{\;!}\colon K_0(Y)\to K_0(Y^\T)$ of \cite[Section 2.1]{YP} satisfies
\beq{YPLee}
\iota_*^{-1}\=\frac{i^{\;!}}{\Lambda\udot N^*}\ \colon\, K^\T_0(Y)\loc\To K^\T_0\(Y^\T\)\loc\;.
\eeq
Writing $N=\oplus_{i=1}^{\,c}N_i\;\t^{u_i}$ as a sum of Chern roots, where $c$ is the locally constant function $\dim P-\dim P^\T\le\dim P$, we have
$$
\Lambda\udot N^*\=\bigotimes_{i=1}^c\(1-N_i^*\t^{-u_i})\=\bigotimes_{i=1}^c(1-\t^{-u_i})\Big(1+\tfrac{\t^{-u_i}}{1-\t^{-u_i}}\(1-N_i^{-1}\)\Big).
$$
The first bracket contributes a pole of order $c\le\dim P$ to $1/\Lambda\udot N^*$. Expanding the product of the reciprocals of the second bracket gives a pole of order $\le\dim Y^\T$ because any product of $>\dim Y^\T$ terms of the form $(1-N_i^{-1})$ is zero. Thus $1/\Lambda\udot N^*$ has a pole of order $\le d:=\dim Y^\T+\dim P$ at $\t=1$. In particular, for $F$ whose limit $F_1$ exists as $\t^{1/2}\to1$, we see that $(1-\t)^d\,\iota_*^{-1}F$ has no pole at $\t^{1/2}=1$, as required.
\end{proof}

Thus we get a $\T$-equivariant analogues of everything in Section \ref{Ksec}, replacing each occurence of $\surd L\in K^0$ by the operator $\surd L\;\star\;\colon K_0^\T\to K_0^\T$ \eqref{stardef}. For instance, given a $\T$-equivariant $SO(2n,\C)$ bundle $(E,q,o)$ over a $\T$-scheme $Y$, the bundle $\rho\colon\wt Y\to Y$ \eqref{cover} is naturally $\T$-equivariant, as is $\Lambda_\rho\subset\rho^*E$. We then define
$$
\sqrt{\mathfrak e\_{\T}}(E)\ :=\ \rho_*\(\sqrt{\det\Lambda_\rho}\star\;\mathfrak e\_{\T}(\Lambda_\rho)\cap(\ \cdot\ )\)  \ \colon\, K_0^\T\(Y)\loc\To K_0^\T\(Y)\loc.
$$
Since the $K$-theoretic cosection localisation of \cite{KL:K} operates on $K_0$  anyway, we can work equivariantly in Definition \ref{def1} to get a $\T$-equivariant localised operator
$$
\e(E,s)\ \colon\, K_0^\T\(Y)\loc\To K_0^\T\(Z(s)\)\loc 
$$ 
when $s$ is a $\T$-equivariant isotropic section. Applying this to a $\T$-equivariant cone $C\subset E$, Definition \ref{def2} gives an operator
$$
\sqrt{0_E^*}\ \colon\,K_0^\T\(C)\loc\To K_0^\T\(Y)\loc\;.
$$
Combining this with the $\T$-equivariant resolution \eqref{surjec} defines, by Definition \ref{vss}, an equivariant virtual structure sheaf on the moduli space $M$,
\beq{TOhat}
\Ohat\ :=\ \sqrt{\det T^*}\star\sqrt{0_E^*}\,\big[\cO_{C_{E\udot}}\big]\ \in\ K_0^\T(M)\loc
\eeq
whose $\t^{1/2}\to1$ limit is \eqref{Ohatdef}.\footnote{Thus the right hand side of Theorem \ref{Kloc} has no poles at $\t^{1/2}=1$ and specialises there to \eqref{Ohatdef}. Beware, however, that before taking $\iota_*$ we typically have poles at $\t^{1/2}=1$.}
So we have the ingredients to give a virtual localisation result for $\Ohat$ in $\T$-equivariant $K$-theory. This is a square-rooted analogue of the usual virtual localisation formula in $K$-theory \cite{Qu}. In the notation of \eqref{efemtftm} we set
$$
\sqrt{\mathfrak e\_{\T}}\;(N^{\vir})\ :=\ \frac{\mathfrak e\_{\T}(T^m)\cdot\sqrt{\det T^m}}{\sqrt{\mathfrak e\_{\T}}\,(E^m)}\ \in\ K^0_\T(M^\T,\Q),
$$
where $\surd\!\det T^m$ is defined as in \eqref{sqdf}.

\begin{Thm}\label{Kloc} We can localise $\Ohat$ to $\iota\colon M^\T\into M$ by
$$
\Ohat\=\iota_*\,\frac{\widehat\cO^{\;\vir}_{\!M^\T}}{\sqrt{\mathfrak e\_{\T}}\;(N^{\vir})}\ \in\ K_0^\T(M)\loc.
$$
\end{Thm}

\begin{proof}
We follow the proof of Theorem \ref{localChow} closely, highlighting small differences. 
The localisation formula \eqref{loc1} is replaced by its $K$-theoretic analogue
\beqa
\iota_*\ \colon\, K_0^\T(M^\T)\loc\rt\sim K_0^\T(M)\loc
\eeqa
of \cite[Theorem 3.3(a)]{EG:Lo}. By \eqref{YPLee},
$$
i^{\;!}\iota_*\=\Lambda\udot N_{P^\T/P}^*\big|_{M^\T}\otimes\ ,
$$
giving the $K$-theoretic analogue of \eqref{shriek1}, so the same argument there gives the following replacement for \eqref{stand1},
$$
\iota_*\left(\frac{i^{\;!}\;\Ohat}{\Lambda\udot N_{P^\T/P}^*\big|_{M^\T}}\right)\=\Ohat.
$$
Splitting $T_P|_{P^\T}=T^f_P\oplus T^m_P$ into fixed and moving parts, where $T^m_P=N_{P^\T/P}$, and substituting in the definition \eqref{TOhat} of $\Ohat$ now gives
\beq{nowh}
\iota_*\,\frac{i^{\;!}\(\sqrt{0_E^*}\,\big[\cO_{C_{E\udot}}\big]\cdot\sqrt{\det T^*}\;\)}{\mathfrak e\_{\T}(T^m_P)}\=
\iota_*\,\frac{\sqrt{0_{\iota^*E}^*}\,i^{\;!}\big[\cO_{C_{E\udot}}\big]}{\mathfrak e\_{\T}(T^m_P)\cdot\sqrt{\det\iota^*T}},
\eeq
where we have used the following consequence of \eqref{YPLee},
$$
i^!\(\sqrt{\det T^*}\,\star(\ \cdot\ )\)\=\sqrt{\iota^*\det T^*}\,\otimes\;i^!(\ \cdot\ ),
$$
%\begin{align*}%\beq{compatt}
%i^!\(\sqrt{\det T^*}\,\star(\ \cdot\ )\)%&\=\Lambda\udot N_{P^\T/P}^*\big|_{M^\T}\otimes\iota_*^{-1}\(\sqrt{\det T^*}\,\star(\ \cdot\ )\)\\
%&\=\Lambda\udot N_{P^\T/P}^*\big|_{M^\T}\otimes\sqrt{\iota^*\det T^*}\,\otimes\;\iota_*^{-1}(\ \cdot\ )\\
%&\=\sqrt{\iota^*\det T^*}\,\otimes\;i^!(\ \cdot\ ).
%\end{align*}%\eeq
and the $K$-theoretic analogue $i^{\;!}\sqrt{0_E^*}\=\sqrt{0_{\iota^*E}^*}\,i^{\;!}$ \eqref{Kshriek} of \eqref{Eshriek}.

From \cite[Lemma 2]{YP} we get the following $K$-theoretic analogue of Vistoli's rational equivalence \eqref{Vist},
$$
i^{\;!}\big[\cO_{C_{M/P}}\big]\=\big[\cO_{C_{M^\T/P^\T}}\big]\ \in\ K_0^\T(\iota^*C_{M/P})\loc\;.
$$
The $K$-theoretic analogue of \eqref{affn} is the isomorphism $K_0^\T(A)\loc\cong K_0^\T(B)\loc$ of \cite[Theorem 5.4.17]{CG} when $A\to B$ is a $\T$-equivariant affine bundle. Thus the same arguments as before show the analogue of \eqref{vis3}, namely $i^{\;!}\big[\cO_{C_{E\udot}}\big]$ is the structure sheaf of
$$
\frac{C_{M^\T/P^\T}\oplus\iota^*T}{\iota^*T_P}\=
\frac{C_{M^\T/P^\T}\oplus T^f}{T_P^f}\oplus\frac{T^m}{T^m_P}
$$
inside $E^f\oplus E^m$.
Therefore \eqref{nowh} becomes
\begin{eqnarray}
\Ohat &=& \iota_*\,\frac{\sqrt{0_{E^f}^*}\,\big[\cO_{C_{E_f\udot}}\big]\sqrt{0_{E^m}^*}\,\big[\cO_{T^m/T^m_P}\big]}{\mathfrak e\_{\T}\(T^m_P\)\cdot\sqrt{\det\iota^*T}} \nonumber \\
&=& \iota_*\,\frac{\widehat\cO^{\;\vir}_{\!M^\T}\sqrt{\det T^f}\e\_\T(E^m)\frac{\mathfrak e\_\T(T^m_P)}{\mathfrak e\_\T(T^m)}}{\mathfrak e\_{\T}\(T^m_P\)\cdot\sqrt{\det\iota^*T}} \nonumber \\
&=& \iota_*\,\frac{\widehat\cO^{\;\vir}_{\!M^\T}\e\_\T(E^m)}{\mathfrak e\_{\T}(T^m)\cdot\sqrt{\det T^m}} \nonumber \\
&=& \iota_*\,\frac{\widehat\cO^{\;\vir}_{\!M^\T}}{\e\_\T(N^{\vir})}\,,\label{KR}
\end{eqnarray}
%\begin{eqnarray}
%\Ohat &=& \iota_*\,\frac{\sqrt{0_{E^f}^*}\,\big[\cO_{C_{E_f\udot}}\big]\sqrt{0_{E^m}^*}\,\big[\cO_{T^m/T^m_P|\_M}\big]}{\mathfrak e\_{\T}\(T^m_P|\_M\)\cdot\sqrt{\det T}} \nonumber \\
%&=& \iota_*\,\frac{\widehat\cO^{\;\vir}_{\!M^\T}\sqrt{\det T^f}\e\_\T(E^m)\frac{\mathfrak e\_\T(T^m_P|\_M)}{\mathfrak e\_\T(T^m)}}{\mathfrak e\_{\T}\(T^m_P|\_M\)\cdot\sqrt{\det T}} \nonumber \\
%&=& \iota_*\,\frac{\widehat\cO^{\;\vir}_{\!M^\T}\e\_\T(E^m)}{\mathfrak e\_{\T}(T^m)\cdot\sqrt{\det T^m}} \nonumber \\
%&=& \iota_*\,\frac{\widehat\cO^{\;\vir}_{\!M^\T}}{\e\_\T(N^{\vir})}\,,\label{KR}
%\end{eqnarray}
where the second equality follows from applying \eqref{KKperpK} and \eqref{K3} to the isotropic  $K:=T^m/T^m_P\subset E^m$.
\end{proof}

\begin{Rmk} This allows us to define invariants for $X$ --- even though $X$ and $M$ were only assumed quasi-projective --- so long as $M$ parameterises compactly supported sheaves and $M^\T$ is projective. Over $[M]^{\vir}$ we integrate lifts of insertions to the equivariant cohomology of $M$, yielding invariants in $\Q[t,t^{-1}]$. Or in $K$-theory we lift the natural invariant $\chi(\Ohat)$ to its equivariant $\T$-character in $\Q(\t^{1/2})=K^0_\T(\;$point$)\loc\;$. Unless $M$ is projective the natural numerical specialisation of these invariants (to $t=0$ in the cohomological case and to $\t^{1/2}=1$ in $K$-theory) may not be defined. However the localisation theorems show the results are invariant under deformations through other $\T$-varieties satisfying the same conditions.

In \cite{KR} Kool and Rennemo give a simple formula (in particular explicitly determining the sign) for the contribution \eqref{KR} to $\Ohat$ of isolated reduced points of $M^\T$ with virtual dimension zero, i.e. when $\vd-\,r=0$ in \eqref{fixvc}.
\end{Rmk}

Finally we note that over $\Q$ the two localisation formulae are equivalent under the Riemann-Roch map of Section \ref{RRsec}. Firstly one can show that for any $SO(2n,\C)$ bundle $F$,
\beq{str}
\ch(\e(F))\=\sqrt\td(F)^{-1}\sqrt e\;(F)
\eeq
by working on the cover \eqref{cover} with the maximal isotropic $\Lambda_\rho$. Applying \eqref{FF*} to $\Lambda_\rho$ quickly reduces \eqref{str} to the identity $c_n(\Lambda_\rho)=\ch(\Lambda\udot\Lambda^*_\rho)\,\td(\Lambda_\rho)$. An easy calculation using \eqref{str} then gives
$$
\ch\(\sqrt{\mathfrak e\_\T}(N^{\vir})\)\=\frac{\sqrt{e\_\T}(N^{\vir})}{\sqrt{\td}(N^{\vir})}\,.
$$
By the module property \cite[Theorem 18.2(2)]{F} of $\tau\_{M^\T}$ this shows
$$
\tau\_{M^\T} \!\left( \frac{\widehat{\cO}^{\vir}_{M^\T} }{\sqrt{\mathfrak{e}_\T }  (N^{\vir}) }  \right) \= \sqrt{\td }(N^{\vir}) \cap \frac{\tau\_{M^\T}\!\left( \widehat{\cO}^{\vir}_{M^\T} \right)}{\sqrt{e_\T }  (N^{\vir})}\,.
$$
Applying Theorem \ref{virRR} to $M^\T$, with virtual tangent bundle $E_\bullet^f$, this gives
\begin{eqnarray} \nonumber
\tau\_{M^\T} \!\left( \frac{\widehat{\cO}^{\vir}_{M^\T} }{\sqrt{\mathfrak{e}_\T }  (N^{\vir}) }  \right) 
&=&  \sqrt{\td }\(E_\bullet^m\) \sqrt{\td }\(E_\bullet^f\)  \cap \frac{\big[M^\T\big]^{\vir} }{\sqrt{e_\T }  (N^{\vir})} \\
&=&  \sqrt{\td }(E_\bullet)  \cap \frac{\big[M^\T\big]^{\vir} }{\sqrt{e_\T }  (N^{\vir})}\,. \label{RRT}
\end{eqnarray}

\section{Local Calabi-Yau 4-folds}\label{KY} We thank Yukinobu Toda for suggesting we study the example of local Calabi-Yau 4-folds $X=K_Y$, i.e. canonical bundles of 3-folds $Y$. These have already been treated in \cite{CL, DSY} by conjecturing the form of the Borisov-Joyce virtual cycle in this case. We show briefly how to prove these conjectures for our virtual cycle. When combined with the sequel \cite{OT2} this makes many of the results in those papers rigorous. \medskip

The correct abstract setting is to start with a projective scheme $M$ with a perfect obstruction theory $\At\colon F\udot\to\LL_M$. Writing $F\udot$ as a 2-term complex of vector bundles $F^{-1}\to F^0$ we get the Behrend-Fantechi cone $j\colon C_{F\udot}\into F_1$ \eqref{recipe}, whose intersection with the zero section $0_{F_1}$,
\beq{BFvirt}
0_{F_1}^{\;!}\big[C_{F\udot}\big]\=[M]^{\vir},
\eeq
defines the usual Behrend-Fantechi virtual cycle of $M$. Intersecting in $K$-theory instead, then twisting by a square root\footnote{Nekrasov-Okounkov show such a square root exists \cite[Section 6.2]{NO} in some situations. By Lemma \ref{sqrts} any choice equals $\surd K_{\vir}$ once we invert 2 in our coefficients.} of $K_{\vir}:=\det F\udot$ gives the twisted virtual structure sheaf
\beq{vss3}
\Ohat\=0_{F_1}^*\big[\cO_{C_{F\udot}}\big]\otimes K_{\vir}^{\frac12}.
\eeq\smallskip

The \emph{$(-2)$-shifted cotangent bundle} of $(M,F\udot)$ is the same scheme $M$ with the different, non-perfect, obstruction theory
\beq{sum}
E\udot\ :=\ F\udot\oplus F_\bullet[2]\rt{(\At,\;0)}\LL_M.
\eeq
Here $E\udot=\{F_0\to F_1\oplus F^{-1}\to F^0\}$ is a self-dual 3-term complex of vector bundles. By \eqref{theta=Q} an orientation on $E\udot$ is the same as one on
$$
F_0\oplus F^0\oplus(F_1\oplus F^{-1})
$$
which in turn admits the canonical orientation\footnote{It would be nice to know the relationship between this orientation and those constructed for noncompact Calabi-Yau 4-folds in \cite{Bo1}.} $o\_{F_0\oplus F_1}$ of \eqref{oTis}. By the construction of Proposition \ref{por} this induces the orientation $o\_{F_1}$ on 
$F_1\oplus F^{-1}$ with respect to which $F_1\subset F_1\oplus F^{-1}$ is a positive maximal isotropic.

Therefore $M$ inherits a new virtual cycle from Definition \ref{vdef}.  We recall its construction. The procedure of Section \ref{vcsec} uses the truncation $\tau E\udot=\{F^{-1}\oplus F_1\to F^0\}$. Since the arrow is zero on the second summand $F_1$, the prescription \eqref{recipe} gives the cone
$$
\xymatrix{C_{F\udot}\ \,\ar@{^(->}[r]<-.15ex>^-{(j,0)}& \ F_1\oplus F^{-1}\= E_1}
$$
lying in the maximal isotropic $F_1\subset E_1$. Therefore the virtual cycle of Definition \ref{vdef} is
\beq{Chowa}
\sqrt{0^{\;!}_{F_1\oplus F^{-1}}}\,\big[C_{F\udot}\big]\ \stackrel{\eqref{Kevin34}}=\ 0_{F_1}^{\;!}\big[C_{F\udot}\big]\=[M]^{\vir}.
\eeq
Similarly we find the virtual structure sheaf \eqref{Ohatdef} is
\begin{align}\nonumber
\sqrt{0^*_{F_1\oplus F^{-1}}}\,\big[\cO_{C_{F\udot}}\big]\cdot\sqrt{\det F^0}\ \stackrel{\eqref{Klam}}=\ 0_{F_1}^*\big[\cO_{C_{F\udot}}\big]\cdot\sqrt{\det(F^{-1})^*}\cdot\sqrt{\det F^0}& \\ \label{vss4}
=\ 0_{F_1}^*\big[\cO_{C_{F\udot}}\big]\cdot\sqrt{\det F\udot}\=0_{F_1}^*\big[\cO_{C_{F\udot}}\big]\cdot\sqrt{K_{\vir}}\,.&
\end{align}
That is, the virtual cycle and virtual structure sheaf of the $(-2)$-shifted cotangent bundle of $(M, F\udot)$ recover the original virtual cycle \eqref{BFvirt} and the usual Nekrasov-Okounkov-twisted virtual structure sheaf \eqref{vss3}.
\medskip

This abstract situation arises in nature by letting $M=M_Y$ be a projective moduli space of stable sheaves of fixed Chern character on a smooth projective 3-fold $Y$ with $\deg K_Y<0$. This condition ensures the standard obstruction theory
$$
F\udot\=\tau^{\le0}\(R\hom_{\pi\_Y}(\cE,\cE)^\vee[-1]\)\rt{\mathrm{At}_Y}\LL_{M_Y}
$$
of $M_Y$, based on $\Ext^*_Y(E,E)$ at $E\in M_Y$, is perfect because $\Ext^3_Y(E,E)=\Hom_Y(E,E\otimes K_Y)^*=0$. (It is also possible to handle the case that $Y$ is Calabi-Yau by using the trace-free obstruction theory $\Ext^*_Y(E,E)\_0$.) Here At$_Y$ is the Atiyah class of the (twisted) universal sheaf $\cE$ over $\pi\_Y\colon Y\times M_Y\to M_Y$ as in \cite[Equation 4.2]{HT}.

Now let $X=K_Y$ with zero section $i\colon Y\into X$. Let $M_X$ denote the moduli space of compactly supported sheaves on $X=K_Y$ with the same Chern character as $i_*\;E$ for $E\in M_Y$. Since $Li^*i_*=\id\oplus K_Y^{-1}[1]$, adjunction shows that
$$
\Ext^i_X(i_*\;E,i_*\;E)\= \Ext^i_Y(E,E)\,\oplus\,\Ext^{i-1}_Y(E,E\otimes K_Y).
$$
Done in a family over moduli space this shows that $i_*\colon M_Y\to M_X$ is an isomorphism of schemes,\footnote{If instead we consider stable pairs on $X$, these are not pushed forward from $Y$, so the analysis is different. The results are much the same for irreducible curve classes, however \cite[Proposition 3.3]{CMT2}.} and that the virtual cotangent bundle of $M_X$ is
$$
E\udot\ :=\ \tau^{[-2,0]}\(R\hom_{\pi\_X}(i_*\;\cE,i_*\;\cE)^\vee[-1]\)\=F\udot\oplus F_\bullet[2].
$$
Moreover the standard obstruction theory given by the Atiyah class map \cite[Equation 4.2]{HT},
$$
E\udot\=F\udot\oplus F_\bullet[2]\rt{\mathrm{At}_X}\LL_{M_X},
$$
is shown in \cite[Proposition 3.2]{DSY} to be precisely $(\mathrm{At}_Y,0)$. Thus we are in the situation of \eqref{sum} with $M_X$ being the $(-2)$-shifted cotangent bundle of $M_Y$. We conclude from (\ref{Chowa},\,\ref{vss4}) that our virtual cycle \eqref{virMdef} for $M_X$ is the same as the 3-fold virtual cycle \cite{Th:Casson} for $M_Y$, and our virtual structure sheaf \eqref{Ohatdef} is just the Nekrasov-Okounkov-twisted 3-fold virtual structure sheaf:
$$
[M_X]^{\vir}\=[M_Y]^{\vir}, \qquad \widehat\cO^{\;\vir}_{\!M_X}\=\widehat\cO^{\;\vir}_{\!M_Y}.
$$
So for local Calabi-Yau 4-folds $X=K_Y$ the 4-fold theory reduces to the now classical 3-fold theory of $Y$.

\appendix
\section{Gauge theoretic motivation}\label{gauge}
On a complex manifold $X$, consider holomorphic bundles which are topologically isomorphic to a fixed $C^\infty$ bundle $F$. From a gauge theory point of view we study them as $\overline\partial$ operators on $F$ satisfying the integrability condition 
\begin{equation}\label{NN} 
\overline\partial\;^2\ =\ 0\ \in\ \Omega^{0,2}(\mathrm{End}\,F), 
\end{equation} 
all modulo the action of the gauge group of $C^\infty$ automorphisms of the bundle. In dimensions greater than three this set-up has ``Fredholm index" $-\infty$: the problem is over-determined. We could try to throw away infinitely many of the equations \eqref{NN} by noting, by the Bianchi identity, that they take values in the kernel of $\overline\partial\colon\Omega^{0,2}(\mathrm{End}\,F)\to\Omega^{0,3}(\mathrm{End}\,F)$, 
but this need not be a vector bundle --- its rank can jump as we vary the $\overline\partial$-operator. 

Equivalently, from an algebro-geometric point of view, the deformation and obstruction spaces $H^i(\mathrm{End}\,F)=\mathrm{Ext}^i(F,F),\ i=1,2$, do not fit together to form a ``perfect obstruction theory" over the moduli space of holomorphic bundles since the higher obstruction groups $H^{\ge3}(\mathrm{End}\,F)=\mathrm{Ext}^{\ge3}(F,F)$ need not vanish. 

Let $(X,\Omega)$ now be a Calabi-Yau 4-fold with a holomorphic $(4,0)$-form $\Omega$. Then $\Omega^{0,2}(\mathrm{End}\,F)$ carries a gauge-invariant complex quadratic form
\beq{cxquad}
q(\ \cdot\ ,\ \cdot\ )\ =\ \int_X\tr(\ \cdot\ \wedge\ \cdot\ )\wedge\Omega.
\eeq
There is an obvious topological obstruction to the existence of a holomorphic structure on $F$,
\beq{p1obs}
\int_Xp_1(F)\wedge[\Omega]\=0
\eeq
since the left hand side is proportional to the integral of $\tr\overline\partial\;^2\wedge\Omega$. So we are only interested in $C^\infty$ bundles $F$ satisfying \eqref{p1obs}. By construction this condition is equivalent to $\overline\partial\;^2\in\Omega^{0,2}(\mathrm{End}\,F)$ being \emph{isotropic} with respect to $q$. So we are in the model situation described in \eqref{model}, with an ambient space (of $\overline\partial$-operators modulo gauge), a bundle with quadratic form over it, and an isotropic section
\beq{dbarsec}
s\=\overline\partial\;^2, \quad q(s,s)=0,
\eeq
cutting out the moduli space of integrable $\overline\partial$-operators or, equivalently, holomorphic bundles\footnote{One can also handle sheaves by spherically twisting them about $\cO_X(-N),\ N\gg0,$ as in \cite[Section 8]{JS} to make them into bundles.} on $X$.

Now give $X$ its Ricci-flat metric, and endow $F$ with a hermitian metric. Then
there is a complex anti-linear Hodge star operator 
\beq{+-split}
\widetilde*\ :=\ \overline{*(\ \cdot\ \wedge\Omega)}
\eeq
on $\Omega^{0,2}(\mathrm{End}\,F)$, where $*$ is the usual Hodge star and $\overline{\,\cdot\,}$ denotes complex conjugation on forms tensored with hermitian transpose in $\mathrm{End}\,F$. Since $\widetilde*^2=\id$, it splits $\Omega^{0,2}$ into real (not complex) $\pm1$ eigenspaces
$$
\Omega^{0,2}(\mathrm{End}\,F)\=\Omega^{0,+}(\mathrm{End}\,F)\,\oplus\,\Omega^{0,-}(\mathrm{End}\,F).
$$
These are maximal positive (respectively negative) definite real subspaces for the quadratic form $q$, which restricts to the $L^2$ hermitian metric on $\Omega^{0,+}$ and its negative on $\Omega^{0,-}$. Splitting the section \eqref{dbarsec} accordingly into $s=s^+\oplus s^-$, the isotropic condition \eqref{dbarsec} becomes the identity
$$
\|s^+\|^2-\|s^-\|^2\ =\ 0
$$
on their $L^2$-norms. Therefore $\|s\|^2=\|s^+\|^2+\|s^-\|^2$ equals $2\|s^+\|^2$ so
$$
s\,=\,0\ \iff\ s^+\,=\,0,
$$
just as in \eqref{s+-}. That is, we are led to consider ``half" of the equations \eqref{NN} by projecting them to $\Omega^{0,+}(\mathrm{End}\,F)$, 
\begin{equation}\label{NN+} 
\big(\;\overline\partial\;^2\;\big)^{\!+}\ =\ 0\ \in\ \Omega^{0,+}(\mathrm{End}\,F).
\end{equation} 
Then we have recovered the standard result (see \cite[page 28]{RTthesis} for instance) that solutions are the same as solutions of the original equations \eqref{NN},
\begin{equation}\label{2eqns}
\overline\partial\;^2\=0\ \ \Longleftrightarrow\ \ \big(\;\overline\partial\;^2\;\big)^{\!+}\=0
\eeq
when the topological condition \eqref{p1obs} holds.

Unlike the overdetermined equations \eqref{NN}, the equations \eqref{NN+} are elliptic modulo gauge, forming part of the $SU(4)$ instanton equations \cite{DT}.
Therefore they have the advantage of endowing $ M$ with a Kuranishi structure and so a virtual cycle, as shown by Borisov-Joyce. Since the equations \eqref{NN+} are real rather than holomorphic they have to use real derived geometry \cite{BJ}. 
(In reality they use finite dimensional local models \eqref{model} projected to a real subbundle $E_{\RR}$ of $E$ in place of the infinite dimensional gauge theory approach.)

\subsection*{Deformation complex} Linearising the the action of the gauge group and the Newlander-Nirenberg equations \eqref{NN} at a fixed integrable $\dbar$-operator gives the first line of the elliptic complex
\begin{multline*}
0\To\Omega^0(\End E)\rt\dbar\Omega^{0,1}(\End E)\rt\dbar\Omega^{0,2}(\End E) \\
\rt\dbar\Omega^{0,3}(\End E)\rt\dbar\Omega^{0,4}(\End E)\To0.
\end{multline*}
The existence of the second line is what makes the equations overdetermined. But the whole complex is self-dual, with sections of $\Omega^{0,i}$ being dual to sections of $\Omega^{0,4-i}$ by integrating against the holomorphic $(4,0)$-form $\Omega$.\footnote{This gives both the quadratic form \eqref{cxquad} and the Serre duality $\Ext^i(E,E)^*\cong\Ext^{4-i}(E,E)$ so crucial to this paper.} Via \eqref{+-split} we can split the complex into two halves, both of which are also elliptic. The first is
$$
0\To\Omega^0(\End E)\rt\dbar\Omega^{0,1}(\End E)\rt{\dbar^+}\Omega^{0,+}(\End E)\To0.
$$
This is the deformation complex of the equations \eqref{NN+}. Its exactness is what  makes Borisov-Joyce theory work. It is the complex analogue of Donaldson theory on a real 4-manifold, where the deformation complex of the asd equations is half of the elliptic complex governing flat connections.\medskip

In this paper we halve the equations in a different way, effectively by intersecting the graph of the isotropic section $\dbar^2$ of $\Omega^{0,2}(\End E)$ with a choice of maximal isotropic subbundle of $\Omega^{0,2}(\End E)$ instead of either of the maximal real definite subbundles $\Omega^{0,\pm}(\End E)$. Since this choice breaks the $SU(4)$ symmetry of the problem it makes it less likely to have gauge-theoretic motivation in general, although Nikita Nekrasov pointed out that one can think of working on the cover $\wt M$ of \eqref{cover} as averaging over all maximal isotropics --- which \emph{is} $SU(4)$ invariant --- and is reminiscent of the twistor approach to solving gauge theory equations (replacing the twistor space of complex structures by the orthogonal Grassmannian of maximal isotropics).

\bigskip \noindent
{\tt{j.oh@imperial.ac.uk\\ richard.thomas@imperial.ac.uk}}\medskip

\noindent Department of Mathematics\\
\noindent Imperial College\\
\noindent London SW7 2AZ \\
\noindent United Kingdom

\end{document}